\documentclass[article]{myclass}
\pdfoutput=1

\excludecomment{mycomment}

\newcommand{\NN}{\mathbb{N}}

\newcommand{\RR}{\mathbb{R}}
\newcommand{\CC}{\mathbb{C}}
\newcommand{\cc}{\mathbbm{c}}

\newcommand{\KK}{\mathbb{K}}

\DeclareMathOperator{\RE}{\textup{Re}}
\DeclareMathOperator{\IM}{\textup{Im}}
\DeclareMathOperator{\DOM}{\mathcal{D}}

\newcommand{\EPS}{\ensuremath{\varepsilon}}
\newcommand{\BAR}[1]{\ensuremath{\overline{#1}}}

\newcommand{\ALL}{\text{for all }}

\newcommand{\PAIR}[3][]{\ensuremath{\langle #2,#3 \rangle_{#1}}}

\newcommand{\WSTAR}{\ensuremath{\text{weak}^{\star}}\xspace}
\newcommand{\WSTARLY}{\ensuremath{\text{weakly}^{\star}}\xspace}
\newcommand{\WS}{\ensuremath{\text{w}^{\star}}\xspace}

\newcommand{\BC}{\ensuremath{\textup{BC}}}

\newcommand{\LP}[1]{\ensuremath{L^{#1}}}

\DeclareMathOperator{\LIP}{Lip}

\newcommand{\BND}{\ensuremath{\mathcal{L}}}

\newcommand{\CM}{\mathcal{W}^{\textup{c}}}
\newcommand{\LCM}{\mathcal{W}^{\textup{c}}_{\textup{loc}}}

\newcommand{\STAR}[1]{\ensuremath{#1^{\star}}}
\newcommand{\SUN}[1]{\ensuremath{#1^{\odot}}}
\newcommand{\SUNSTAR}[1]{\ensuremath{#1^{\odot\star}}}

\DeclareMathOperator{\FAV}{Fav}

\newcommand{\DEF}{\ensuremath{\coloneqq}}
\newcommand{\FED}{\ensuremath{\eqqcolon}}

\newcommand{\ONE}{\ensuremath{\mathbbm{1}}}

\renewcommand{\phi}{\varphi}

\newcommand\mytitle{A class of abstract delay differential equations\xspace}
\newcommand\mysubtitle{in the light of suns and stars. II\xspace}
\newcommand\myauthor{Sebastiaan G. Janssens\xspace}

\title{\mytitle\\\mysubtitle}

\author{\myauthor\thanks{Department of Mathematics, Utrecht University, Budapestlaan 6, 3508 TA Utrecht, The Netherlands (\email{s.g.janssens@uu.nl} or \email{sj@dydx.nl})}}

\hypersetup{
  pdftitle={\mytitle \mysubtitle},
  pdfauthor={\myauthor}
}

\date{}

\begin{document}

\maketitle

\begin{abstract}
  Recent work \cite{ADDE1} by the author about a class of abstract delay differential equations (DDEs), as well as earlier work \cite{Diekmann2008,DiekmannGyllenberg2012} by Diekmann and Gyllenberg on other classes of delay equations, motivates the introduction of the general notion of an admissible range and an admissible perturbation for a given $\mathcal{C}_0$-semigroup $T_0$ on a Banach space $X$ that is not assumed to be sun-reflexive with respect to $T_0$. We investigate the relationship between admissible ranges for $T_0$ and the subspace $X^{\odot\times}$ of $X^{\odot\star}$ introduced by Van Neerven in \cite{VanNeerven1992}. We answer two questions about robustness of admissibility with respect to bounded linear perturbations and we use these answers to study the semilinear problem and its linearization. Partly as an application of the material developed up to that point, and partly as a justification of existing work on local bifurcations in models taking the form of abstract DDEs, we compare the construction of center manifolds in the non-sun-reflexive case with known results by Diekmann and Van Gils \cite{Delay1995} for the sun-reflexive case. We show that a systematic use of the space $X^{\odot\times}$ facilitates a generalization of the existing results with relatively little effort. In this context we also give sufficient conditions for the existence of appropriate spectral decompositions of $X$ and $X^{\odot\times}$ without assuming that the linearized semiflow is eventually compact. A center manifold theorem for the motivating class of abstract DDEs then follows as a particular case.
\end{abstract}

\begin{keywords}
  delay equation, abstract integral equation, dual perturbation theory, adjoint semigroup, sun-star calculus, non-sun-reflexivity, linearization, center manifold, non-compactness.
\end{keywords}

\begin{MSC}
Primary 34K30; Secondary 47D06.
\end{MSC}

\section{Introduction}\label{sec:introduction}
This article is a continuation of \cite{ADDE1}, but the approach here is a bit more general and, as a consequence, all results except those in \cref{sec:ddes} have a wider applicability; also see \cref{sec:structure} for an overview of the structure and contents.

We are still motivated by abstract delay differential equations (DDEs) of the form
\begin{subequations}
  \label{eq:adde-ic}
  \begin{equation}
    \label{eq:adde}
    \dot{x}(t) = B x(t) + F(x_t), \qquad t \ge 0,
  \end{equation}
  where the unknown $x$ takes values in a real or complex Banach space $Y$. (The adjective \emph{abstract} comes from the fact that $Y$ is allowed to be infinite-dimensional.) It is assumed that the possibly unbounded operator $B : \DOM(B) \subseteq Y \to Y$ generates a $\mathcal{C}_0$-semigroup $S$ of bounded linear operators on $Y$. As the state space for \cref{eq:adde} we choose the Banach space $X \DEF C([-h,0],Y)$ of continuous $Y$-valued functions on $[-h,0]$, endowed with the supremum-norm. The history of $x$ at time $t$ is denoted by $x_t \in X$, so
  \[
    x_t(\theta) \DEF x(t + \theta), \qquad \theta \in [-h,0].
  \]
  In particular, an initial condition for \cref{eq:adde} is specified as
  \begin{equation}
    \label{eq:ic}
    x_0 = \phi \in X.
  \end{equation}
\end{subequations}
Finally, we assume that $F : X \to Y$ is continuous and possibly nonlinear.

In \cite{ADDE1} the shift semigroup $\{T_0(t)\}_{t \ge 0}$ was defined as the $\mathcal{C}_0$-semigroup on $X$ corresponding to the solution of \cref{eq:adde-ic} with $F = 0$. The sun dual $\SUN{X}$ of $X$ with respect to $\{T_0(t)\}_{t \ge 0}$ was characterized as
\begin{equation}
  \label{eq:xsun-adde}
  X^{\odot} \simeq Y^{\odot} \times \LP{1}([0,h], Y^{\star}),
\end{equation}
where $\simeq$ denotes an (explicit and simple) isometric isomorphism. It was also shown in \cite{ADDE1} that if 
\begin{equation}
  \label{eq:ell}
  \ell : Y \to \SUNSTAR{X}, \qquad \ell y \DEF (j_Y y, 0),  
\end{equation}
is the embedding induced by the canonical embedding $j_Y : Y \to \SUNSTAR{Y}$, then the \WSTAR convolution integral in the right-hand side of the abstract integral equation
\begin{equation}
  \label{eq:aie0-adde}
  u(t) = T_0(t)\phi + j^{-1}\int_0^t\SUNSTAR{T_0}(t - \tau)(\ell \circ F)(u(\tau))\,d\tau, \qquad t \ge 0,
\end{equation}
takes values in the range of the canonical embedding $j : X \to \SUNSTAR{X}$. Furthermore, it was shown that there is a one-to-one correspondence between the global mild solutions of the initial value problem \cref{eq:adde-ic} and the global solutions of \cref{eq:aie0-adde}. The same correspondence exists between \emph{local} solutions of \cref{eq:adde-ic} and \cref{eq:aie0-adde}.

\subsection{Structure and outline}\label{sec:structure}
In this second part we adopt a more general point of view than in \cite{ADDE1}. Let $\{T_0(t)\}_{t \ge 0}$ be an arbitrary $\mathcal{C}_0$-semigroup of bounded linear operators on an arbitrary real or complex Banach space $X$. It is not assumed that $X$ is $\odot$-reflexive with respect to $\{T_0(t)\}_{t \ge 0}$. There exist constants $\omega \in \RR$ and $M \ge 1$ such that
\begin{equation}
  \label{eq:expboundT0}
  \|T_0(t)\| \le M e^{\omega t} \qquad \ALL t \ge 0.
\end{equation}
Given a continuous (possibly nonlinear) operator $G : X \to X^{\odot\star}$, we are interested in solutions of the abstract integral equation
\begin{equation}
  \label{eq:aie0}
  u(t) = T_0(t)x + j^{-1}\int_0^t\SUNSTAR{T_0}(t - \tau)G(u(\tau))\,d\tau, \qquad t \ge 0,
\end{equation}
where $x \in X$ is an initial condition. (In the particular case of abstract DDEs we have $G \DEF \ell \circ F$.) If $u$ is continuous on some time interval $[0,t_e)$, then the \WSTAR Riemann integral appearing in the right-hand side of \cref{eq:aie0} takes values in $X^{\odot\odot}$ for all $0 \le t < t_e$, but the inclusion $jX \subseteq X^{\odot\odot}$ may be strict. As a consequence, \cref{eq:aie0} generally does not give rise to a nonlinear semiflow on $X$, which is a fundamental complication from a dynamical systems perspective. Here we aim to address this complication in a systematic manner, motivated by previous work on abstract renewal equations \cite{Diekmann2008}, classical coupled systems with infinite delay \cite{DiekmannGyllenberg2012}, and abstract DDEs \cite{ADDE1}. 

Let $J$ be any non-degenerate interval and denote $\Omega_J \DEF \{(t,s) \in J \times J\,:\, t \ge s\}$. A continuous function $f : J \to \SUNSTAR{X}$ will be called a \textbf{forcing function}. Given a forcing function $f$, introduce
\begin{equation}
  \label{eq:v0}
  v_0(\cdot,\cdot,f) : \Omega_J \to \SUNSTAR{X}, \qquad v_0(t,s,f) \DEF \int_s^t\SUNSTAR{T_0}(t - \tau)f(\tau)\,d\tau.
\end{equation}
Motivated by the above considerations, we are interested in forcing functions $f$ with the property that
\begin{equation}
  \label{eq:admissible}
  v_0(t,s,f) \in jX \qquad \ALL (t,s) \in \Omega_J.
\end{equation}

\begin{definition}\label{def:admrange}
  A closed subspace $\mathscr{X}_0$ of $X^{\odot\star}$ is called an \textbf{admissible range for} $\{T_0(t)\}_{t \ge 0}$ if \cref{eq:admissible} holds \emph{for every forcing function} $f : J \to \mathscr{X}_0$. A continuous linear or nonlinear operator\footnote{If the continuous linear operator $L$ is admissible for $\{T_0(t)\}_{t \ge 0}$, then $L$ satisfies condition \textbf{(H0)} in \cite[Section 6]{ADDE1}, so the results from there apply.} $G : X \to X^{\odot\star}$ is called an \textbf{admissible perturbation for} $\{T_0(t)\}_{t \ge 0}$ if there exists an admissible range $\mathscr{X}_0$ for $\{T_0(t)\}_{t \ge 0}$ such that $G$ takes its values in $\mathscr{X}_0$.
\end{definition}

The strategy is now to proceed as closely as possible to the $\odot$-reflexive case, but allowing only \emph{admissible} perturbations for $\{T_0(t)\}_{t \ge 0}$. In \cref{sec:admissibility} we prove some elementary and less elementary properties related to admissibility for a given $\mathcal{C}_0$-semigroup. For example, in \cref{prop:admindependent} we show that $\mathscr{X}_0$ is independent of $J$, so in \cref{def:admrange} there is no need to include $J$ in the terminology or notation. In \cref{thm:admconstants} we give a simple criterion for a closed subspace of $X^{\odot\star}$ to be an admissible range. After a small digression on norm convergence of \WS-integrals over unbounded intervals, we discuss the relationship between admissible ranges and a certain subspace of $X^{\odot\star}$ that was introduced in \cite{VanNeerven1992}. This discussion leads to \cref{thm:xsuncross} and its corollary.

Next, in \cref{sec:inhomogeneous} we address two interrelated questions about perturbation of $\{T_0(t)\}_{t \ge 0}$ by an admissible bounded linear operator. The first of these questions concerns robustness of a given admissible range, while the second question will prove to be of particular relevance for the local analysis of the semiflow generated by \cref{eq:aie0}. After some preparations, the answers are presented in \cref{thm:inhom:answers} and its corollary.

In \cref{sec:linearization} we move from linear to semilinear theory. In general, a perturbative analysis near an equilibrium of the semiflow generated by \cref{eq:aie0} requires a splitting of $G$ into a linear and a nonlinear part and a subsequent comparison between the linearly perturbed semigroup $\{T(t)\}_{t \ge 0}$ and the nonlinear semiflow $\Sigma$. In turn, such a comparison depends on the equivalence in \cref{prop:aie0_aie} and on the (uniform) differentiability of $\Sigma$ with respect to the state in \cref{prop:linearization}.

In \cref{sec:cm} we discuss the construction of smooth center manifolds in the non-$\odot$-reflexive case, offering both a comparison with \cite[Chapter IX]{Delay1995}, itself based on \cite{VanGils1982,DiekmannVanGils1984,VanderbauwhedeVanGils1987,Ball1973}, as well as an application of the material developed up until that point. We demonstrate that, with this material at hand, the non-trivial consequences of the lack of $\odot$-reflexivity are both relatively minor as well as localized exclusively in the proof of \cref{prop:Keta} on a pseudo-inverse for the linear inhomogeneous equation. The results of the construction are then summarized in \cref{thm:cm-summary}.

In the accompanying \cref{sec:decomposition-appendix} we first discuss the `lifting' of hypothesis \cref{hyp:cm} in \cref{sec:cm} from $X$ to a subspace of $X^{\odot\star}$ that includes the range of $G$. The result of this discussion, in the form of \cref{prop:xsuncross-cm}, is used in the main text. Second, we provide a proof of \cref{thm:hypcm} which gives sufficient conditions, in terms of eventual norm continuity of $\{T(t)\}_{t \ge 0}$ and a decomposition of the spectrum of its generator, for both central hypotheses \cref{hyp:cm} and \cref{hyp:cm:j} in \cref{sec:cm} to hold. Hence this theorem demonstrates that the usual assumption of eventual compactness of $\{T(t)\}_{t \ge 0}$ is more restrictive than necessary.

In \cref{sec:ddes} we return to \cref{eq:adde-ic} and we discuss some implications of the general results in the foregoing sections for the class of abstract DDEs. Indeed, the original motivation for the present work can be found in \cite{VanGils2013}, which was continued in \cite{Dijkstra2015} and extended in \cite{SpekVanGilsKuznetsov2019} to include diffusion via the unbounded operator $B$. In particular, \cref{thm:cm-dde} provides a justification of the center manifold reduction that underlies the local normal form calculations performed in \cite{VanGils2013,Dijkstra2015,SpekVanGilsKuznetsov2019}. It is \emph{not} assumed that $\{T(t)\}_{t \ge 0}$ is eventually compact and therefore the theorem applies equally well to the cases $B = 0$ and $B \neq 0$. This eliminates the somewhat unsatisfactory dichotomy in the technical treatment of the two cases that results from the usual assumption of eventual compactness \cite{Travis1974,Wu1996,Faria2002,Faria2006}.

\subsection{Conventions and notation}\label{sec:conventions}
\begin{enumerate}[wide=0pt,leftmargin=\parindent,label=\roman*.]
\item
  We use the notations $\RR_+ \DEF [0,\infty)$ and $\RR_- \DEF (-\infty,0]$. Unless explicitly stated otherwise, \emph{all intervals are assumed to be non-degenerate}. They need not be open, closed or bounded.
\item
  Unless explicitly stated otherwise, the scalar field for a given vector space may be either real or complex and is denoted by $\KK$.
\item
  The duality pairing between a Banach space $E$ and its continuous dual space $\STAR{E}$ is written as
\[
  \PAIR{x}{\STAR{x}} \DEF \STAR{x}(x), \qquad \ALL\,x \in E \text{ and } \STAR{x} \in \STAR{E},
\]
and we commonly use the prefix \WS to indicate the \WSTAR-topology on $\STAR{E}$. 
\item
  If $E_1$ and $E_1$ are Banach spaces over the same field, then $\BND(E_1,E_2)$ is the Banach space of all bounded linear operators from $E_1$ to $E_2$, equipped with the operator norm.
\item
  If $J$ is an interval and $E$ is a Banach space, then $C(J,E)$ is the vector space of continuous functions from $J$ into $E$, equipped with the topology of uniform convergence on compact subsets. (Of course, if $J$ is compact, then this topology is induced by the usual supremum-norm.)
\item
  For any interval $J$ we denote by $\ONE$  the function on $J$ with the constant value $1 \in \RR$, and we denote by $\ONE \otimes x \in C(J,E)$ the function with the constant value $x \in E$.
\end{enumerate}

We briefly recall the definition of the general \WS-integral as it appears in \cite[Interlude 3.13 in Appendix II]{Delay1995}; also see the treatment in \cite[Appendix A2]{VanNeerven1992} for an equivalent definition that is on the one hand more general (involving an arbitrary measure space), but on the other hand too restrictive (the measure is assumed to be finite).

Given a Banach space $E$ and an interval $J$, not necessarily bounded, let $q : J \to E^{\star}$ be \WS-Lebesgue integrable, i.e. $\tau \mapsto \PAIR{x}{q(\tau)}$ is in $L^1(J,\KK)$ for all $x \in E$. By virtue of the closed graph theorem, the map
\[
  x \mapsto \int_J \PAIR{x}{q(\tau)}\,d\tau, \qquad x \in E,
\]
defines an element $Q^{\star}$ of $E^{\star}$. We call $Q^{\star}$ the \textbf{\WS-integral of $q$ over $J$} and put
\begin{equation}
  \label{eq:wsintegral_general}
  \int_J q(\tau)\,d\tau \DEF Q^{\star}.
\end{equation}
If $J$ is compact and $q$ is \WS-continuous, then the above \WS-integral may be evaluated as a \WS-Riemann integral. So far, examples of \WS-Riemann integrals occurred in \cref{eq:aie0-adde,eq:aie0,eq:v0}. On the other hand, if $q \in L^1(J,E^{\star})$, then $q$ is \WS-Lebesgue integrable and the Bochner- and \WS-integrals coincide. (This is a direct consequence of the fact that Bochner integrals commute with bounded linear operators.)

\section{Admissible ranges and admissible forcing functions}\label{sec:admissibility}
Let $\{T_0(t)\}_{t \ge 0}$ be a $\mathcal{C}_0$-semigroup on a Banach space $X$, as in \cref{sec:structure}. The following notion is useful in conjunction with \cref{def:admrange}.

\begin{definition}\label{def:admfunc}
  A forcing function $f : J \to \SUNSTAR{X}$ is called an \textbf{admissible forcing function for $\{T_0(t)\}_{t \ge 0}$ on $J$} if \cref{eq:admissible} holds \emph{for this particular choice of $f$}. The vector space $\mathscr{F}_0(J)$ of all such functions is called the \textbf{admissible forcing class for $\{T_0(t)\}_{t \ge 0}$ on $J$}. 
\end{definition}

\subsection{Elementary properties}\label{sec:admprops}
First we record a trivial but useful relationship between admissible ranges and admissible forcing classes.

\begin{proposition}\label{prop:rangeclass}
Let $\mathscr{X}_0$ be a closed subspace of $\SUNSTAR{X}$. If $\mathscr{X}_0$ is an admissible range for $\{T_0(t)\}_{t \ge 0}$ and $J$ is an interval, then
\begin{equation}
  \label{eq:prop:rangeclass}
  C(J,\mathscr{X}_0) \subseteq \mathscr{F}_0(J).
\end{equation}
Conversely, if \cref{eq:prop:rangeclass} holds for some interval $J$, then $\mathscr{X}_0$ is an admissible range for $\{T_0(t)\}_{t \ge 0}$.
\end{proposition}

Next, we show that the admissible range is independent of the particular interval, as announced following \cref{def:admrange}. This justifies calling $\mathscr{X}_0$ an \textbf{admissible range for $\{T_0(t)\}_{t \ge 0}$} if $\mathscr{X}_0$ is an admissible range for $\{T_0(t)\}_{t \ge 0}$ on \emph{some} interval.

\begin{proposition}\label{prop:admindependent}
  If $\mathscr{X}_0$ is an admissible range for $\{T_0(t)\}_{t \ge 0}$ on \emph{some} interval, then it is admissible for $\{T_0(t)\}_{t \ge 0}$ on \emph{every} interval.
\end{proposition}
\begin{proof}
 Suppose that $\mathscr{X}_0$ is admissible for $\{T_0(t)\}_{t \ge 0}$ on the interval $J$. Let $J'$ be any arbitrary interval, let $f : J' \to \mathscr{X}_0$ be a continuous function and let $(t,s) \in \Omega_{J'}$ with $t > s$ strictly. The interval $J$ is non-degenerate by the convention from \cref{sec:conventions}, so there exists $n \in \NN$ such that $J$ contains an interval $[s_0,t_0]$ with $t_0 - s_0 = \frac{t - s}{n}$. Define
  \[
    \EPS \DEF \frac{t - s}{n}, \qquad \tau_i \DEF s + i\EPS, \qquad i = 0,\ldots,n.
  \]
  Then, noting that $t - \tau_i \ge 0$, 
  \begin{align*}
    \int_s^t\SUNSTAR{T_0}(t - \tau)f(\tau)\,d\tau &= \sum_{i=1}^n \int_{\tau_{i-1}}^{\tau_i}\SUNSTAR{T_0}(t - \tau)f(\tau)\,d\tau\\
                                                            &= \sum_{i=1}^n \SUNSTAR{T_0}(t - \tau_i)\int_{\tau_{i-1}}^{\tau_i}\SUNSTAR{T_0}(\tau_i - \tau)f(\tau)\,d\tau,
  \end{align*}
  and we need to show that this is in $jX$. Since $jX$ is a positively $\SUNSTAR{T_0}$-invariant subspace, it is sufficient to show this for each of the \WS-Riemann integrals inside the sum. So, for fixed $1 \le i \le n$ we introduce the new integration variable $\sigma = a \tau + b$ with $a$ and $b$ determined by the conditions
  \[
    \sigma(\tau_{i-1}) = s_0, \qquad \sigma(\tau_i) = t_0,
  \]
  so $a = 1$ and $b$ is some irrelevant expression. This yields
  \[
    \int_{\tau_{i-1}}^{\tau_i}\SUNSTAR{T_0}(\tau_i - \tau)f(\tau)\,d\tau = \int_{s_0}^{t_0}\SUNSTAR{T_0}(t_0 - \sigma)f(\sigma - b)\,d\sigma.
  \]
  The function $[s_0, t_0] \ni \sigma \mapsto f(\sigma - b) \in \mathscr{X}_0$ can be trivially extended to an element of $C(J,\mathscr{X}_0)$. We conclude that the right hand side of the above equality is indeed in $jX$.
\end{proof}

The next result is not surprising, but we record it explicitly for later use.

\begin{proposition}
  \label{prop:jXadmissible}
$jX$ is an admissible range for $\{T_0(t)\}_{t \ge 0}$.  
\end{proposition}
\begin{proof}
  For arbitrary $f \in C(\RR,jX)$ and $(t,s) \in \Omega_{\RR}$ we have
  \[
    \int_s^t T_0^{\odot\star}(t - \tau)f(\tau)\,d\tau = \int_s^t j T_0(t - \tau)(j^{-1} \circ f)(\tau)\,d\tau = j \int_s^t T_0(t - \tau) (j^{-1} \circ f)(\tau)\,d\tau \in jX,
  \]
  where the first integral is a \WS-Riemann integral and the others are ordinary Riemann integrals.
\end{proof}

\begin{proposition}\label{prop:admclosed}
  $\mathscr{F}_0(J)$ is a closed subspace of $C(J,\SUNSTAR{X})$.
\end{proposition}
\begin{proof}
  Let $(f_{\alpha})$ be a net in $\mathscr{F}_0(J)$ converging to some $f \in C(J,\SUNSTAR{X})$ and let $(t, s) \in \Omega_J$ be arbitrary. The interval $[s, t]$ is compact in $J$, so $f_{\alpha} \to f$ uniformly on $[s, t]$. By a standard estimate we therefore have
  \begin{equation}
    \label{eq:net_of_integrals}
    \int_s^t \SUNSTAR{T_0}(t - \tau)f_{\alpha}(\tau)\,d\tau \to \int_s^t \SUNSTAR{T_0}(t - \tau)f(\tau)\,d\tau
  \end{equation}
  in the norm of $\SUNSTAR{X}$. The integrals on the left-hand side are elements of the closed subspace $jX$ of $X^{\odot\star}$, so the same is true for the integral on the right-hand side. Since $(t,s) \in \Omega_J$ was arbitrary, this proves that $f \in \mathscr{F}_0(J)$.
\end{proof}

The following result shows that, given a closed subspace $\mathscr{X}_0$ of $X^{\odot\star}$, in order to establish its admissibility, it is sufficient to verify admissibility of all constant $\mathscr{X}_0$-valued forcing functions defined on some interval.

\begin{theorem}\label{thm:admconstants}
  If $\mathscr{X}_0$ is a closed subspace of $\SUNSTAR{X}$ and there exists an interval $J$ such that $\ONE \otimes x^{\odot\star} \in \mathscr{F}_0(J)$ for all $x^{\odot\star} \in \mathscr{X}_0$, then $\mathscr{X}_0$ is an admissible range for $\{T_0(t)\}_{t \ge 0}$. 
\end{theorem}
\begin{proof}
  Since $J$ is non-degenerate, we may assume that $J$ is compact. We show that $\mathscr{F}_0(J)$ contains every $\mathscr{X}_0$-valued affine function. The result is then a simple consequence of the fact that continuous functions on compact intervals admit uniform approximations by linear splines, as will be detailed.
  \begin{steps}
  \item
    We prove that for every $x^{\odot\star} \in \mathscr{X}_0$ the linear function $\tau \mapsto \tau x^{\odot\star}$ is in $\mathscr{F}_0(J)$. Let $(t,s) \in \Omega_J$ with $t > s $ and $x^{\odot} \in X^{\odot}$ be arbitrary. Define
    \[
      w : [s,t] \to X^{\odot\star}, \qquad w(\tau) \DEF \int_s^{\tau} T_0^{\odot\star}(t - \sigma) x^{\odot\star} \,d\sigma,
    \]
    and note that
    \[
      \frac{d}{d\tau}\PAIR{x^{\odot}}{w(\tau)} = \PAIR{x^{\odot}}{T_0^{\odot\star}(t - \tau)x^{\odot\star}}.
    \]
    Then partial integration yields
    \begin{align*}
      \PAIR{x^{\odot}}{\int_s^t T_0^{\odot\star}(t - \tau) \tau x^{\odot\star} \,d\tau} &= \int_s^t \tau \PAIR{x^{\odot}}{T_0^{\odot\star}(t - \tau) x^{\odot\star}} \,d\tau\\
                                                                                        &= t \PAIR{x^{\odot}}{w(t)} - \int_s^t \PAIR{x^{\odot}}{w(\tau)} \,d\tau.
    \end{align*}
    Since adjoints of bounded linear operators commute with \WS-integration and $jX$ is positively $T_0^{\odot\star}$-invariant, it follows that
    \[
      w(\tau) = T_0^{\odot\star}(t - \tau) \int_s^{\tau} T_0^{\odot\star}(\tau - \sigma) x^{\odot\star} \,d\sigma \in jX \qquad \ALL\,\tau \in [s,t],
    \]
    and $w$ is norm-continuous. Hence
    \[
      \int_s^t T_0^{\odot\star}(t - \tau) \tau x^{\odot\star} \,d\tau = t w(t) - \int_s^t w(\tau) \,d\tau \in jX,
    \]
    where for the inclusion it was also used that the integral in the right-hand side is an ordinary Riemann integral and $jX$ is norm-closed. We conclude that every linear (hence: every affine) function with values in $\mathscr{X}_0$ is in $\mathscr{F}_0(J)$.
  \item
    Let $f : J \to \mathscr{X}_0$ be continuous, hence uniformly continuous, so $f$ is the uniform limit of a sequence $(f_n)$ of continuous piecewise affine functions. We check that $f_n \in \mathscr{F}_0(J)$ for all $n \in \NN$. Let $n \in \NN$ and $(t,s) \in \Omega_J$ with $t > s$ be arbitrary. There exist $m \in \NN$ and a partition $s = t_0 < t_1 < \ldots < t_m = t$ of $[s,t]$ such for every $i = 1,\ldots,m$ the restriction of $f_n$ to $[t_{i - 1},t_i]$ is affine. We have
    \[
      \int_s^t T_0^{\odot\star}(t - \tau) f_n(\tau) \,d\tau = \sum_{i=1}^m T_0^{\odot\star}(t - t_i) \int_{t_{i-1}}^{t_i} T_0^{\odot\star}(t_i - \tau) f_n(\tau) \,d\tau
    \]
    and each summand in the right-hand side is in $jX$, so the left-hand side is in $jX$ as well. We conclude that $f_n \in \mathscr{F}_0(J)$.
  \item
    \Cref{prop:admclosed} and the uniform convergence $f_n \to f$ as $n \to \infty$ imply that $f \in \mathscr{F}_0(J)$ as well. The second part of \cref{prop:rangeclass} then implies that $\mathscr{X}_0$ is an admissible range for $\{T_0(t)\}_{t \ge 0}$. \hfill \qedhere
  \end{steps}
\end{proof}

The following trivial corollary will be used in the proof of \cref{thm:inhom:answers} in \cref{sec:inhomogeneous}.

\begin{corollary}\label{cor:admlipschitz}
  If $\mathscr{X}_0$ is a closed subspace of $\SUNSTAR{X}$ and there exists an interval $J$ such that every Lipschitz function $f : J \to \mathscr{X}_0$ is in $\mathscr{F}_0(J)$, then $\mathscr{X}_0$ is an admissible range for $\{T_0(t)\}_{t \ge 0}$.
\end{corollary}

\subsection{\texorpdfstring{$\text{Weak}^{\star}$}{Weak*}-integration over unbounded intervals}
We discuss the situation where in the \WS-integral in \cref{eq:v0} defining $v_0$ one (or both) of the integration limits is infinite. This situation will occur in \cref{sec:xsuncross,sec:cm}.

The importance of the norm-closedness of $jX$ for questions of admissibility has already become apparent in the proofs of \cref{prop:admclosed,thm:admconstants}, for example. Following is a simple criterion that ensures that a given \WS-integral over an unbounded interval equals the limit in norm of a sequence of \WS-integrals over compact intervals.

\begin{lemma}
  \label{lem:wsnormconv}
  Let $J$ be an interval and let $(J_n)$ be an increasing sequence of intervals such that $J = \bigcup_n J_n$. If $q : J \to E^{\star}$ is \WS-Lebesgue measurable and there exists $\hat{q} \in L^1(J,\RR)$ such that $\|q(\tau)\| \le \hat{q}(\tau)$ for a.e. $\tau \in J$, then
  \[
    \lim_{n \to \infty} \int_{J_n} q(\tau)\,d\tau = \int_J q(\tau)\,d\tau 
  \]
  in norm.
\end{lemma}
\begin{proof}
  The assumptions imply that $q$ is \WS-Lebesgue integrable, so the \WS-integral of $q$ over $J$ exists. Let $\chi_n$ be the characteristic function of the interval $J_n$. For arbitrary $x \in E$ with $\|x\| \le 1$,
  \[
    \Bigl|\PAIR{x}{\int_J q(\tau)\,d\tau - \int_{J_n} q(\tau)\,d\tau}\Bigr| = \Bigl| \int_J \PAIR{x}{(1 - \chi_n(\tau))q(\tau)}\,d\tau \Bigr| \le \int_J (1 - \chi_n(\tau))\hat{q}(\tau)\,d\tau,
  \]
  so
  \begin{equation}
    \label{eq:lem:wsnormconv}
    \Bigl\| \int_J q(\tau)\,d\tau - \int_{J_n} q(\tau)\,d\tau \Bigr\| \le \int_J (1 - \chi_n(\tau))\hat{q}(\tau)\,d\tau \qquad \ALL\, n \in \NN.
  \end{equation}
  The measurable functions $\tau \mapsto (1 - \chi_n(\tau))\hat{q}(\tau)$ converge to zero, pointwise on $J$ as $n \to \infty$, and
  \[
    |(1 - \chi_n(\tau))\hat{q}(\tau)| \le \hat{q}(\tau) \qquad \ALL\,n \in \NN \text{ and } \tau \in J.
  \]
  Hence the right-hand side of \cref{eq:lem:wsnormconv} tends to zero as $n \to \infty$ by the dominated convergence theorem.
\end{proof}

As an application, for use in \cref{sec:xsuncross}, we have the following result. We recall from \cref{sec:conventions} that $\KK$ denotes either the real or complex scalar field.

\begin{proposition}
  \label{prop:resolvnormconv}
  Let $U$ be a $\mathcal{C}_0$-semigroup on $E$ with generator $C$ and let $M_U \ge 1$ and $\omega_U \in \RR$ be such that $\|U(t)\| \le M_U e^{\omega_U t}$ for all $t \ge 0$. For any $\lambda \in \KK$ with $\RE{\lambda} > \omega_U$, the resolvent of $C^{\star}$ at $\lambda$ is given by
  \[
    R(\lambda,C^{\star})x^{\star} = \lim_{t \to \infty} \int_0^t e^{-\lambda \tau} U^{\star}(\tau) x^{\star}\,d\tau, \qquad x^{\star} \in E^{\star}, 
  \]
  i.e. $R(\lambda,C^{\star})x^{\star}$ is the limit in norm of a net of \WS-Riemann integrals over compact intervals.
\end{proposition}
\begin{proof}
  Choose $\lambda \in \KK$ with $\RE{\lambda} > \omega_U$. Then we have the Laplace transform representation
  \[
    R(\lambda,C)x = \int_0^{\infty} e^{-\lambda \tau} U(\tau)x\,d\tau \qquad \ALL\,x \in E,
  \]
  where the integral is an improper Riemann integral. Let $x^{\star} \in E^{\star}$ be arbitrary and let $(t_n)$ be an arbitrary nonnegative, strictly increasing sequence such that $\lim_{n \to \infty} t_n = \infty$. Then
  \begin{equation}
    \label{eq:prop:resolvnormconv}
    \PAIR{\int_0^{t_n} e^{-\lambda \tau} U(\tau) x\,d\tau}{x^{\star}} = \PAIR{x}{\int_0^{t_n} e^{-\lambda \tau} U^{\star}(\tau) x^{\star}\,d\tau} \qquad \ALL\,x \in E \text{ and } n \in \NN.
  \end{equation}
  The integrand inside the integral on the right-hand side is \WS-Lebesgue measurable on $\RR_+$ and
  \[
    \|e^{-\lambda \tau} U^{\star}(\tau) x^{\star}\| \le M_{U} e^{-(\RE{\lambda} - \omega_U)\tau} \|x^{\star}\| \qquad \ALL\,\tau \ge 0,
  \]
  so the integrand is \WS-Lebesgue integrable and its norm is dominated by an element of $L^1(\RR_+,\RR)$. \Cref{lem:wsnormconv} then shows that
  \[
    \lim_{n \to \infty} \int_0^{t_n} e^{-\lambda \tau} U^{\star}(\tau) x^{\star}\,d\tau = \int_0^{\infty} e^{-\lambda \tau} U^{\star}(\tau) x^{\star}\,d\tau
  \]
  in norm. Take the limit $n \to \infty$ in \cref{eq:prop:resolvnormconv} to obtain
  \[
    \PAIR{R(\lambda,C)x}{x^{\star}} = \PAIR{x}{\int_0^{\infty} e^{-\lambda \tau} U^{\star}(\tau) x^{\star}\,d\tau} \qquad \ALL\,x \in E,
  \]
  so $R(\lambda,C^{\star}) = R(\lambda,C)^{\star} = \int_0^{\infty} e^{-\lambda \tau} U^{\star}(\tau) x^{\star}\,d\tau$.
\end{proof}

\subsection{Relationship with the subspace \texorpdfstring{$X_0^{\odot\times}$}{X0suncross}}\label{sec:xsuncross}
The notion of an admissible range for a given $\mathcal{C}_0$-semigroup $\{T_0(t)\}_{t \ge 0}$, introduced in \cref{def:admrange}, is useful from the viewpoint of particular classes of delay equations \cite{Diekmann2008,DiekmannGyllenberg2012,ADDE1}, but also a bit unsatisfactory from a more fundamental perspective: The specification of $\mathscr{X}_0$ requires a class-dependent choice and, moreover, it is not clear whether a chosen admissible range has an extension that is in some sense maximal. In this light, it is relevant to refer to \cite[p.56]{VanNeerven1992}, where the author introduces the subspace\footnote{We point out that what we denote by $X_0^{\odot\times}$ is denoted by $X^{\odot\times}$ (so, \emph{without} the subscript) in \cite[p.56]{VanNeerven1992}. The notation $X^{\odot\times}_0$ (so, \emph{with} the subscript) is also introduced in \cite[p.56]{VanNeerven1992}, but it has another meaning there, related to the more general case that $A_0$ is a (not necessarily densely defined) Hille-Yosida operator.} 
\begin{equation}
  \label{eq:x0suncross}
  X_0^{\odot\times} \DEF \{x^{\odot\star} \in X^{\odot\star}\,:\,R(\lambda,A_0^{\odot\star})x^{\odot\star} \in jX\}, \qquad \lambda \in \rho(A_0),
\end{equation}
with $R(\lambda,A_0^{\odot\star})$ the resolvent of $A_0^{\odot\star}$, for $\lambda$ in the resolvent set $\rho(A_0) = \rho(A_0^{\odot\star}) \subseteq \KK$. It is \emph{not} assumed that $\KK = \CC$. We now discuss the relation between $X_0^{\odot\times}$ and the notion of an admissible range for $\{T_0(t)\}_{t \ge 0}$.

The following simple observation implies the invariance of $jX$ for $R(\lambda,A_0^{\odot\star})$ which, when combined with the resolvent identity, implies that $X_0^{\odot\times}$ does not depend on the choice for $\lambda$.

\begin{proposition}\label{prop:jR}
  $R(\lambda,A_0^{\odot\star})j = j R(\lambda,A_0)$.
\end{proposition}
\begin{proof}
  Let $x \in X$ be arbitrary. Then $y \DEF R(\lambda,A_0)x$ is in $\DOM(A_0)$, so
  \[
    x = (\lambda I - A_0)y = j^{-1}(\lambda I - A_0^{\odot\star})jy.
  \]
  Apply $j$ to both sides to obtain that $R(\lambda,A_0^{\odot\star})jx = jy$. 
\end{proof}

The space $X_0^{\odot\times}$ is closed and positively $T_0^{\odot\star}$-invariant. If $X$ is $\odot$-reflexive for $\{T_0(t)\}_{t \ge 0}$, then $\DOM(A_0^{\odot\star}) \subseteq jX$, so in that case $X_0^{\odot\times} = X^{\odot\star}$. In \cite[Theorem 4.2.2]{VanNeerven1992} it is proven that if $L : X \to X_0^{\odot\times}$ is bounded and linear, then there exists a unique $\mathcal{C}_0$-semigroup $\{T(t)\}_{t \ge 0}$ on $X$ that satisfies \cref{eq:aie0} with $G = L$,
\begin{equation}
  \label{eq:laie0}
  T(t)x = T_0(t)x + j^{-1}\int_0^t T_0^{\odot\star}(t - \tau)L T(\tau)x\,d\tau \qquad \ALL\,x \in X \text{ and } t \ge 0.
\end{equation}
Furthermore, in \cite[Section 4.3]{VanNeerven1992} it is proven that $X_0^{\odot\times}$ is indeed maximal in three different senses; see also \cite[Theorem III.8.4]{Delay1995} for \cite[Theorem 4.3.5]{VanNeerven1992} and \cite{DGH1989} for \cite[Theorems 4.3.5 and 4.3.8]{VanNeerven1992}. Meanwhile, \cref{prop:jR} implies the following trivial case of \cite[Theorem 4.3.6]{VanNeerven1992}.

\begin{corollary}\label{cor:jXXsuncross}
  $jX \subseteq X_0^{\odot\times}$.
\end{corollary}

We now show that $X_0^{\odot\times}$ is an admissible range for $\{T_0(t)\}_{t \ge 0}$ that is also maximal, in the sense that $X_0^{\odot\times}$ includes every range that is admissible for $\{T_0(t)\}_{t \ge 0}$.

\begin{theorem}
  \label{thm:xsuncross}
  $X_0^{\odot\times}$ is an admissible range for $\{T_0(t)\}_{t \ge 0}$. If $\mathscr{X}_0$ is an admissible range for $\{T_0(t)\}_{t \ge 0}$, then $\mathscr{X}_0 \subseteq X_0^{\odot\times}$.
\end{theorem}
\begin{proof}
  Fix $\lambda > \omega$. We first prove admissibility, then maximality.
  \begin{steps}
  \item
    Since $X_0^{\odot\times}$ is closed, by \cref{thm:admconstants} it suffices to prove that $\ONE \otimes x^{\odot\times} \in \mathscr{F}_0(\RR)$ for all $x^{\odot\times} \in X_0^{\odot\times}$. Let $x^{\odot\times} \in X_0^{\odot\times}$ and $(t,s) \in \Omega_{\RR}$ be arbitrary and write $y^{\odot\star}_{\lambda} \DEF R(\lambda,A_0^{\odot\star})x^{\odot\times}$. Then
    \begin{align*}
      \int_s^t T_0^{\odot\star}(t - \tau) x^{\odot\times}\,d\tau  &= \int_s^t T_0^{\odot\star}(t - \tau) (\lambda I - A_0^{\odot\star}) y^{\odot\star}_{\lambda}\,d\tau\\
                                                                  &= \lambda \int_s^t T_0^{\odot\star}(t - \tau) y^{\odot\star}_{\lambda}\,d\tau - \int_0^{t - s} T_0^{\odot\star}(\tau) A_0^{\odot\star} y^{\odot\star}_{\lambda} \,d\tau\\
                                                                  &= \lambda \int_s^t T_0^{\odot\star}(t - \tau) y^{\odot\star}_{\lambda} \,d\tau - (T_0^{\odot\star}(t - s) - I)y^{\odot\star}_{\lambda}.
    \end{align*}
    By definition $y^{\odot\star}_{\lambda} \in jX$ so \cref{prop:jXadmissible} implies that the integral in the right-hand side is in $jX$. Since $\{T_0^{\odot\star}(t)\}_{t \ge 0}$ and $R(\lambda,A_0^{\odot\star}) = R(\lambda,A_0^{\odot})^{\star}$ commute and \WS-integration commutes with the adjoint of a bounded linear operator, that integral is also in $\DOM(A_0^{\odot\star})$. The other term in the right-hand side is in $jX \cap \DOM(A_0^{\odot\star})$ as well, because of the commutativity of $\{T_0^{\odot\star}(t)\}_{t \ge 0}$ and $R(\lambda,A_0^{\odot\star})$ and the positive $T_0^{\odot\star}$-invariance of $X_0^{\odot\times}$. Therefore, the left-hand side is in $jX \cap \DOM(A_0^{\odot\star})$. (In fact, only the inclusion in $jX$ is needed to apply \cref{thm:admconstants}.)
  \item
    Let $\mathscr{X}_0$ be an admissible range for $\{T_0(t)\}_{t \ge 0}$ and let $x^{\odot\star} \in \mathscr{X}_0$ be arbitrary. \Cref{prop:resolvnormconv} shows that
    \begin{equation}
      \label{eq:thm:xsuncross}
      R(\lambda,A_0^{\odot\star})x^{\odot\star} = \lim_{t \to \infty} \int_0^t T_0^{\odot\star}(\tau) e^{-\lambda \tau} x^{\odot\star}\,d\tau
    \end{equation}
    in norm. The map $\tau \mapsto e^{\lambda \tau} x^{\odot\star}$ is in $\mathscr{F}_0(\RR)$, so
    \[
      \int_0^t T_0^{\odot\star}(\tau) e^{-\lambda \tau} x^{\odot\star}\,d\tau = e^{-\lambda t} \int_0^t T_0^{\odot\star}(t - \tau) e^{\lambda \tau} x^{\odot\star}\,d\tau \in jX \qquad \ALL\,t \ge 0.
    \]
    The norm convergence in \cref{eq:thm:xsuncross} then implies that $R(\lambda,A_0^{\odot\star})x^{\odot\star} \in jX$ as well. \qedhere
  \end{steps}
\end{proof}

\begin{corollary}
  \label{cor:xsuncross}
  A continuous perturbation $G : X \to X^{\odot\star}$ is admissible for $\{T_0(t)\}_{t \ge 0}$ if and only if $G$ takes its values in $X_0^{\odot\times}$.
\end{corollary}

\section{Admissibility and perturbation}\label{sec:inhomogeneous}
Let $L : X \to X^{\odot\star}$ be an admissible linear perturbation for the $\mathcal{C}_0$-semigroup $\{T_0(t)\}_{t \ge 0}$. By \cref{def:admrange} $L$ takes its values in an admissible range $\mathscr{X}_0$. The $\mathcal{C}_0$-semigroup $\{T(t)\}_{t \ge 0}$ with generator $A$ is obtained from $\{T_0(t)\}_{t \ge 0}$ by perturbation with $L$ as in \cite[Theorem 19]{ADDE1}. Given a (possibly infinite) terminal time $0 < t_e \le \infty$ and a forcing function $f : [0,t_e) \to \mathscr{X}_0$, in this section we study perturbations of $\SUNSTAR{A_0}$ and $\SUNSTAR{A}$ by $f$ on intervals in $[0,t_e)$.

In the application of dual perturbation theory to specific classes of delay equations (`abstract' or otherwise) the question of higher time regularity of the \WS-integral in \cref{eq:admissible} is usually deemed less relevant. The reason for this is the direct correspondence between mild solutions of \cref{eq:adde-ic} and solutions of the abstract \emph{integral} equation \cref{eq:aie0-adde}, so the abstract \emph{differential} equation associated with \cref{eq:aie0-adde} (or, more generally, \cref{eq:aie0}) plays at most a motivating role. Nevertheless, differentiability with respect to time, in various senses, was considered in \cite{DSG3} under the assumption of $\odot$-reflexivity. In this section we restrict our attention to \WS-differentiability.

\begin{definition}
  \label{def:wstar-derivative}
  Let $J$ be an interval and let $E$ be a Banach space. A function $q : J \to \STAR{E}$ is \textbf{\WS-differentiable} with \textbf{\WS-derivative} $\STAR{d}q : J \to \STAR{E}$ if
  \[
    \frac{d}{dt}\PAIR{x}{q(t)} = \PAIR{x}{\STAR{d}q(t)} \qquad \ALL x \in E \text{ and } t \in J.
  \]
  If in addition $\STAR{d}q$ is \WS-continuous then $q$ is called \textbf{\WS-continuously differentiable}.
\end{definition}

\begin{remark}\label{rem:wsloclip}
We note that it is a direct consequence of the uniform boundedness principle and the fundamental theorem of calculus that \WS-continuously differentiable functions in the sense of the above definition are locally Lipschitz continuous.   
\end{remark}

There are two ways to introduce inhomogeneous perturbations on the generator level. We may either perturb $\SUNSTAR{A_0}$ by $\phi \mapsto L\phi + f$ or we may perturb $\SUNSTAR{A}$ by $f$. In the first case, the initial-value problem for the associated abstract ODE is
\begin{subequations}
  \label{eq:inhom0}
  \begin{equation}
    \label{eq:aode0_inhom}
    \STAR{d}(j \circ u)(t) = \SUNSTAR{A_0} j u (t) + L u(t) + f(t), \qquad u(0) = \phi \in X,
  \end{equation}
  and formal variation-of-constants suggests the corresponding abstract integral equation
  \begin{equation}
    \label{eq:aie0_inhom}
    u(t) = T_0(t)\phi + j^{-1}\int_0^t\SUNSTAR{T_0}(t - \tau)[L u(\tau) + f(\tau)]\,d\tau.
  \end{equation}
\end{subequations}
If instead we perturb $\SUNSTAR{A}$ then the initial-value problem is
\begin{subequations}
  \label{eq:inhom}
  \begin{equation}
    \label{eq:aode_inhom}
    \STAR{d}(j \circ u)(t) = \SUNSTAR{A} j u (t) + f(t), \qquad u(0) = \phi \in X,
  \end{equation}
  along with the explicit abstract integral expression
  \begin{equation}
    \label{eq:aie_inhom}
    u(t) = T(t)\phi + j^{-1}\int_0^t\SUNSTAR{T}(t - \tau)f(\tau)\,d\tau.
  \end{equation}
\end{subequations}
We define a \textbf{subinterval} $J$ of $[0,t_e)$ to be an interval such that $0 \in J \subseteq [0,t_e)$. A \textbf{solution of} \cref{eq:aode0_inhom} on a subinterval $J$ of $[0,t_e)$ is a function $u : J \to X$ taking values in $j^{-1}\DOM(\SUNSTAR{A_0})$ and such that $j \circ u$ is \WS-continuously differentiable on $J$ and satisfies \cref{eq:aode0_inhom} there. The definition for \cref{eq:aode_inhom} is analogous. Then by \cite[Proposition 22]{ADDE1} a solution of \cref{eq:aode0_inhom} is also a solution of \cref{eq:aode_inhom} and vice versa, so in this sense \cref{eq:aode0_inhom,eq:aode_inhom} are equivalent. Two natural and interrelated questions arise:
\begin{enumerate}[wide=0pt,leftmargin=\parindent,label=\roman*.]
\item
  It must be checked that the \WS-integral in \cref{eq:aie_inhom} takes values in the range of $j$. By assumption $\mathscr{X}_0$ is an admissible range for $\{T_0(t)\}_{t \ge 0}$, but it is not clear whether this implies that $\mathscr{X}_0$ is also an admissible range for the perturbed semigroup $\{T(t)\}_{t \ge 0}$. In other words, we ask about robustness of admissibility of $\mathscr{X}_0$ with respect to the bounded linear perturbation $L$.
\item
  We expect that $u$ given by \cref{eq:aie_inhom} is the unique solution of \cref{eq:aie0_inhom} on $[0,t_e)$. Already in the $\odot$-reflexive case this seemingly obvious fact has a less than obvious proof. It relies on a combination of $\odot$-reflexivity and a variation-of-constants formula relating the integrated semigroups corresponding to $\{\SUNSTAR{T_0}(t)\}_{t \ge 0}$ and $\{\SUNSTAR{T}(t)\}_{t \ge 0}$ \cite[Proposition 2.5]{DSG3}, \cite[Lemma III.2.23]{Delay1995}. We seek an alternative proof that also works without $\odot$-reflexivity.
\end{enumerate}

\subsection{Preparatory results}
We assume the setting introduced in the first paragraph of \cref{sec:inhomogeneous}. The resolution of the two questions above can be found in \cref{thm:inhom:answers} in \cref{sec:two-answers}. Its implications are of importance for the nonlinear analysis of \cref{eq:aie0} in a neighborhood of an equilibrium solution. Here we present some preparations for the proof that may also be of independent interest. 

\begin{definition}
  \label{def:favlinear}
  The \textbf{Favard class} of the $\mathcal{C}_0$-semigroup $\{T_0(t)\}_{t \ge 0}$ is the linear subspace
  \[
    \FAV(T_0) \DEF \Bigl\{\phi \in X\,:\, \limsup_{h \downarrow 0}\frac{1}{h}\|T_0(h)\phi - \phi\| < \infty\Bigl\}.
  \]
\end{definition}

The semigroup property of $\{T_0(t)\}_{t \ge 0}$ implies that $\FAV(T_0)$ is positively $T_0$-invariant. From the definition it also follows easily that $\FAV(T_0)$ consists precisely of those $\phi \in X$ for which $T_0(\cdot)\phi$ is locally Lipschitz. Furthermore, we also have the important equalities
\begin{equation}
  \label{eq:Fav_A0sunstar}
  \FAV(T_0) = j^{-1}\DOM(\SUNSTAR{A}_0) = j^{-1}\DOM(\SUNSTAR{A}) = \FAV(T)
\end{equation}
that do not require $\odot$-reflexivity. The left and right equalities are due to \cite[(3.36) in Section 3.4]{Clement1987b} and the equality in the middle follows directly from \cite[Proposition 22]{ADDE1}.

\begin{proposition}[\normalfont{cf. \cite[Proposition 2.2]{DSG3}}]
  \label{prop:inhom:wsdiff}
  If $f : [0,t_e) \to \mathscr{X}_0$ is locally Lipschitz continuous, then $v_0(\cdot,\cdot,f) : \Omega_{[0,t_e)} \to \SUNSTAR{X}$ defined by \cref{eq:v0} takes values in $\DOM(\SUNSTAR{A_0})$ and for every fixed $s \in [0,t_e)$ the function
  \[
    [s,t_e) \ni t \mapsto v_0(t,s,f) \in \SUNSTAR{X}
  \]
  is \WS-differentiable with \WS-derivative
  \begin{equation}
    \label{eq:sunstar:16}
    \STAR{d}_t v_0(t,s,f) = \SUNSTAR{A_0}v_0(t,s,f) + f(t) \qquad \ALL t \in [s,t_e)
  \end{equation}
  and $\STAR{d}_tv_0(\cdot,\cdot,f) : \Omega_{[0,t_e)} \to \SUNSTAR{X}$ is \WS-continuous. 
\end{proposition}
\begin{proof}
  It is very close to the proof in \cite{DSG3} but here we allow for an arbitrary lower limit in the integral defining $v_0(\cdot,\cdot,f)$ and we also require that $f$ takes its values in $\mathscr{X}_0$. Therefore we provide a detailed proof. It is convenient to use the shorthands
  \[
    J \DEF [0,t_e), \qquad w(t,s) \DEF v_0(t,s,f) \qquad \ALL\,(t,s) \in \Omega_J.
  \]
  \begin{steps}
  \item
    Let $K \subseteq \Omega_J$ be an arbitrary compact subset. We will show that $w(t,s) \in \DOM(\SUNSTAR{A_0})$ for all $(t,s) \in K$. By \cref{eq:Fav_A0sunstar} we have in particular the inclusion $j \FAV(T_0) \subseteq \DOM(\SUNSTAR{A_0})$, so it is sufficient to prove that
    \begin{equation}
      \label{eq:sunstar:13}
      j^{-1}w(t,s) \in \FAV(T_0) \qquad \ALL\,(t,s) \in K.
    \end{equation}
    (Here it is used that $f$ takes its values in $\mathscr{X}_0$.) We note that the map
    \[
      \Omega_J \ni (t,s) \mapsto \sup_{s \le \tau \le t}\|f(\tau)\| \in \RR
    \]
    is continuous, so there exists a constant $C_K \ge 1$ such that
    \[
      e^{\omega(t - s)} + \sup_{s \le \tau \le t}\|f(\tau)\| \le C_K \qquad \ALL\,(t,s) \in K.
    \]
    Also, let $K_{1,2} \subseteq J$ be the compact sets obtained by projecting $K$ onto the first and second coordinate. There exists a Lipschitz constant $L_K > 0$ for $f$ on the compact set $K_1 \cup K_2$.
  \item    
    Let $(t,s) \in K$ be arbitrary. If $t = s$ then $w(t,s) = 0$, so we may assume that $t > s$ strictly. For $h \in (0,t - s)$ we consider the difference
    \begin{align*}
      \SUNSTAR{T_0}(h)w(t,s) - w(t,s) &= \int_s^t\SUNSTAR{T_0}(t + h - \tau)f(\tau)\,d\tau - \int_s^t\SUNSTAR{T_0}(t - \tau)f(\tau)\,d\tau\\
                                      &= \int_h^{t - s + h}\underbracket{\SUNSTAR{T_0}(\tau)f(t - \tau + h)}_{(\cdots)_h}\,d\tau - \int_0^{t - s}\underbracket{\SUNSTAR{T_0}(\tau)f(t - \tau)}_{(\cdots)_0}\,d\tau\\
                                      &= \int_h^{t - s}(\cdots)_h\,d\tau + \int_{t - s}^{t - s + h}(\cdots)_h\,d\tau - \int_0^h(\cdots)_0\,d\tau - \int_h^{t - s}(\cdots)_0\,d\tau,
    \end{align*}
    and this splitting leads to the estimate
    \begin{equation}
      \label{eq:sunstar:15}
      \begin{aligned}
        \frac{1}{h}\|\SUNSTAR{T_0}(h)w(t,s) - w(t,s)\| &\le \frac{1}{h}\Bigl\|\int_{t - s}^{t - s + h}\SUNSTAR{T_0}(\tau)f(t - \tau + h)\,d\tau\Bigr\|\\
        &+ \frac{1}{h}\Bigl\|\int_h^{t - s}\SUNSTAR{T_0}(\tau)[f(t - \tau + h) - f(t - \tau)]\,d\tau \Bigr\|\\
        &+ \frac{1}{h}\Bigl\|\int_0^h\SUNSTAR{T_0}(\tau)f(t - \tau)\,d\tau \Bigr\|.
      \end{aligned}
    \end{equation}
    A standard estimate shows that
    \begin{align*}
      \Bigl\|\int_{t - s}^{t - s + h}\SUNSTAR{T_0}(\tau)f(t - \tau + h)\,d\tau\Bigr\| &\le \frac{M}{\omega}(e^{\omega h} - 1)e^{\omega(t - s)}\sup_{s \le \tau \le t}\|f(\tau)\|,\\
               \Bigl\|\int_0^h\SUNSTAR{T_0}(\tau)f(t - \tau)\,d\tau \Bigr\| &\le \frac{M}{\omega}(e^{\omega h} - 1)\sup_{s \le \tau \le t}\|f(\tau)\|,
    \end{align*}
    so the superior limits of the first and third terms in the right-hand side of \cref{eq:sunstar:15} do not exceed $M C_K^2$. Also, we have the estimate
    \[
      \Bigl\|\int_h^{t - s}\SUNSTAR{T_0}(\tau)[f(t - \tau + h) - f(t - \tau)]\,d\tau \Bigr\| \le \frac{M}{\omega}L_K h (e^{\omega(t - s)} - 1),
    \]
    and this proves that the superior limit of the second term in the right-hand side of \cref{eq:sunstar:15} does not exceed $\frac{M}{\omega} L_K C_K$. We conclude that there exists a constant, again denoted by $C_K \ge 0$ and depending only on the compact set $K$, such that
    \begin{equation}
      \label{eq:limsup}
      \limsup_{h \downarrow 0} \frac{1}{h}\|\SUNSTAR{T_0}(h)w(t,s) - w(t,s)\| \le C_K \qquad \ALL\,(t,s) \in K.
    \end{equation}
  \item
    For all $(t,s) \in K$ and all $h > 0$ it holds that
    \begin{align*}
      \frac{1}{h}\|T_0(h)j^{-1}w(t,s) - j^{-1}w(t,s)\| &= \frac{1}{h}\|j^{-1}\SUNSTAR{T_0}(h)w(t,s) - j^{-1}w(t,s)\|\\
                                                       &\le \frac{\|j^{-1}\|}{h}\|\SUNSTAR{T_0}(h)w(t,s) - w(t,s)\|
    \end{align*}
    so \cref{eq:limsup} implies that \cref{eq:sunstar:13} holds. Moreover, for every $(t,s) \in K$ and every $\SUN{\phi} \in \SUN{X}$ with $\|\SUN{\phi}\| \le 1$ we have
    \[
      |\PAIR{\SUN{\phi}}{\SUNSTAR{A_0}w(t,s)}| = \lim_{h \downarrow 0}{\frac{1}{h}|\PAIR{\SUN{\phi}}{\SUNSTAR{T_0}(h)w(t,s) - w(t,s)}|} \le C_K.
    \]
    Since the compact set $K$ was chosen arbitrarily, we conclude that the map
    \begin{equation}
      \label{eq:A0v}
      \Omega_J \ni (t,s) \mapsto \SUNSTAR{A_0}w(t,s) \in \SUNSTAR{X}
    \end{equation}
    is bounded on compact subsets of $\Omega_J$.
  \item
    We show that \cref{eq:A0v} is \WS-continuous. Let $(t,s) \in \Omega_J$ be arbitrary. For any $\SUN{\phi} \in \DOM(\SUN{A_0})$ we have
\[
  \PAIR{\SUN{\phi}}{\SUNSTAR{A_0}w(t,s)} = \PAIR{\SUN{A_0}\SUN{\phi}}{w(t,s)},
\]
which is a continuous function of $(t,s)$ by the continuity of $w$ on $\Omega_J$. Next, let $\SUN{\phi} \in \SUN{X}$ be arbitrary. By norm-density of $\DOM(\SUN{A_0})$ in $\SUN{X}$ there exists a sequence $(\SUN{\phi}_n)$ in $\DOM(\SUN{A_0})$ such that $\SUN{\phi_n} \to \SUN{\phi}$ in norm as $n \to \infty$. Let $(t_m,s_m)$ be a sequence in $\Omega_J$ converging to $(t,s)$ as $m \to \infty$. Then we estimate
    \begin{align*}
      |\PAIR{\SUN{\phi}}{\SUNSTAR{A_0}w(t,s)} - \PAIR{\SUN{\phi}}{\SUNSTAR{A_0}w(t_m,s_m)}| &\le (\|\SUNSTAR{A_0}w(t,s)\| + \|\SUNSTAR{A_0}w(t_m,s_m)\|)\cdot\|\SUN{\phi} - \SUN{\phi_n}\|\\
                                                                                                    &+ |\PAIR{\SUN{\phi_n}}{\SUNSTAR{A_0}w(t,s)} - \PAIR{\SUN{\phi_n}}{\SUNSTAR{A_0}w(t_m,s_m)}|.
    \end{align*}
    The first term in the right-hand side can be made as small as desired merely by fixing $n$ sufficiently large, thanks to the boundedness of \cref{eq:A0v} on compact subsets of $\Omega_J$. The continuity of $\PAIR{\SUN{\phi_n}}{\SUNSTAR{A_0}w(\cdot,\cdot)}$ then implies that the second term becomes arbitrarily small as $m \to \infty$.
  \item
    It remains to prove the statement about \WS-differentiability. We do this by showing that
    \begin{equation}
      \label{eq:sunstar:17}
      w(t,s) = \int_s^t{[\SUNSTAR{A_0}w(\tau,s) + f(\tau)]\,d\tau} \qquad \ALL t \in J \text{ with } t \ge s.
    \end{equation}
    (Observe that the right-hand side is well-defined as a \WS-Riemann integral since the integrand is \WS-continuous, as shown in the previous step.) Indeed, if \cref{eq:sunstar:17} holds, then for every $\SUN{\phi} \in \SUN{X}$ we have
    \[
      \frac{d}{dt} \PAIR{\SUN{\phi}}{w(t,s)} = \frac{d}{dt} \int_s^t{\PAIR{\SUN{\phi}}{\SUNSTAR{A_0}w(\tau,s) + f(\tau)}\,d\tau} = \PAIR{\SUN{\phi}}{\SUNSTAR{A_0}w(t,s) + f(t)},
    \]
     which is precisely \cref{eq:sunstar:16}. A small direct calculation shows that
    \[
      \PAIR{\SUN{\phi}}{\int_s^t\SUNSTAR{A_0}w(\tau,s)\,d\tau} = \PAIR{\SUN{\phi}}{w(t,s) - \int_s^tf(\tau)\,d\tau} \qquad \ALL t \in J \text{ with } t \ge s,
    \]
    at first for $\SUN{\phi} \in \DOM(\SUN{A_0})$ and then by norm-density for all $\SUN{\phi} \in \SUN{X}$. (Change the order of integration so \cite[Lemma 3.15 in Appendix II.3]{Delay1995} can be applied.) This proves \cref{eq:sunstar:17}.
    \hfill \qedhere
  \end{steps}
\end{proof}

\begin{corollary}\label{cor:inhom:aie0_to_aode0}
  Suppose that $\phi \in j^{-1}\DOM(\SUNSTAR{A_0})$ and $f$ is locally Lipschitz. If $J$ is a subinterval of $[0,t_e)$ and $u$ is a locally Lipschitz continuous solution of \cref{eq:aie0_inhom} on $J$ then $u$ is a solution of \cref{eq:aode0_inhom} on $J$.
\end{corollary}
\begin{proof}
  Apply $j$ to \cref{eq:aie0_inhom} to obtain
  \begin{equation}
    \label{eq:jut}
    ju(t) = \SUNSTAR{T_0}(t)j\phi + \int_0^t\SUNSTAR{T_0}(t - \tau)[Lu(\tau) + f(\tau)]\,d\tau \qquad \ALL t \in J.
  \end{equation}
  The first term on the right takes values in $\DOM(\SUNSTAR{A_0})$ and it is \WS-continuously differentiable with respect to $t \in J$ with \WS-derivative
  \[
    \STAR{d}_t \SUNSTAR{T_0}(t)j\phi = \SUNSTAR{A_0}\SUNSTAR{T_0}(t)j\phi.
  \]
  Also, by \cref{prop:inhom:wsdiff} the \WS-integral in \cref{eq:jut} takes values in $\DOM(\SUNSTAR{A_0})$ and it is \WS-continuously differentiable with respect to $t \in J$ with \WS-derivative
  \[
    \STAR{d}_t \int_0^t\SUNSTAR{T_0}(t - \tau)[Lu(\tau) + f(\tau)]\,d\tau = \SUNSTAR{A_0} \int_0^t\SUNSTAR{T_0}(t - \tau)[Lu(\tau) + f(\tau)]\,d\tau + Lu(t) + f(t).
  \]
  So $u$ takes values in $j^{-1}\DOM(\SUNSTAR{A_0})$ and $j \circ u$ is \WS-continuously differentiable and satisfies \cref{eq:aode0_inhom} on $J$.
\end{proof}

Concerning the following proposition, at first sight it may seem a bit odd to establish well-posedness of the \emph{linear} inhomogeneous problem \cref{eq:aie0_inhom} by means of a fixed-point argument. However, direct substitution of \cref{eq:aie_inhom} into \cref{eq:aie0_inhom} is not successful in the non-$\odot$-reflexive case, even if just for the reason given in the first of the two questions above. Indeed, the following independent well-posedness result will turn out to be instrumental in the proof of \cref{thm:inhom:answers}.

\begin{proposition}\label{prop:inhom:aie0sol}
  Let $J \subseteq [0,t_e)$ be a compact subinterval. The following two statements hold.
  \begin{enumerate}[label=\roman*.]
  \item
    For every $\phi \in X$ there exists a unique solution $u_{\phi,f}$ of \cref{eq:aie0_inhom} on $J$ and the map
    \begin{equation}
      \label{eq:umap_cont}
      X \times C(J,\mathscr{X}_0) \ni (\phi,f) \mapsto u_{\phi,f} \in C(J,X)
    \end{equation}
    is continuous.
  \item
    If $\phi \in j^{-1}\DOM(\SUNSTAR{A_0})$ and $f$ is locally Lipschitz then there exist sequences of Lipschitz functions $u_m : J \to X$ and $f_m : J \to \mathscr{X}_0$ such that
    \begin{equation}
      \label{eq:aie0_inhom_approx}
      u_m(t) = T_0(t)\phi + j^{-1}\int_0^t\SUNSTAR{T_0}(t - \tau)[L u_m(\tau) + f_m(\tau)]\,d\tau \qquad \ALL\,t \in J,
    \end{equation}
    and $f_m \to f$ and $u_m \to u_{\phi,f}$ as $m \to \infty$, both uniformly on $J$.
  \end{enumerate}
\end{proposition}
\begin{proof}
  \begin{steps}
  \item
    The first statement is proven by a standard fixed-point argument. Let $M \ge 1$ and $\omega \in \RR$ be as in \cref{eq:expboundT0}. Following \cite{Duistermaat1995}, on the Banach space $C(J,X)$ we introduce the one-parameter family of equivalent norms
    \[
      \|u\|_{\eta} \DEF \sup_{t \in J}e^{-\eta t}\|u(t)\|, \qquad \eta \in \RR.
    \]
    Clearly each of these norms is complete and $\|\cdot\|_0$ is the usual supremum-norm. For each $(\phi,f) \in X \times C(J,\mathscr{X}_0)$ we define the operator $K_{\phi,f}$ on $C(J,X)$ by
    \begin{equation}
      \label{eq:inhom_fp}
      (K_{\phi,f}u)(t) \DEF T_0(t)\phi + j^{-1}\int_0^t\SUNSTAR{T_0}(t - \tau)[Lu(\tau) + f(\tau)]\,d\tau \qquad \ALL\, t \in J.
    \end{equation}
    Choose $\eta > \omega$. For all $u, \hat{u} \in C(J,X)$ and all $t \in J$ we have
    \begin{align*}
      e^{-\eta t}\|(K_{\phi,f}u)(t) - (K_{\phi,f}\hat{u})(t)\| &\le \|j^{-1}\| M \|L\| \int_0^t e^{-(\eta - \omega)(t - \tau)} e^{-\eta \tau}\|u(\tau) - \hat{u}(\tau)\|\,d\tau\\
                                             &\le \frac{\|j^{-1}\| M \|L\|}{\eta - \omega} \|u - \hat{u}\|_{\eta}
    \end{align*}
    so if we choose $\eta$ such that $2\|j^{-1}\| M \|L\| \le \eta - \omega$ then $K_{\phi,f}$ is a uniform contraction with respect to the $\|\cdot\|_{\eta}$-norm. Moreover, the maps $X \times C(J,\mathscr{X}_0) \ni (\phi,f) \mapsto K_{\phi,f}u \in C(J,X)$ are continuous for each fixed $u \in C(J,X)$. The uniform contraction principle \cite[Theorem 0.3.2]{Hale1980} therefore gives the first statement.
  \item
    Assume in addition that $\phi \in j^{-1}\DOM(\SUNSTAR{A_0})$ and $f$ is locally Lipschitz. Let $\LIP(J,X)$ be the subspace of $C(J,X)$ consisting of Lipschitz functions. We will show that $K_{\phi,f}$ maps $\LIP(J,X)$ into itself. First, from \cref{eq:Fav_A0sunstar} we see that $T_0(\cdot)\phi$ is in $\LIP(J,X)$.  Second, let $u \in \LIP(J,X)$ be arbitrary and let $\hat{u}$ be a Lipschitz extension of $u$ to $[0,t_e)$. Then $L \circ \hat{u} + f$ is locally Lipschitz on $[0,t_e)$ with values in $\mathscr{X}_0$. \Cref{prop:inhom:wsdiff} and \cref{rem:wsloclip} show that $v_0(\cdot,0,L \circ \hat{u} + f)$ is locally Lipschitz on $[0,t_e)$ as well. Hence $K_{\phi,f}u = T_0(\cdot)\phi + j^{-1}\circ v_0(\cdot,0,L \circ \hat{u} + f)$ is in $\LIP(J,X)$.
  \item
    Choose an arbitrary $u_0 \in \LIP(J,X)$. The sequence of fixed-point iterates defined by
    \[
      u_m \DEF K_{\phi,f}u_{m-1}, \qquad m \in \NN,
    \]
    is in $\LIP(J,X)$. From \cref{eq:inhom_fp} we have for all $m \in \NN$ and all $t \in J$,
    \[
      u_m(t) = T_0(t)\phi + j^{-1}\int_0^t\SUNSTAR{T_0}(t - \tau)[L u_m(\tau) + f(\tau) + L\{u_{m-1}(\tau) - u_m(\tau)\}]\,d\tau.
    \]
    For every $m \in \NN$ we define $f_m \DEF f + L \circ (u_{m-1} - u_m)$. Then each $f_m$ is Lipschitz on $J$ with values in $\mathscr{X}_0$. Furthermore, \cref{eq:aie0_inhom_approx} holds and $u_m \to u_{\phi,f}$ and $f_m \to f$, both uniformly on $J$. \hfill \qedhere
  \end{steps}
\end{proof}


We recall that $\{T(t)\}_{t \ge 0}$ is the $\mathcal{C}_0$-semigroup defined in \cite[Theorem 19]{ADDE1}. For any interval $J$ we denote by $\mathscr{F}(J)$ the class of admissible forcing functions for $\{T(t)\}_{t \ge 0}$ on $J$. 

\begin{proposition}\label{prop:inhom:aode_to_aie}
  If $J$ is a subinterval of $[0,t_e)$ and $u : J \to X$ is a solution of \cref{eq:aode_inhom} then $f \in \mathscr{F}(J)$ and $u$ is given by \cref{eq:aie_inhom}.
\end{proposition}
\begin{proof}
  \begin{steps}
  \item
    The proof is rather standard, see for instance \cite[Lemma 13]{ADDE1}, but in the present case we need to work with the \WS-topology instead of the norm-topology. Let $(t,s) \in \Omega_J$ with $t > s$ be arbitrary. Define $w : [s,t] \to \SUNSTAR{X}$ by $w(\tau) \DEF \SUNSTAR{T}(t - \tau)ju(\tau)$. We claim that $w$ is \WS-differentiable with derivative
    \begin{equation}
      \label{eq:sunstar:30}
      \STAR{d}w(\tau) = \SUNSTAR{T}(t - \tau)\STAR{d}(j \circ u)(\tau) - \SUNSTAR{T}(t - \tau)\SUNSTAR{A}ju(\tau) \qquad \ALL\,\tau \in [s,t].
    \end{equation}
    which is just what one would expect on formal grounds. To prove it, let $\tau \in [s,t]$ and $\SUN{\phi} \in \SUN{X}$ be arbitrary. For any $h \in \RR$ such that $\tau + h \in [s,t]$ we have
    \begin{align*}
      \PAIR{\SUN{\phi}}{w(\tau + h) - w(\tau)} &= \PAIR{\SUN{\phi}}{\SUNSTAR{T}(t - (\tau + h))ju(\tau + h) - \SUNSTAR{T}(t -\tau)ju(\tau)}\\
      &= \PAIR{\SUN{\phi}}{\SUNSTAR{T}(t - (\tau + h))[ju(\tau + h) - ju(\tau)]}\\
      &+ \PAIR{\SUN{\phi}}{[\SUNSTAR{T}(t - (\tau + h)) - \SUNSTAR{T}(t - \tau)]ju(\tau)}\\
      &= \PAIR{\SUN{T}(t - (\tau + h))\SUN{\phi}}{ju(\tau + h) - ju(\tau)}\\
      &+ \PAIR{\SUN{\phi}}{[\SUNSTAR{T}(t - (\tau + h)) - \SUNSTAR{T}(t - \tau)]ju(\tau)}
    \end{align*}
    Regarding the first pairing above, by the strong continuity of $\SUN{T}$ we have
    \[
      \SUN{T}(t - (\tau + h))\SUN{\phi} \to \SUN{T}(t - \tau)\SUN{\phi} \qquad \text{in norm as } h \to 0.
    \]
    We also have
    \[
\frac{1}{h}(ju(\tau + h) - ju(\tau)) \to \STAR{d}(j \circ u)(\tau) \qquad \text{\WSTARLY as } h \to 0,
    \]
    while the difference quotient remains bounded in norm thanks to the Lipschitz continuity of $j \circ u$ on $[s,t]$. Together this implies that
    \[
      \frac{1}{h}\PAIR{\SUN{T}(t - (\tau + h))\SUN{\phi}}{ju(\tau + h) - ju(\tau)} \to \PAIR{\SUN{T}(t - \tau)\SUN{\phi}}{\STAR{d}(j \circ u)(\tau)} \qquad \text{as } h \to 0.
    \]
    Regarding the second pairing, since $ju(\tau) \in \DOM(\SUNSTAR{A})$ it follows that
    \[
       \frac{1}{h}\PAIR{\SUN{\phi}}{[\SUNSTAR{T}(t - (\tau + h)) - \SUNSTAR{T}(t - \tau)]ju(\tau)} \to -\PAIR{\SUN{\phi}}{\SUNSTAR{T}(t - \tau)\SUNSTAR{A}ju(\tau)} \qquad \text{as } h \to 0.
    \]   
    Consequently, we have
    \[
      \frac{1}{h}\PAIR{\SUN{\phi}}{w(\tau + h) - w(\tau)} \to \PAIR{\SUN{\phi}}{\SUNSTAR{T}(t - \tau)\STAR{d}(j \circ u)(\tau) - \SUNSTAR{T}(t - \tau)\SUNSTAR{A}ju(\tau)}
    \]
    and this proves \cref{eq:sunstar:30}.
  \item
    Substituting \cref{eq:aode_inhom} into \cref{eq:sunstar:30} yields
    \[
      \STAR{d}w(\tau) = \SUNSTAR{T}(t - \tau)f(\tau) \qquad \ALL\,\tau \in [s,t]
    \] 
    so $\STAR{d}w$ is \WS-continuous. For every $\SUN{\phi} \in \SUN{X}$ we have
    \begin{align*}
      \PAIR{\SUN{\phi}}{ju(t) - \SUNSTAR{T}(t - s)ju(s)} &= \PAIR{\SUN{\phi}}{w(t)} - \PAIR{\SUN{\phi}}{w(s)}\\
                                                         &= \int_s^t{\PAIR{\SUN{\phi}}{\STAR{d}w(\tau)}\,d\tau}\\
                                                         &= \int_s^t{\PAIR{\SUN{\phi}}{\SUNSTAR{T}(t - \tau)f(\tau)}\,d\tau}
    \end{align*}
    Since $\SUN{\phi}$ and also $s$ and $t$ were arbitrary, we conclude that 
    \[
      ju(t) - \SUNSTAR{T}(t - s)ju(s) = \int_s^t{\SUNSTAR{T}(t - \tau)f(\tau)\,d\tau} \qquad \ALL\,(t,s) \in \Omega_J.
    \]
    We recall that $\SUNSTAR{T}(t - s)ju(s) = jT(t - s)u(s)$, so the above equality implies that \cref{eq:admissible} holds as well with $\{T_0(t)\}_{t \ge 0}$ replaced by $\{T(t)\}_{t \ge 0}$. We conclude that indeed $f \in \mathscr{F}(J)$ and $u$ is given by \cref{eq:aie_inhom} on $J$. \qedhere
  \end{steps}
\end{proof}

\subsection{A theorem on admissibility and perturbation}\label{sec:two-answers}
We are now in a position to answers the two questions that were asked following \cref{eq:inhom0,eq:inhom}.

\begin{theorem}\label{thm:inhom:answers}
  The following two statements hold.
  \begin{thmenum}
  \item\label{thm:inhom:answers:admrange}
    $\mathscr{X}_0$ is an admissible range for $\{T(t)\}_{t \ge 0}$.
  \item\label{thm:inhom:answers:sol}
    The unique solution of \cref{eq:aie0_inhom} on $[0,t_e)$ is given by \cref{eq:aie_inhom}.
  \end{thmenum}
\end{theorem}
\begin{proof}
  In the first two steps we prove the first statement of the theorem and we show that the unique solution of \cref{eq:aie0_inhom} on any compact subinterval of $[0,t_e)$ is given by \cref{eq:aie_inhom}. In the third step it will then be easy to extend the latter result to $[0,t_e)$ itself.
  \begin{steps}
  \item
   Let $J$ be an arbitrary compact subinterval of $[0,t_e)$. Assume that $\phi \in j^{-1}\DOM(\SUNSTAR{A_0})$ and that $f : [0,t_e) \to \mathscr{X}_0$ is locally Lipschitz. \Cref{prop:inhom:aie0sol} gives a unique solution $u_{\phi,f} : J \to X$ of \cref{eq:aie0_inhom} as well as the sequences of Lipschitz functions $u_m : J \to X$ and $f_m : J \to \mathscr{X}_0$ appearing in \cref{eq:aie0_inhom_approx}. For each $m \in \NN$ let $\hat{f}_m : [0,t_e) \to \mathscr{X}_0$ be a Lipschitz extension of $f_m$. \Cref{cor:inhom:aie0_to_aode0} with $\hat{f}_m$ and $u_m$ instead of $f$ and $u$ shows that each $u_m$ is a solution of the initial-value problem
    \[
      \STAR{d}(j \circ u_m)(t) = \SUNSTAR{A_0} j u_m(t) + L u_m(t) + \hat{f}_m(t), \qquad t \in J, \qquad u_m(0) = \phi.
    \]
    Hence each $u_m$ is also a solution of the initial-value problem
    \[
      \STAR{d}(j \circ u_m)(t) = \SUNSTAR{A} j u_m(t) + \hat{f}_m(t), \qquad t \in J, \qquad u_m(0) = \phi.
    \]
    \Cref{prop:inhom:aode_to_aie} with $u_m$ and $\hat{f}_m$ instead of $u$ and $f$ implies that 
    \begin{equation}
      \label{eq:aie_inhom_approx}
      f_m \in \mathscr{F}(J), \qquad u_m(t) = T(t)\phi + j^{-1}\int_0^t\SUNSTAR{T}(t - \tau)f_m(\tau)\,d\tau, \qquad \ALL\,m \in \NN,\,t \in J.
    \end{equation}
    Here it was also used that $\hat{f}_m|_J = f_m$ for all $m \in \NN$. \Cref{prop:admclosed} for $\{T(t)\}_{t \ge 0}$ instead of $\{T_0(t)\}_{t \ge 0}$ and the uniform convergence of $f_m$ to $f$ on $J$ then imply that $f \in \mathscr{F}(J)$. Since $J$ was chosen arbitrarily, this proves that $f \in \mathscr{F}([0,t_e))$. \Cref{cor:admlipschitz} lets us conclude that $\mathscr{X}_0$ is an admissible range for $\{T(t)\}_{t \ge 0}$.
  \item
    Taking the limit $m \to \infty$ in \cref{eq:aie_inhom_approx} we obtain the identity
    \begin{equation}
      \label{eq:aie_inhom_dense}
      u_{\phi,f}(t) = T(t)\phi + j^{-1}\int_0^t\SUNSTAR{T}(t - \tau)f(\tau)\,d\tau, \qquad \ALL\,t \in J,
    \end{equation}
    and for all $\phi \in j^{-1}\DOM(\SUNSTAR{A_0})$ and all locally Lipschitz functions $f : [0,t_e) \to \mathscr{X}_0$. The set of such pairs $(\phi,f)$ is dense in $X \times C(J,\mathscr{X}_0)$, so the continuity of \cref{eq:umap_cont} implies that \cref{eq:aie_inhom_dense} holds for all $\phi \in X$ and all continuous functions $f : [0,t_e) \to \mathscr{X}_0$. Therefore the unique solution $u_{\phi,f}$ of \cref{eq:aie0_inhom} on $J$ is given by \cref{eq:aie_inhom}.
  \item
    We extend this result to $[0,t_e)$. \Cref{prop:inhom:aie0sol} implies that any solution of \cref{eq:aie0_inhom} on $[0,t_e)$ must be unique. Define $\hat{u}_{\phi,f} : [0,t_e) \to X$ by $\hat{u}_{\phi,f}(t) \DEF u_{\phi,f}^J(t)$ where $J$ is a compact subinterval of $[0,t_e)$ such that $t \in J$ and $u_{\phi,f}^J$ is the unique solution of \cref{eq:aie0_inhom} on $J$. Then $\hat{u}_{\phi,f}$ is well-defined, continuous and satisfies \cref{eq:aie0_inhom} on $[0,t_e)$. The definition of $\hat{u}_{\phi,f}$ and the fact that each $u_{\phi,f}^J$ is given by \cref{eq:aie_inhom} imply that $\hat{u}_{\phi,f}$ itself is given by \cref{eq:aie_inhom}. \hfill \qedhere
  \end{steps}
\end{proof}

Once it has been established that $\mathscr{X}_0$ is an admissible range for $\{T(t)\}_{t \ge 0}$ as well, it becomes possible to prove \cref{thm:inhom:answers:sol} in a way that exploits the symmetry between the semigroups $\{T_0(t)\}_{t \ge 0}$ and $\{T(t)\}_{t \ge 0}$. This may be of interest particularly in the $\odot$-reflexive case, when $\SUNSTAR{X}$ itself is an admissible range for $\{T(t)\}_{t \ge 0}$ and the following is an alternative for \cite[Proposition 2.5]{DSG3} and \cite[Lemma III.2.23]{Delay1995}.

\begin{proof}[Alternative proof of \cref{thm:inhom:answers:sol}]
  For every $\phi \in X$ and every $f : [0,t_e) \to \mathscr{X}_0$ continuous we define $u_{\phi,f} : [0,t_e) \to X$ by \cref{eq:aie_inhom}. \Cref{thm:inhom:answers:admrange} ensures that $u_{\phi,f}$ is well-defined. First assume that $\phi \in j^{-1}\DOM(\SUNSTAR{A})$ and that $f$ is locally Lipschitz, so $u_{\phi,f}$ is locally Lipschitz. We now apply \cref{cor:inhom:aie0_to_aode0} with $\{T(t)\}_{t \ge 0}$ in place of $\{T_0(t)\}_{t \ge 0}$ and $-L \circ u_{\phi,f} + f$ in place of $f$. This yields that $u_{\phi,f}$ is a solution on $[0,t_e)$ of
  \[
    \STAR{d}(j \circ u_{\phi,f})(t) = \SUNSTAR{A}ju_{\phi,f}(t) + f(t), \qquad u(0) = \phi,
  \]
  and therefore $u_{\phi,f}$ is a solution on $[0,t_e)$ of
  \[
    \STAR{d}(j \circ u_{\phi,f})(t) = \SUNSTAR{A_0}ju_{\phi,f}(t) + [Lu_{\phi,f}(t) + f(t)], \qquad u(0) = \phi.
  \]
  Next, we apply \cref{prop:inhom:aode_to_aie} with $\{T_0(t)\}_{t \ge 0}$ in place of $\{T(t)\}_{t \ge 0}$ and $L \circ u_{\phi,f} + f$ in place of $f$ to conclude that $u_{\phi,f}$ satisfies
  \begin{equation}
      \label{eq:aie_inhom_dense_alt}
    u_{\phi,f}(t) = T_0(t)\phi + j^{-1}\int_0^t\SUNSTAR{T_0}(t - \tau)[L u_{\phi,f}(\tau) + f(\tau)]\,d\tau
  \end{equation}
  for all $t \in [0,t_e)$, all $\phi \in j^{-1}\DOM(\SUNSTAR{A})$ and all locally Lipschitz functions $f : [0,t_e) \to \mathscr{X}_0$. For every compact subinterval $J$ of $[0,t_e)$ the set of all such points $(\phi,f)$ is dense in $X \times C(J,\mathscr{X}_0)$ and from \cref{eq:aie_inhom} we see that the map \cref{eq:umap_cont} is continuous. Therefore \cref{eq:aie_inhom_dense_alt} holds for all $t \in J$, all $\phi \in X$ and all continuous $f : [0,t_e) \to \mathscr{X}_0$. Since $J$ can be chosen arbitrarily, this concludes the proof.
\end{proof}

In conclusion of this section we return to the space $X_0^{\odot\times}$ from \cref{eq:x0suncross}. Define $X^{\odot\times}$ as the analogue for $\{T(t)\}_{t \ge 0}$ of that space,
\begin{equation}
  \label{eq:xsuncross}
  X^{\odot\times} \DEF \{x^{\odot\star} \in X^{\odot\star}\,:\, R(\lambda,A^{\odot\star})x^{\odot\star} \in jX\}, \qquad \lambda \in \rho(A).
\end{equation}
As a consequence of \cref{thm:xsuncross,thm:inhom:answers:admrange} the subscript can - and will - be dropped.
\begin{corollary}
  \label{cor:x0suncross-is-xsuncross}
  The maximal admissible ranges for $T$ and $T_0$ coincide: $X^{\odot\times} = X_0^{\odot\times}$.
\end{corollary}

\section{The nonlinear semiflow and linearization}\label{sec:linearization}
\emph{All vector spaces in this section are over $\RR$.}

\medskip
\noindent Let $\{T_0(t)\}_{t \ge 0}$ be a $\mathcal{C}_0$-semigroup on $X$. We assume:

\begin{hypothesis}\label{hyp:G}
  $G : X \to X^{\odot\star}$ in \cref{eq:aie0} is an admissible perturbation for $\{T_0(t)\}_{t \ge 0}$ and of class $C^k$ for some $k \ge 1$.
\end{hypothesis}

\noindent \Cref{cor:xsuncross,cor:x0suncross-is-xsuncross} show that for the first part we could equivalently have assumed that $G$ takes its values in the space $X^{\odot\times}$. The mean value inequality \cite[Theorem 1.8 in Chapter 1]{AmbrosettiProdi1993} implies that $G$ is locally Lipschitz. The properties of admissibility and local Lipschitz continuity together guarantee that for each $\phi \in X$ there exists a maximal solution $u_{\phi} : J_{\phi} \to X$ of \cref{eq:aie0} on a maximal interval of existence $J_{\phi} = [0,t_{\phi})$ for some $0 < t_{\phi} \le \infty$. This is proven exactly as in \cite[Theorem VII.3.4]{Delay1995} but with admissibility of $G$ for $\{T_0(t)\}_{t \ge 0}$ as a substitute for $\odot$-reflexivity of $X$ with respect to $\{T_0(t)\}_{t \ge 0}$. With the family of maximal solutions of \cref{eq:aie0} we then associate the map $\Sigma : \DOM(\Sigma) \to X$ defined by
\begin{equation}
  \label{eq:semiflow}
  \DOM(\Sigma) \DEF \{(t,\phi) \in \RR_+ \times X\,:\, t \in J_{\phi}\}, \qquad \Sigma(t,\phi) \DEF u_{\phi}(t),
\end{equation}
and we can verify that $\Sigma$ is a semiflow on $X$ in the sense of \cite[Definition VII.2.1]{Delay1995}.

\subsection{Splitting of the perturbation and linearization}
Let $\hat{\phi} \in X$ be an equilibrium of $\Sigma$. We assume without loss of generality that $\hat{\phi} = 0$, i.e.
\[
  J_0 = \RR_+, \qquad \Sigma(t,0) = 0 \qquad \ALL\,t \ge 0.
\]
It is not difficult to verify that this is equivalent to the condition $G(0) = 0$. Let $L \DEF DG(0) : X \to \SUNSTAR{X}$ be the Fr\'echet derivative of $G$ at $\hat{\phi}$ and write
\begin{equation}
  \label{eq:splittingGLR}
  G(\phi) = L\phi + R(\phi), \qquad \phi \in X,
\end{equation}
which defines the $C^k$-smooth operator $R : X \to \SUNSTAR{X}$ as the nonlinear part of $G$ at $\hat{\phi}$. The admissibility of $G$ for $\{T_0(t)\}_{t \ge 0}$ implies the admissibility of $L$ and $R$ for $\{T_0(t)\}_{t \ge 0}$, so in particular $L$ satisfies \textbf{(H0)} in \cite[Section 6]{ADDE1}. Let $\{T(t)\}_{t \ge 0}$ be the $\mathcal{C}_0$-semigroup defined in \cite[Theorem 19]{ADDE1}. We now consider the abstract integral equation
\begin{equation}
  \label{eq:aie}
  u(t) = T(t)\phi + j^{-1}\int_0^t\SUNSTAR{T}(t - \tau)R(u(\tau))\,d\tau, \qquad t \ge 0.
\end{equation} 
\Cref{thm:inhom:answers:admrange} implies that $R$ is an admissible perturbation for $\{T(t)\}_{t \ge 0}$, so the \WS-Riemann integral in \cref{eq:aie} takes values in $jX$. The following is a generalization of \cite[Proposition VII.5.4]{Delay1995}.
 
\begin{proposition}\label{prop:aie0_aie}
  Given $0 < t_e \le \infty$ and $\phi \in X$, a function $u : [0,t_e) \to X$ is a solution of \cref{eq:aie0} if and only if $u$ is a solution of \cref{eq:aie}.
\end{proposition}
\begin{proof}
  Suppose that $u$ is a solution of \cref{eq:aie0} on $[0,t_e)$. Then
  \[
    u(t) = T_0(t)\phi + j^{-1}\int_0^t\SUNSTAR{T_0}(t - \tau)[Lu(\tau) + R(u(\tau))]\,d\tau \qquad \ALL\,t \in [0,t_e),
  \]
  so $u$ is a solution of \cref{eq:aie0_inhom} on $[0,t_e)$ with $f = R \circ u$. \Cref{thm:inhom:answers:sol} implies that $u$ is given by \cref{eq:aie_inhom} with $f = R \circ u$, so $u$ satisfies \cref{eq:aie} on $[0,t_e)$. The converse is proven by reversing the order of the steps.
\end{proof}

\begin{proposition}\label{prop:linearization}
  Let $\hat{\phi} \in X$ be an equilibrium of the semiflow $\Sigma$. Then $\Sigma$ is partially differentiable with respect to the state at $\hat{\phi}$, uniformly on compact time intervals, with partial derivative $D_2\Sigma(t,\hat{\phi}) = T(t)$. Explicitly, for every $t_0 \ge 0$ and every $\EPS > 0$ there exists $\delta > 0$ such that
  \[
    \|\Sigma(t,\phi) - \hat{\phi} - T(t)(\phi - \hat{\phi})\| \le \EPS \|\phi - \hat{\phi}\|
  \]
  for all $\phi \in X$ with $\|\phi - \hat{\phi}\| \le \delta$ and for all $t \in [0,t_0]$.
\end{proposition}
\begin{proof}
  We may assume that $\hat{\phi} = 0$. \Cref{prop:aie0_aie} implies that $\Sigma(\cdot,\phi)$ is a solution of \cref{eq:aie} on $J_{\phi}$ for every $\phi \in X$, so
  \begin{equation}
    \label{eq:sigma_T_diff}
    \Sigma(t,\phi) - T(t)\phi = j^{-1}\int_0^t\SUNSTAR{T}(t - \tau)R(\Sigma(\tau,\phi))\,d\tau \qquad \ALL\,t \in J_{\phi}.
  \end{equation}
  In order to estimate the right-hand side, we can proceed exactly as in the proof of \cite[Proposition VII.5.5]{Delay1995}, but with the following minor changes. First, in this context we are not interested in uniformity with respect to the parameter, and this shortens the proof. Second, at the beginning of the proof of the claim in step 4., assuming the contrary gives the existence of $t \in (t_1,t_0]$ such that
  \[
    \|\Sigma(s,\phi)\| < \overline{\delta} \text{ for all } s \in [0,t) \text{ and } \|\Sigma(t,\phi)\| \ge \overline{\delta},
  \]
  which corrects some small misprints. All else in the proof remains valid without modifications.
\end{proof}

\begin{remark}
  If $G$ satisfies a \emph{global} Lipschitz condition with Lipschitz constant $L_G$ then the proof becomes a lot less subtle, because the a priori estimate of $\Sigma(t,\phi)$ on $[0,t_0]$ can now easily be derived using Gronwall's inequality in integral form \cite[Corollary I.6.6]{Hale1980}. Indeed, from \cref{eq:aie0} we have
  \[
    e^{-\omega t}\|\Sigma(t,\phi)\| \le M \|\phi\| + \|j^{-1}\| M L_G \int_0^t e^{-\omega \tau} \|\Sigma(\tau,\phi)\|\,d\tau,
  \]
  so Gronwall implies
  \[
    \|\Sigma(t,\phi)\| \le M e^{(\omega + \|j^{-1}\| M L_G)t} \|\phi\|
  \]
  for all $\phi \in X$ and for all $t \in J_{\phi}$. This estimate is precisely of the form \cite[(5.5) in the proof of Proposition VII.5.5]{Delay1995}.
\end{remark}

\subsection{The translation-invariant integral equation}
As part of the construction of local center manifolds in \cref{sec:cm}, we will be interested in solutions that exist for all time, such as periodic orbits. In order to discuss such solutions in a meaningful way, we introduce the translation-invariant version of \cref{eq:aie},
\begin{equation}
  \label{eq:aie-st}
  u(t) = T(t-s)u(s) + j^{-1}\int_s^t{\SUNSTAR{T}(t - \tau)R(u(\tau))\,d\tau}, \qquad -\infty < s \le t < \infty,
\end{equation}
and we briefly comment on the simple relationship between \cref{eq:aie,eq:aie-st}. Let $J$ be an interval. By definition, a \textbf{solution on $J$} of \cref{eq:aie-st} is a continuous function $u : J \to X$ that satisfies \cref{eq:aie-st} for all $(t,s) \in \Omega_J$.

\begin{proposition}\label{prop:aie-translation-invariant}
  Let $J$ be an interval. The function $u : J \to X$ is a solution of \textup{\cref{eq:aie-st}} if and only if
  \begin{equation}
    \label{eq:usigma}
    t - s \in J_{u(s)}, \qquad u(t) = \Sigma(t - s, u(s)),
  \end{equation}
  for all $(t,s) \in \Omega_J$.
\end{proposition}
\begin{proof}
  Suppose first that $u$ is a solution of \cref{eq:aie-st} and let $(t,s) \in \Omega_J$. We may assume $s < t$ strictly, for otherwise \cref{eq:usigma} holds trivially. Then $w \DEF u(\cdot + s)$ is a solution on $[0,t-s]$ of \cref{eq:aie} with $\phi = u(s)$. Hence $t - s \in J_{u(s)}$ so $(t - s, u(s)) \in \mathcal{D}(\Sigma)$, and $\Sigma(t - s, u(s)) = w(t - s) = u(t)$.
  
  Conversely, let us assume that \cref{eq:usigma} holds for all $(t,s) \in \Omega_J$. Continuity of $u$ is not difficult: If $t \in J$ is an interior point, then there exists $s \in J$ with $s < t$ and for $\delta > 0$ small enough,
  \[
    u(t \pm \delta) = \Sigma(t \pm \delta - s, u(s)) \to \Sigma(t - s,u(s)) = u(t)
  \]
  as $\delta \downarrow 0$. The case that $t$ is an endpoint is even simpler. It remains to show that $u$ satisfies \cref{eq:aie-st} on $J$. Let $(t,s) \in \Omega_J$ be arbitrary and let $u_{\phi} : J_{\phi} \to X$ be the maximal solution of \cref{eq:aie} for $\phi = u(s)$. By \cref{eq:usigma} we have $u(t) = u_{\phi}(t - s)$ and this equality together with \cref{eq:aie} then implies that \cref{eq:aie-st} holds.
\end{proof}

\section{Center manifolds in the \texorpdfstring{non-$\odot$-reflexive}{non-sun-reflexive} case}\label{sec:cm}
\emph{All vector spaces in this section are over $\RR$.}

\medskip
Let $\hat{\phi} = 0$ be an equilibrium of the semiflow $\Sigma$ defined by \cref{eq:semiflow} and let the $\mathcal{C}_0$-semigroup $\{T(t)\}_{t \ge 0}$ be the linearization of $\Sigma$ at $\hat{\phi}$ as in \cref{prop:linearization}. In this section we explain how the construction of a local center manifold for $\hat{\phi}$ in \cite[Chapter IX]{Delay1995} can be adapted to the non-$\odot$-reflexive case, using the results obtained in the previous sections.

While we deliberately stay close to the presentation in \cite{Delay1995}, there are also differences. As far as these differences stem from non-$\odot$-reflexivity, they are confined to \cref{sec:Keta}. This illustrates the general principle that, with the results from \cref{sec:admissibility,sec:inhomogeneous,sec:linearization} at hand, one can overcome the lack of $\odot$-reflexivity rather easily by applying the substitution rule $X^{\odot\star} \to X^{\odot\times}$; see also the comments in \cref{sec:conclusion}.

Another difference of some significance, although by itself unrelated to non-$\odot$-reflexivity, is in the treatment of the spectral decomposition in \cref{sec:decomposition,sec:lifting}. Smaller modifications and additions are indicated in the places where they occur.

\subsection{Spectral decompositions of \texorpdfstring{$X$}{X} and \texorpdfstring{$X^{\odot\times}$}{Xsuncross}}\label{sec:decomposition}
In \cite[Sections VIII.2 and IX.2]{Delay1995} the construction of local invariant manifolds for $\hat{\phi}$ starts from assumptions about the existence of a topological direct sum decomposition of $X^{\odot\star}$ into certain positively $\SUNSTAR{T}$-invariant subspaces, and about the behavior of $\{T^{\odot\star}(t)\}_{t \ge 0}$ on those subspaces. There the assumptions are formulated directly in the large space $\SUNSTAR{X}$, motivated by the fact that the nonlinearity does not map $X$ into itself, but rather into $\SUNSTAR{X}$.

Alternatively, one could first formulate these assumptions on $X$ and only then prove that they can be lifted to $\SUNSTAR{X}$, or rather to $X^{\odot\times} \subseteq X^{\odot\star}$, since in the general (non-$\odot$-reflexive) case, the nonlinearity $R$ introduced in \cref{eq:splittingGLR} takes its values in $X^{\odot\times}$. Apart from the substitution by $X^{\odot\times}$, this is along the lines of \cite[Theorems IV.2.11 and IV.2.12]{Delay1995} but with the additional observation that the assumption of eventual compactness of $\{T(t)\}_{t \ge 0}$ included there is not really needed. Indeed, the decomposition of $X$ and the corresponding exponential estimates may \emph{themselves} be taken as the assumption from which the analogous properties in $X^{\odot\times}$ can then be deduced. I see two advantages of this approach:

\begin{enumerate}
\item  
The task of lifting from $X$ to $X^{\odot\times}$ is solved once and for all. There is no need to repeat the procedure for different classes of delay equations with different spectral properties, e.g. first for classical DDEs generating eventually compact $\mathcal{C}_0$-semigroups, next for abstract DDEs or for equations with infinite delay, and so on. We recall from \cref{sec:xsuncross} that $X^{\odot\times} = X^{\odot\star}$ in the $\odot$-reflexive case.
\item
  A direct formulation of the assumptions in $X^{\odot\times}$ obfuscates the fact that the involved subspaces and operators stem from corresponding objects originally defined in or on $X$. Making this relationship more explicit by starting out on $X$ instead of $X^{\odot\times}$ adds clarity.
\end{enumerate}

We therefore begin by reformulating the assumptions from \cite[Section IX.2]{Delay1995} in $X$.

\begin{hypothesis}
  \label{hyp:cm}
  The space $X$ and the $\mathcal{C}_0$-semigroup $\{T(t)\}_{t \ge 0}$ on $X$ have the following properties:
  \begin{hypenum}
  \item
    \label{hyp:cm:sum}
    $X$ admits a direct sum decomposition
    \begin{equation}
      \label{eq:decomposition-of-x}
      X = X_- \oplus X_0 \oplus X_+,
    \end{equation}
    which is topological, i.e. each summand is closed.
  \item
    \label{hyp:cm:invariance}
    The subspaces $X_-$, $X_0$ and $X_+$ are positively $T$-invariant.
  \item
    \label{hyp:cm:group}
    $\{T(t)\}_{t \ge 0}$ can be extended to a $\mathcal{C}_0$-group on $X_0$ and on $X_+$.
  \item
    \label{hyp:cm:trichotomy}
    The decomposition \cref{eq:decomposition-of-x} is an \textbf{exponential trichotomy} on $\RR$, meaning that there exist $a < 0 < b$ such that for every $\EPS > 0$ there exists $K_{\EPS} > 0$ such that
    \begin{subequations}
      \label{eq:trichotomy}
      \begin{alignat}{2}
        \|T(t)\phi\| &\le K_{\EPS} e^{at}\|\phi\| && \qquad \ALL\,t \ge 0 \text{ and } \phi \in X_-,\label{eq:trichotomy:stable}\\
        \|T(t)\phi\| &\le K_{\EPS} e^{\EPS|t|}\|\phi\| && \qquad \ALL\,t \in \RR \text{ and } \phi \in X_0, \label{eq:trichotomy:center}\\
        \|T(t)\phi\| &\le K_{\EPS} e^{bt}\|\phi\| && \qquad \ALL\,t \le 0 \text{ and } \phi \in X_+.\label{eq:trichotomy:unstable}      
      \end{alignat}
    \end{subequations}
  \end{hypenum}
  We call $X_-$, $X_0$ and $X_+$ the \textbf{stable subspace}, \textbf{center subspace} and \textbf{unstable subspace}.
\end{hypothesis}

In \cref{sec:lifting} it is explained in some detail how these assumptions induce a decomposition of $X^{\odot\times}$ with identical properties. The end result can be found in \cref{prop:xsuncross-cm}.

\begin{hypothesis}
  \label{hyp:cm:j}
  The subspaces $[X^{\odot\times}]_0$ and $[X^{\odot\times}]_+$ are contained in $jX$.
\end{hypothesis}

In concrete cases, the hypotheses \cref{hyp:cm} and \cref{hyp:cm:j} are usually verified by decomposition of the spectrum of the generator of the complexification of $\{T(t)\}_{t \ge 0}$. The following result in this spirit applies to a reasonably large class of $\mathcal{C}_0$-semigroups. Its proof can be found in \cref{sec:real-case}, while an application is given in \cref{thm:cm-dde} in \cref{sec:ddes}.

\begin{restatable}{theorem}{hypcmthm}
  \label{thm:hypcm}
  Suppose $\{T(t)\}_{t \ge 0}$ is eventually norm continuous and let $A_{\cc}$ be the complexification of its generator. If $\sigma(A_{\cc})$ is the pairwise disjoint union of the nonempty sets
  \begin{alignat*}{2}
    &\sigma_- &&\DEF \{\lambda \in \sigma(A_{\cc}) \,:\, \RE{\lambda} < 0\},\\
    &\sigma_0 &&\DEF \{\lambda \in \sigma(A_{\cc}) \,:\, \RE{\lambda} = 0\},\\
    &\sigma_+ &&\DEF \{\lambda \in \sigma(A_{\cc}) \,:\, \RE{\lambda} > 0\},
  \end{alignat*}
  where $\sigma_-$ is closed and both $\sigma_0$ and $\sigma_+$ are compact, and if
  \[
    \gamma_- \DEF \sup_{\lambda \in \sigma_-}\RE{\lambda} < 0 < \inf_{\lambda \in \sigma_+}\RE{\lambda} \FED \gamma_+,
  \]
  then \cref{hyp:cm} and \cref{hyp:cm:j} hold for $\{T(t)\}_{t \ge 0}$ on $X$.
\end{restatable}

Let $E$ be a Banach space and let $\eta \in \RR$. From \cite[Definitions VIII.2.5 and IX.2.2]{Delay1995} we recall the three Banach spaces
\begin{align*}
  \BC^{\eta}(\RR_{\pm},E) &\DEF \{f \in C(\RR_{\pm},E) \,:\, \sup_{t \in \RR_{\pm}}{e^{-\eta t}\|f(t)\|} < \infty \},\\
  \BC^{\eta}(\RR,E) &\DEF \{f \in C(\RR,E)\,:\, \sup_{t \in \RR}{e^{-\eta |t|}\|f(t)\|} < \infty \},
\end{align*}
equipped with their respective weighted supremum norms. If $\eta = 0$ we will suppress the superscript $\eta$.

Let $J$ be an interval. Analogously to \cref{eq:aie-st}, a \textbf{solution on $J$} of the linear homogeneous equation
\begin{equation}
  \label{eq:hom-st}
  u(t) = T(t - s)u(s), \qquad (t,s) \in \Omega_J,
\end{equation}
is defined as a continuous function $u : J \to X$ such that \cref{eq:hom-st} holds for all $(t,s) \in \Omega_J$. 

\begin{lemma}[\normalfont{\cite[Lemma IX.2.4, proof as exercise]{Delay1995}}]
  \label{lem:XsigmaBC}
  Let $\eta \in (0,\min\{-a,b\})$. Then
  \begin{align*}
    X_- &= \{\phi \in X \, :\, \text{there exists a solution of \cref{eq:hom-st} on $\RR_+$ through $\phi$ which belongs to $\BC^a(\RR_+,X)$}\},\\
    X_0 &= \{\phi \in X \, :\, \text{there exists a solution of \cref{eq:hom-st} on $\RR$ through $\phi$ which belongs to $\BC^{\eta}(\RR,X)$}\},\\
    X_+ &= \{\phi \in X \, :\, \text{there exists a solution of \cref{eq:hom-st} on $\RR_-$ through $\phi$ which belongs to $\BC^b(\RR_-,X)$}\}.
  \end{align*}
\end{lemma}
\begin{proof}
  We give the proof for $X_0$ only. If $\phi \in X_0$ then $u_{\phi} : \RR \to X$ defined by $u_{\phi}(t) \DEF T(t)\phi$ is a solution of \cref{eq:hom-st} through $\phi$. Choose $0 < \EPS \le \eta$ and $K_{\EPS}$ according to \cref{hyp:cm:trichotomy}, then
  \[
    e^{-\eta|t|} \|u_{\phi}(t)\| \le K_{\EPS} e^{-(\eta - \EPS)|t|}\|\phi\| \le K_{\EPS}\|\phi\| \qquad \ALL\,t \in \RR,
  \]
  so $u_{\phi} \in \BC^{\eta}(\RR,X)$. Conversely, suppose that $\phi \in X$ is such that there exist a solution $u_{\phi} \in \BC^{\eta}(\RR,X)$ of \cref{eq:hom-st} and a time $t_0 \in \RR$ such that $u_{\phi}(t_0) = \phi$. We will show that $P_{\pm}\phi = 0$. By \cref{hyp:cm:group},
  \[
    P_+\phi = T(t_0 - t)T(t - t_0)P_+\phi \qquad \ALL\,t \ge t_0,
  \]
  so \cref{eq:trichotomy:unstable} implies, for all $t \ge t_0$,
  \begin{align*}
    \|P_+\phi\| &\le K_{\EPS} e^{b(t_0 - t)}\|T(t - t_0)P_+\phi\|\\
                &\le K_{\EPS} e^{b(t_0 - t)}\|u_{\phi}(t)\|,
  \end{align*}
  and therefore, for $t \ge \max\{t_0,0\}$,
  \[
    e^{-\eta t}\|u_{\phi}(t)\| \ge \frac{e^{-b t_0}}{K_{\EPS}}e^{(b - \eta)t}\|P_+\phi\| \to \infty \qquad \text{as } t \to \infty,
  \]
  unless $P_+\phi = 0$. Similarly, since $u_{\phi}$ is a solution of \cref{eq:hom-st} on $\RR$ and $u_{\phi}(t_0) = \phi$,
  \[
    \begin{aligned}
      P_-\phi &= P_-T(t_0 - t)u_{\phi}(t)\\
      &= T(t_0 - t)P_-u_{\phi}(t)
    \end{aligned} \qquad \ALL\,t \le t_0,
  \]
  so \cref{eq:trichotomy:stable} implies
  \[
    \|P_-\phi\| \le K_{\EPS} e^{a(t_0 - t)}\|u_{\phi}(t)\| \qquad \ALL\,t \le t_0,
  \]
  and therefore, for $t \le \min\{t_0,0\}$,
  \[
    e^{-\eta|t|}\|u_{\phi}(t)\| \ge \frac{e^{-a t_0}}{K_{\EPS}} e^{(a + \eta)t} \|P_-\phi\| \to \infty \qquad \text{as } t \to -\infty,
  \]
  unless $P_-\phi = 0$. We have shown that $P_{\pm}\phi = 0$, so $\phi \in X_0$.
\begin{mycomment}
  \medskip
  \par
  Following is also the proof for the stable subspace. If $\phi \in X_-$, then $u_{\phi} : \RR_+ \to X$ defined by $u_{\phi}(t) \DEF T(t)\phi$ is clearly a solution of \cref{eq:hom-st} on $\RR_+$. By \cref{eq:trichotomy:stable} we have
    \[
      e^{-a t}\|u_{\phi}(t)\| \le K_{\EPS} \|\phi\| \qquad \ALL\,t \ge 0,
    \]
    so $u_{\phi} \in \BC^a(\RR_+,E)$. Conversely, let $\phi \in X$ and suppose there exists a solution $u_{\phi}$ of \cref{eq:hom-st} on $\RR_+$ that belongs to $\BC^a(\RR_+,X)$ and such that $u_{\phi}(t_0) = \phi$ for some $t_0 \ge 0$. We may assume that $t_0 = 0$, because otherwise we can consider the translated solution $v_{\phi}(\cdot) \DEF T(t_0)u_{\phi}(\cdot)$. Then, by \cref{hyp:cm:group},
    \[
      (I - P_-)\phi = T(-t)T(t)(I - P_-)\phi \qquad \ALL\,t \ge 0,
    \]
    and since $I - P_- = P_0 + P_+$, it follows, for $t \ge 0$,
    \begin{align*}
      \|(I - P_-)\phi\| &\le \|T(-t)T(t)P_0\phi\| + \|T(-t)T(t)P_+\phi\|\\
                        &\le K_{\EPS} e^{\EPS t}\|T(t)P_0\phi\| + K_{\EPS} e^{-b t}\|T(t)P_+\phi\|\\
                        &\le K_{\EPS} (e^{\EPS t} + e^{-bt})\|u_{\phi}(t)\|,
    \end{align*}
    where the second line is due to \cref{hyp:cm:trichotomy}. Now choose $\EPS = \frac{|a|}{2}$ and $K_{\EPS}$ according to \cref{hyp:cm:trichotomy}, then
    \[
      e^{-a t} \|u_{\phi}(t)\| \ge \frac{e^{\EPS t}}{K_{\EPS}(1 + e^{-(b +\EPS)t})}\|(I - P_-)\phi\| \to \infty \qquad \text{as } t \to \infty,
    \]
    unless $(I - P_-)\phi = 0$, which shows that $\phi \in X_-$.
\end{mycomment}
\end{proof}

\subsection{Bounded solutions of the linear inhomogeneous equation}\label{sec:Keta}
Let $J$ be an interval. Analogously to \cref{eq:aie-st,eq:hom-st}, a \textbf{solution on $J$} of the linear \emph{in}homogeneous equation
\begin{equation}
  \label{eq:inhom-st}
  u(t) = T(t - s)u(s) + j^{-1}\int_s^t T^{\odot\star}(t - \tau)f(\tau)\,d\tau, \qquad (t,s) \in \Omega_J,
\end{equation}
is defined to be a continuous function $u : J \to X$ such that \cref{eq:inhom-st} holds for all $(t,s) \in \Omega_J$. In the analytic (as distinguished from: geometric) construction of a local center manifold, a key step is the introduction of a linear operator that associates with each appropriate forcing function a solution of \cref{eq:inhom-st} on $\RR$ with prescribed behavior both at $t = 0$ and $t = \pm \infty$. For the $\odot$-reflexive case this is done in \cite{Delay1995}. The results from \cref{sec:admissibility,sec:inhomogeneous} can be used to obtain an extension to the non-$\odot$-reflexive setting with relatively little effort. Define, for any $\eta \in (0,\min\{-a,b\})$,
\begin{align*}
  \mathcal{K}_{\eta} : \BC^{\eta}(\RR,X^{\odot\times}) \to \BC^{\eta}(\RR,X), \qquad (\mathcal{K}_{\eta}f)(t) &\DEF j^{-1}\int_0^t T^{\odot\star}(t - \tau)P_0^{\odot\times}f(\tau)\,d\tau\\
                                                                                                              &+ j^{-1}\int_{\infty}^t T^{\odot\star}(t - \tau)P_+^{\odot\times}f(\tau)\,d\tau\\
                                                                                                              &+ j^{-1}\int_{-\infty}^t T^{\odot\star}(t - \tau)P_-^{\odot\times}f(\tau)\,d\tau.
\end{align*}
Regarding the following lemma, we would like to stress that the only fundamental difference between the proof for the $\odot$-reflexive case in \cite{Delay1995} and the proof here lies in the verification that the last \WS-integral in the above definition of $\mathcal{K}_{\eta}$ indeed takes values in $jX$; also see \cref{eq:Imin-lim}. Still, we have included some more details because the proof in \cite{Delay1995} is rather condensed and part of the estimates were left as exercises.

\begin{proposition}[\normalfont{cf. \cite[Lemma IX.3.2 and Exercise 3.4]{Delay1995}}]\label{prop:Keta}
  Let $\eta \in (0,\min\{-a,b\})$.
  \begin{propenum}
  \item
    $\mathcal{K}_{\eta}$ is a well-defined bounded linear operator.
  \item
    \label{prop:Keta:solution}
    $\mathcal{K}_{\eta}f$ is the unique solution of \cref{eq:inhom-st} in $\BC^{\eta}(\RR,X)$ with zero $X_0$-component at $t = 0$.
  \end{propenum}
\end{proposition}
\begin{proof}
  Let $\EPS < \eta$ and select $K_{\EPS} > 0$ according to \cref{prop:xsuncross-cm:trichotomy}. 
  \begin{steps}[label=\roman*.]
  \item
    For any given $f \in \BC^{\eta}(\RR,X^{\odot\times})$, the three integrals in the definition of $\mathcal{K}_{\eta}f$ naturally define functions $I_{\mu} : \RR \to X^{\odot\star}$ for $\mu \in \{0,+,-\}$. We begin by checking that these functions are well-defined, continuous, take values in $jX$ and satisfy certain estimates.
    
    \begin{description}[wide,labelindent=0pt,itemsep=0.5em]
    \item[$I_0$:]
      This is the simplest case, because the integration domain is compact. \Cref{hyp:cm:j,lem:Pj:3} imply, for all $t \in \RR$,
      \begin{align*}
        I_0(t) &= \int_0^t T^{\odot\star}(t - \tau)j j^{-1} P_0^{\odot\times}f(\tau)\,d\tau\\
               &= j \int_0^t T(t - \tau) j^{-1} P_0^{\odot\times}f(\tau)\,d\tau,
      \end{align*}
      where the second integral is a Riemann-integral. We see that $I_0$ takes values in $jX$. It is straightforward to obtain the estimate
      \begin{equation}
        \label{eq:I0-est}
        \|I_0(t)\| \le \frac{K_{\EPS}}{\eta - \EPS}e^{\eta|t|}\|f\|_{\eta} \qquad \ALL\,t \in \RR
      \end{equation}
      by using \cref{prop:xsuncross-cm:trichotomy}, distinguishing between the cases $t \ge 0$ and $t \le 0$.
    \item[$I_+$:]
      \Cref{hyp:cm:j,lem:Pj:3} imply that, for any $t \in \RR$,
      \[
        T^{\odot\star}(t - \tau)P_+^{\odot\times}f(\tau)=  jT(t - \tau)j^{-1}P_+^{\odot\times}f(\tau) \qquad \ALL\,\tau \ge t,
      \]
      so the integrand in $I_+(t)$ is continuous on $[t,\infty)$. \Cref{prop:xsuncross-cm:trichotomy} implies the estimate
      \[
        \|T^{\odot\star}(t - \tau)P_+^{\odot\times}f(\tau)\| \le K_{\EPS} e^{bt} e^{-b\tau + \eta|\tau|} \|f\|_{\eta} \qquad \ALL\,\tau \ge t.
      \]
      Hence the \WS-integral defining $I_+(t)$ exists, and it can be evaluated as a Bochner integral over $[t,\infty)$,
      \[
        I_+(t) = -j\int_t^{\infty} T(t - \tau)j^{-1}P_+^{\odot\times}f(\tau) \,d\tau,
      \]
      showing that $I_+$ takes values in $jX$. From the above estimate it follows that
      \[
       \|I_+(t)\| \le K_{\EPS} e^{bt} \|f\|_{\eta} \int_t^{\infty} e^{-b\tau + \eta|\tau|}\,d\tau.
     \]
     For the integral in the right-hand side,
     \[
       \int_t^{\infty}{e^{-b\tau + \eta|\tau|}\,d\tau} =
       \begin{cases}
         \frac{e^{-(b - \eta)t}}{b - \eta} &\text{if } t \ge 0,\\
         \frac{e^{-(b + \eta)t}}{b + \eta} - \frac{1}{b + \eta} + \frac{1}{b - \eta} &\text{if } t \le 0.
       \end{cases}
     \]
     Now, for $\alpha \ge 1$ we have the inequality 
     \[
       \frac{\alpha}{b + \eta} + \frac{1}{b - \eta} \le \frac{\alpha}{b - \eta} + \frac{1}{b + \eta}.
     \]
     (Subtract the left-hand side from the right-hand side and find that the result is non-negative.) If $t \le 0$ then $e^{-(b + \eta)t} \ge 1$, so in this case
     \[
       \int_t^{\infty}{e^{-b\tau + \eta|\tau|}\,d\tau} \le \frac{e^{-(b + \eta)t}}{b - \eta} \qquad \text{if } t \le 0.
     \]
     In summary, we have obtained the estimate
     \begin{equation}
       \label{eq:exp-int-est}
       \int_t^{\infty}{e^{-b\tau + \eta|\tau|}\,d\tau} \le \frac{ e^{-bt}e^{\eta|t|}}{b - \eta} \qquad \ALL\,t \in \RR,
     \end{equation}
     and therefore
     \begin{equation}
       \label{eq:Iplus-est}
       \|I_+(t)\| \le K_{\EPS}\|f\|_{\eta}\frac{e^{\eta|t|}}{b - \eta} \qquad \ALL\,t\in \RR.
     \end{equation}
   \item[$I_-$:] For any $t \in \RR$, the integrand in $I_-(t)$ is \WS-continuous, hence $\WS$-Lebesgue measurable, and by \cref{prop:xsuncross-cm:trichotomy} it satisfies the estimate
     \[
       \|T^{\odot\star}(t - \tau)P_-^{\odot\times}f(\tau)\| \le K_{\EPS} e^{a(t - \tau) + \eta|\tau|}\|f\|_{\eta} \qquad \ALL\,\tau \le t,
     \]
     so the \WS-integral defining $I_-(t)$ exists. \Cref{lem:wsnormconv} implies that
     \begin{equation}
       \label{eq:Imin-lim}
       I_-(t) = \lim_{n \to \infty} \int_{t-n}^t T^{\odot\star}(t - \tau)P_-^{\odot\times} f(\tau)\,d\tau,
     \end{equation}
     in norm. The function $P_-^{\odot\times} \circ f$ takes values in $X^{\odot\times}$, which is an admissible range for $T$, so each integral inside the limit is an element of $jX$. The norm convergence implies that the same is true for $I_-(t)$. From the above estimate it follows that
     \[
       \|I_-(t)\| \le K_{\EPS} e^{at} \|f\|_{\eta} \int_{-\infty}^t e^{-a\tau + \eta|\tau|}\,d\tau
     \]
     The substitution $-\tau = s$ inside the integral enables an application of \cref{eq:exp-int-est}, which yields
     \[
       \int_{-\infty}^t{e^{-a\tau + \eta|\tau|}\,d\tau} = \int_{-t}^{\infty}{e^{-(-a)s + \eta|s|}\,ds} \le \frac{e^{-(-a)(-t)}e^{\eta|t|}}{-a - \eta} = \frac{e^{-at}e^{\eta|t|}}{-a - \eta},
     \]
     and therefore
     \begin{equation}
       \label{eq:Imin-est}
       \|I_-(t)\| \le K_{\EPS}\|f\|_{\eta} \frac{e^{\eta|t|}}{-a - \eta} \qquad \ALL t \in \RR.
     \end{equation}
   \end{description}
   
   Combining the estimates \cref{eq:I0-est,eq:Iplus-est,eq:Imin-est}, we obtain
   \[
     e^{-\eta|t|}\|(\mathcal{K}_{\eta}f)(t)\| \le \|j^{-1}\| K_{\EPS} \|f\|_{\eta} \left(\frac{1}{\eta - \EPS} + \frac{1}{b - \eta} + \frac{1}{-a - \eta}\right) \qquad \ALL\,t \in \RR,
   \]
   so $\mathcal{K}_{\eta}f \in \BC^{\eta}(\RR,X)$ and $\mathcal{K}_{\eta}$ is a bounded linear operator.
 \item
   Given any $f \in \BC^{\eta}(\RR,X^{\odot\times})$, it is straightforward to check that $\mathcal{K}_{\eta}f$ is a solution of \cref{eq:inhom-st}.

   By \cref{prop:xsuncross-cm:group,prop:xsuncross-cm:invariance} the subspaces $[X^{\odot\times}]_0$ and $[X^{\odot\times}]_+$ are invariant for the group $\{T^{\odot\star}(t)\}_{t \in \RR}$ and the subspace $[X^{\odot\times}]_-$ is positively invariant for the semigroup $\{T^{\odot\star}(t)\}_{t \ge 0}$. So, for every $\mu \in \{0,+,-\}$ the projectors inside $I_{\mu}(t)$ commute with the (semi)group operators, which gives $I_{\mu}(t) = P_{\mu}^{\odot\star}I_{\mu}(t)$ for all $t \in \RR$. \Cref{lem:Pj:1} then implies that $j^{-1} \circ I_{\mu}$ maps into $X_{\mu}$. Since $I_0(0) = 0$ it follows that $(\mathcal{K}_{\eta}f)(0)$ has a vanishing component in $X_0$.

   For uniqueness, suppose that $v \in \BC^{\eta}(\RR,X)$ is another solution of \cref{eq:inhom-st} such that $P_0v(0) = 0$. Then $w \DEF \mathcal{K}_{\eta}f - v$ is a solution of \cref{eq:hom-st} in $\BC^{\eta}(\RR,X)$ through $w(t)$ for all $t \in \RR$, so \cref{lem:XsigmaBC} shows that $w$ takes values in $X_0$ and, in particular, $w(0) = P_0w(0) = 0$. \Cref{hyp:cm:group} implies that $w = 0$ identically. \qedhere
 \end{steps}
\end{proof}

\subsection{Modification of the nonlinearity}\label{sec:Rdelta}
The next step in the construction of a local center manifold is a modification of the nonlinearity $R$ introduced in the splitting \cref{eq:splittingGLR} in \cref{sec:linearization}. We recall that $R : X \to X^{\odot\times}$ is a $C^k$-smooth operator for some $k \ge 1$, and
\begin{equation}
  \label{eq:R0DR0}
  R(0) = 0, \qquad DR(0) = 0.
\end{equation}
For any $\delta > 0$, let $R_{\delta} : X \to X^{\odot\times}$ be its $\delta$-modification as defined in \cite[Section IX.4]{Delay1995}. The purpose of this modification is to obtain a nonlinearity that is globally Lipschitz (which will ensure that the corresponding substitution operators $\tilde{R}_{\delta}$ in \cref{eq:substitution} below are well-defined) with a Lipschitz constant controlled by $\delta$ (which will ensure the contractivity of the parameterized fixed point operator that defines the local center manifold.)

However, it is not clear to me how \cite[Lemma IX.4.1]{Delay1995} applies to obtain the aforementioned properties of $R_{\delta}$ in \cite[Corollary IX.4.2]{Delay1995}. First, if $X$ is infinite-dimensional, then the $C^k$-smoothness of $R$ does not imply Lipschitz continuity on balls of arbitrary radius. \emph{Locally} this is of course not a problem:

\begin{lemma}\label{lem:lipconst}
  There exist $\delta_1 > 0$ and $L : [0,\delta_1] \to \RR_+$ such that $L(0) = 0$, $L(\delta)$ is a Lipschitz constant for $R$ on the open ball $B_{\delta} \subseteq X$ for every $0 < \delta \le \delta_1$ and $L$ is continuous at zero.
\end{lemma}
\begin{proof}
  By continuity of $DR$ and \cref{eq:R0DR0} there exists $\delta_1 > 0$ such that $\sup\left\{\|DR(w)\|\,:\, w \in B_{\delta_1}\right\} \le 1$. Define $L(0) \DEF 0$ and, for any $0 < \delta \le \delta_1$,
  \[
    L(\delta) \DEF \sup\left\{\|DR(w)\|\,:\, w \in B_{\delta} \right\}.
  \]
  By the mean value inequality $L(\delta)$ is a Lipschitz constant for $R$ on $B_{\delta}$. Moreover, given $\EPS > 0$ there exists $0 < \delta_{\EPS} \le \delta_1$ such that $\sup\left\{\|DR(w)\|\,:\, w \in B_{\delta_{\EPS}} \right\} \le \EPS$, and if $\delta \le \delta_{\EPS}$ then $L(\delta)$ does not exceed the left-hand side of this inequality. 
\end{proof}

Second, it is unclear to me that $R_{\delta}$ has the appropriate functional form to make \cite[Lemma IX.4.1]{Delay1995} applicable, so here is a minor adaptation that uses the above lemma.

\begin{proposition}[\normalfont{cf. \cite[Lemma IX.4.1 and Corollary IX.4.2]{Delay1995}}]\label{prop:Rdelta-Lip}
  For $\delta > 0$ sufficiently small, the modified nonlinearity $R_{\delta}$ is globally Lipschitz continuous with Lipschitz constant $L_{R_{\delta}} \to 0$ as $\delta \downarrow 0$.
\end{proposition}
\begin{proof}
  Let $\xi : \RR_+ \to [0,1]$ be a standard cut-off function as introduced at the start of \cite[Section IX.4]{Delay1995} and define the auxiliary functions $\xi_{\delta} : X \to [0,1]$ and  $\Xi_{\delta} : X \to [0,1]$ by
  \[
    \xi_{\delta}(x) \DEF \xi\Bigl(\frac{\|x\|}{\delta}\Bigr), \qquad \Xi_{\delta}(x) \DEF \xi_{\delta}(P_0x)\xi_{\delta}((I - P_0)x), \qquad \delta > 0.
  \]
  In terms of these functions, we have $R_{\delta}(x) = R(x)\Xi_{\delta}(x)$ for all $x \in X$.
  \begin{steps}
  \item
    Let $C > 0$ be a global Lipschitz constant for $\xi$. Using that the composition of two Lipschitz functions is Lipschitz with constant equal to the product of the two Lipschitz constants, we obtain that $\xi_{\delta}$ has the global Lipschitz constant $\frac{C}{\delta}$. This implies the global Lipschitz estimate
    \begin{align*}
      \left| \Xi_{\delta}(x) - \Xi_{\delta}(y) \right| &\le \left|\xi_{\delta}(P_0x) - \xi_{\delta}(P_0y)\right| + \left|\xi_{\delta}((I - P_0)x) - \xi_{\delta}((I - P_0)y)\right|\\
                                                       &\le \frac{C}{\delta}\|x - y\| + \frac{C}{\delta}\|x - y\| \lesssim \frac{C}{\delta}\|x - y\|,
    \end{align*}
    for all $x, y \in X$. (The numerical factor was absorbed into $C$.) We furthermore note that
    \[
      \|x\| = \|P_0x + (I - P_0)x\| \le \|P_0x\| + \|(I - P_0)x\| \qquad \ALL\,x \in X,
    \] 
    so if $\|x\| \ge 4\delta$ then $\max\{\|P_0x\|, \|(I - P_0)x\|\} \ge 2\delta$ and consequently $\Xi_{\delta}(x) = 0$.
  \item
    Let $\delta > 0$ be such that $4\delta \le \delta_1$ with $\delta_1$ as in \cref{lem:lipconst}. For any $x, y \in X$ we estimate
    \begin{align*}
      \|R_{\delta}(x) - R_{\delta}(y)\| &\le \|R(x) - R(y)\| \cdot \Xi_{\delta}(y) + |\Xi_{\delta}(x) - \Xi_{\delta}(y)| \cdot \|R(x)\|\\
                                        &\le
                                          \begin{cases}
                                            L(4\delta)\|x - y\| + \frac{C}{\delta}\|x - y\|L(4\delta)4\delta& \text{if } \|x\|, \|y\| < 4\delta,\\
                                            0 & \text{if } \|x\|, \|y\| \ge 4\delta,\\
                                            \frac{C}{\delta}\|x - y\|L(4\delta)4\delta    & \text{if } \|x\| < 4\delta, \|y\| \ge 4\delta,
                                          \end{cases}\\
                                        &\le L(4\delta)(4C + 1)\|x - y\|,
    \end{align*}
    so $L_{R_{\delta}} = L(4\delta)(4C + 1)$ is a global Lipschitz constant for $R_{\delta}$ and $L_{R_\delta} \to 0$ as $\delta \downarrow 0$. \hfill \qedhere
  \end{steps}
\end{proof}

In the proof it was also obtained that for $\delta > 0$ sufficiently small, $\|x\| \ge 4\delta$ implies $\Xi_{\delta}(x) = 0$. Together with \cref{prop:Rdelta-Lip} itself, this gives
\begin{equation}
  \label{eq:Rdelta-estimate}
  \|R_{\delta}(x)\| \le 4\delta L_{R_{\delta}} \qquad \ALL\,x \in X.
\end{equation}
We associate with $R_{\delta}$ the \textbf{substitution operator}
\begin{equation}\label{eq:substitution}
  \tilde{R}_{\delta} : \BC^{\eta}(\RR,X) \to \BC^{\eta}(\RR,X^{\odot\times}), \qquad \tilde{R}_{\delta}(u) \DEF R_{\delta} \circ u.
\end{equation}

\begin{corollary}\label{cor:substitution}
For all $\delta > 0$ sufficiently small, $\tilde{R}_{\delta}$ is well-defined and globally Lipschitz with Lipschitz constant $L_{R_{\delta}} \to 0$ as $\delta \downarrow 0$.
\end{corollary}
\begin{proof}
  $\tilde{R}_{\delta}$ is well-defined due to \cref{eq:Rdelta-estimate} and, for any $u,v \in \BC^{\eta}(\RR,X)$,
  \begin{align*}
    \|\tilde{R}_{\delta}(u) - \tilde{R}_{\delta}(v)\|_{\eta} &= \sup_{t \in \RR} e^{-\eta|t|} \|R_{\delta}(u(t)) - R_{\delta}(v(t))\|\\
                                                             &\le L_{R_{\delta}} \sup_{t \in \RR} e^{-\eta|t|} \|u(t) - v(t)\| = L_{R_{\delta}} \|u - v\|_{\eta},
  \end{align*}
  so $\tilde{R}_{\delta}$ inherits the global Lipschitz constant from $R_{\delta}$.
\end{proof}

\subsection{The fixed-point operator and the center manifold}\label{sec:cm-fixed-point}
Motivated by the characterization of $X_0$ provided by \cref{lem:XsigmaBC}, we will define a parameterized fixed point operator, in such a way that its fixed points correspond to exponentially bounded solutions on $\RR$ of the modified equation
\begin{equation}
  \label{eq:aie-st-delta}
  u(t) = T(t-s)u(s) + j^{-1}\int_s^t{\SUNSTAR{T}(t - \tau)R_{\delta}(u(\tau))\,d\tau}, \qquad -\infty < s \le t < \infty.
\end{equation}
(This equation is obtained from \cref{eq:aie-st} by replacing $R$ with the modified nonlinearity $R_{\delta}$ from \cref{sec:Rdelta}.) Let $\eta$ and $\mathcal{K}_{\eta}$ be as in \cref{prop:Keta} and let $\tilde{R}_{\delta}$ be as in \cref{cor:substitution}. Define the fixed point operator
\[
  \mathcal{G} : \BC^{\eta}(\RR,X) \times X_0 \to \BC^{\eta}(\RR,X), \qquad \mathcal{G}(u,\phi) \DEF T(\cdot)\phi + \mathcal{K}_{\eta}\tilde{R}_{\delta}(u),
\]
where its second argument in $X_0$ is regarded as a parameter.

\begin{theorem}[\normalfont{cf. \cite[Theorem IX.5.1]{Delay1995}}]\label{thm:cm}
  If $\eta \in (0, \min\{-a,b\})$ and if $\delta > 0$ is sufficiently small, then the following statements hold.
  \begin{thmenum}
  \item
    For every $\phi \in X_0$ the equation $\mathcal{G}(u,\phi) = u$ has a unique solution $u = u^{\star}(\phi)$.
  \item
    \label{thm:cm:lipschitz}
    The map $u^{\star} : X_0 \to \BC^{\eta}(\RR,X)$ is globally Lipschitz and $u^{\star}(0) = 0$.
  \end{thmenum}
\end{theorem}
\begin{proof}
  Let $u,v \in \BC^{\eta}(\RR,X)$ and $\phi, \psi \in X_0$ be arbitrary. Then
  \begin{align*}
    \|\mathcal{G}(u,\phi) - \mathcal{G}(v,\psi)\|_{\eta} &\le \sup_{t \in \RR} e^{-\eta|t|}\|T(t)(\phi - \psi)\| + \|\mathcal{K}_{\eta}\| L_{R_{\delta}} \|u - v\|_{\eta}\\
                                                         &\le K_{\EPS} \sup_{t \in \RR} e^{-(\eta - \EPS)|t|} \|\phi - \psi\| + \|\mathcal{K}_{\eta}\| L_{R_{\delta}} \|u - v\|_{\eta}\\
                                                         &\le K_{\EPS} \|\phi - \psi\| + \|\mathcal{K}_{\eta}\| L_{R_{\delta}} \|u - v\|_{\eta},
  \end{align*}
  where in the second line we used \cref{hyp:cm:trichotomy} with $\EPS < \eta$. By \cref{cor:substitution} there exists $\delta_2 > 0$ such that $\|\mathcal{K}_{\eta}\| L_{R_{\delta}} \le \frac{1}{2}$ for all $0 < \delta \le \delta_2$ sufficiently small, and we select $\delta$ accordingly.
  \begin{steps}[label=\roman*.]
  \item
    If $\psi = \phi$ then the above estimate shows that $\mathcal{G}(\cdot,\phi)$ is globally Lipschitz with Lipschitz constant $\frac{1}{2}$, so by the contraction mapping principle $\mathcal{G}(\cdot,\phi)$ has a unique fixed point $u^{\star}(\phi)$.
  \item
    Let $u^{\star}(\phi)$ and $u^{\star}(\psi)$ be the unique fixed points of $\mathcal{G}(\cdot,\phi)$ and $\mathcal{G}(\cdot,\psi)$, respectively. Then
    \begin{align*}
      \|u^{\star}(\phi) - u^{\star}(\psi)\|_{\eta} &= \|\mathcal{G}(u^{\star}(\phi),\phi) - \mathcal{G}(u^{\star}(\psi),\psi)\|_{\eta}\\
      &\le K_{\EPS}\|\phi - \psi\| + \frac{1}{2}\|u^{\star}(\phi) - u^{\star}(\psi)\|_{\eta},
    \end{align*}
    so $\|u^{\star}(\phi) - u^{\star}(\psi)\|_{\eta} \le 2K_{\EPS}\|\phi - \psi\|$. It is clear that $u^{\star}(0) = 0$. \qedhere
  \end{steps}
\end{proof}

Let $u^{\star} : X_0 \to \BC^{\eta}(\RR,X)$ be the parameterized fixed point from \cref{thm:cm:lipschitz}. Clearly the map $\mathcal{C}$ in the definition below inherits the global Lipschitz continuity from $u^{\star}$.

\begin{definition}\label{def:CM}
  A \textbf{\emph{global} center manifold} $\CM$ for \cref{eq:aie-st-delta} is defined as the image of the mapping
  \begin{equation}
    \label{eq:cmmap}
    \mathcal{C} : X_0 \to X, \qquad \mathcal{C} \DEF \operatorname{ev} \circ u^{\star},
  \end{equation}
  where $\operatorname{ev} : \BC^{\eta}(\RR,X) \to X$ is the evaluation at zero. 
\end{definition}

As alluded to above, $\CM$ is a nonlinear generalization of the center subspace $X_0$ that was characterized in \cref{lem:XsigmaBC}. This statement can be made more precise.

\begin{proposition}\label{prop:CMcharacterization}
  It holds that
  \[
    \CM = \{\psi \in X \, :\, \text{there exists a solution of \cref{eq:aie-st-delta} on $\RR$ through $\psi$ which belongs to $\BC^{\eta}(\RR,X)$}\}.
  \]
\end{proposition}
\begin{proof}
  Let $\psi \in \CM$ be arbitrary, so $\psi = \mathcal{C}(\phi) = u^{\star}(\phi)(0)$ for some $\phi \in X_0$. We prove that $u = u^{\star}(\phi)$ is a solution of \cref{eq:aie-st-delta} in $\BC^{\eta}(\RR,X)$. \Cref{prop:Keta:solution} shows that $\mathcal{K}_{\eta}\tilde{R}_{\delta}(u)$ is a solution of \cref{eq:inhom-st} with $f = \tilde{R}_{\delta}(u)$. It follows that
  \begin{align*}
    u(t) &= T(t)\phi + (\mathcal{K}_{\eta}\tilde{R}_{\delta}(u))(t)\\
         &= T(t)\phi + T(t - s)(\mathcal{K}_{\eta}\tilde{R}_{\delta}(u))(s) + j^{-1}\int_s^t{\SUNSTAR{T}(t - \tau)R_{\delta}(u(\tau))\,d\tau}\\
         &= T(t)\phi + T(t - s)\left(u(s) - T(s)\phi\right) + j^{-1}\int_s^t{\SUNSTAR{T}(t - \tau)R_{\delta}(u(\tau))\,d\tau}\\
         &= T(t - s)u(s) + j^{-1}\int_s^t{\SUNSTAR{T}(t - \tau)R_{\delta}(u(\tau))\,d\tau}
  \end{align*}
  for all $(t,s) \in \Omega_{\RR}$.

  Conversely, suppose that $\psi \in X$ is such that there exist a solution $u$ in $\BC^{\eta}(\RR,X)$ of \cref{eq:aie-st-delta} and a time $t_0 \in \RR$ such that $u(t_0) = \psi$. We may assume that $t_0 = 0$, since by translation invariance $u_0 \DEF u(\cdot + t_0)$ is a solution in $BC^{\eta}(\RR,X)$ that satisfies $u_0(0) = \psi$. For $(t,s) \in \Omega_{\RR}$ we write \cref{eq:aie-st-delta} as
  \begin{align*}
    u(t) &= T(t - s)P_0u(s) + T(t - s)(I - P_0)u(s) + j^{-1}\int_s^t{\SUNSTAR{T}(t - \tau)R_{\delta}(u(\tau))\,d\tau}\\
         &= T(t - s)P_0u(s) + (\mathcal{K}_{\eta} \tilde{R}_{\delta}(u))(t)
  \end{align*}
  where \cref{prop:Keta:solution} was used for the second equality. Rearranging, we obtain
  \[
    T(-t)(u(t) - (\mathcal{K}_{\eta} \tilde{R}_{\delta}(u))(t)) = T(-s)P_0 u(s) \qquad \ALL\,(t,s) \in \Omega_{\RR}
  \]
  and it follows that both sides equal the same constant $\phi \in X_0$. Hence
  \[
    u(t) - (\mathcal{K}_{\eta} \tilde{R}_{\delta}(u))(t) = T(t)\phi \qquad \ALL\,t \in \RR
  \]
  which shows that $\mathcal{G}(u,\phi) = u$, implying that $\psi = u(0) = \mathcal{C}(\phi) \in \CM$.
\end{proof}

Let $B_{\delta}(X)$ be the open $\delta$-ball centered at the origin in $X$. The restrictions of $R_{\delta}$ and $R$ to this ball are equal, so if we restrict the unknown $u$ to take values in $B_{\delta}(X)$, then \cref{eq:aie-st-delta} and \cref{eq:aie-st} coincide as well. We note that \cref{thm:cm:lipschitz} implies that $U$ in the following definition is an open neighborhood of the origin in $X_0$.

\begin{definition}\label{def:LCM}
  Let $\mathcal{C}$ be as in \cref{eq:cmmap}. A \textbf{\emph{local} center manifold} $\LCM$ for \cref{eq:aie-st} is defined as the image of the restriction of $\mathcal{C}$ to $U \DEF \{\phi \in X_0\,:\,\mathcal{C}(\phi) \in B_{\delta}(X)\}$.
\end{definition}

In \cite{Delay1995} the proof of the next result is suggested as an exercise, but in the present context it also follows directly from \cref{prop:CMcharacterization}. We recall the definition of the semiflow $\Sigma$ in \cref{eq:semiflow}.

\begin{corollary}[\normalfont{cf. \cite[Theorem IX.5.3]{Delay1995}}]
  The following two statements hold.
  \begin{corenum}
  \item
    $\LCM$ is locally positively invariant: If $\psi \in \LCM$ and $0 < t_e \le \infty$ are such that $\Sigma(t,\psi) \in B_{\delta}(X)$ for all $t \in J_{\psi} \cap [0,t_e)$, then $\Sigma(t,\psi) \in \LCM$ for all $t \in J_{\psi} \cap [0,t_e)$.
  \item
    $\LCM$ contains every solution of \cref{eq:aie-st} that exists on $\RR$ and remains sufficiently small for all positive and negative time: If $u : \RR \to B_{\delta}(X)$ is a solution of \cref{eq:aie-st} then $u$ takes its values in $\LCM$.
  \end{corenum}
\end{corollary}
\begin{proof}
  \begin{steps}[label=\roman*.]
  \item
    It is convenient to write $J_{\psi}^e \DEF J_{\psi} \cap [0,t_e)$. \Cref{prop:CMcharacterization} implies that there exists a solution $u \in \BC^{\eta}(\RR,X)$ of \cref{eq:aie-st-delta} passing through $\psi$, and by translation invariance we may assume that $u(0) = \psi$. So, $\Sigma(\cdot,\psi)$ and $u$ are both solutions of \cref{eq:aie-st-delta} on $J_{\psi}^e$ and $\Sigma(0,\psi) = \psi = u(0)$. Hence $\Sigma(\cdot,\psi)$ and $u$ coincide on $J_{\psi}^e$. A second application of \cref{prop:CMcharacterization} then implies that $\Sigma(t,\psi) \in \CM$ for all $t \in J_{\psi}^e$. Since $\LCM = \CM \cap B_{\delta}(X)$ the result follows.
  \item
    If $u$ is such a solution, then $u \in \BC^{\eta}(\RR,X)$. The assumption that $u$ takes its values in $B_{\delta}(X)$ and \cref{prop:CMcharacterization} together imply the result. \qedhere
  \end{steps}
\end{proof}

In order to establish that $\mathcal{C}$ has the same degree $k \ge 1$ of smoothness as the \emph{un}modified nonlinearity $R$, it can be verified that the general results on contractions on scales of Banach spaces \cite[Section IX.6]{Delay1995} apply exactly as in the $\odot$-reflexive case \cite[Section IX.7]{Delay1995}, \emph{provided} that one consistently makes the substitution $X^{\odot\star} \to X^{\odot\times}$. In particular, the important results \cite[Corollaries IX.7.8 and IX.7.10]{Delay1995} on $C^k$-smoothness and tangency, respectively, remain true.

We give a summary of the results discussed in this subsection.

\begin{theorem}\label{thm:cm-summary}
  Let $\hat{\phi} = 0$ be an equilibrium of the semiflow $\Sigma$ associated with the maximal solutions of \cref{eq:aie0} and let $\{T(t)\}_{t \ge 0}$ be the linearization of $\Sigma$ at $\hat{\phi}$. Assume that \cref{hyp:cm} and \cref{hyp:cm:j} hold and that $X_0$ is finite-dimensional. Then there exist a $C^k$-smooth mapping $\mathcal{C} : X_0 \to X$ and an open neighborhood $U$ of the origin in $X_0$ such that $\mathcal{C}(0) = 0$, $D\mathcal{C}(0) = I_{X_0 \to X}$ and $\LCM = \mathcal{C}(U)$ is locally positively invariant for $\Sigma$ and contains every solution of \cref{eq:aie-st} that exists on $\RR$ and remains sufficiently small for all time.
\end{theorem}

\section{The special case of abstract DDEs}\label{sec:ddes}
At this point it is relatively straightforward to show that abstract DDEs \cref{eq:adde} form an example of a class of delay equations satisfying, under certain conditions, the hypotheses in \cref{hyp:G}, \cref{hyp:cm}, and \cref{hyp:cm:j} of the results in the previous sections.

We recall the specific setting. Let $\{T_0(t)\}_{t \ge 0}$ be the shift semigroup on $X \DEF C([-h,0],Y)$ where $Y$ is a \emph{real} Banach space. We assume that $F : X \to Y$ is of class $C^k$ for some $k \ge 1$ and we set $G \DEF \ell \circ F$. It was already shown in \cite[Proposition 8]{ADDE1} that
\begin{equation}
  \label{eq:admissible:dde}
  \int_0^t T_0^{\odot\star}(t - \tau)\ell w(\tau)\,d\tau \in jX \qquad \text{for all continuous } w : \RR_+ \to Y \text{ and all } t \ge 0,
\end{equation}
with $\ell : Y \to X^{\odot\star}$ given by \cref{eq:ell}. In the terminology of the present article, we have:

\begin{proposition}\label{prop:admrange-dde}
  $\ell Y$ is an admissible range for $\{T_0(t)\}_{t \ge 0}$ and $G = \ell \circ F$ is an admissible perturbation.
\end{proposition}
\begin{proof}
  The closedness of $\ell Y$ in $X^{\odot\star}$ is due to \cite[Lemma 6]{ADDE1}. If $f : \RR_+ \to \ell Y$ is any forcing function, then $f = \ell \circ w$ where $w \DEF \ell^{-1} \circ f : \RR_+ \to Y$ is continuous. So, for any $(t,s) \in \Omega_{\RR_+}$,
  \begin{align*}
    v_0(t,s,f) &= \int_s^t T_0^{\odot\star}(t - \tau) \ell w(\tau)\,d\tau\\
               &= \int_0^{t - s} T_0^{\odot\star}(t - s - \tau) \ell w(s + \tau)\,d\tau
  \end{align*}
  The function $w(s + \cdot) : \RR_+ \to Y$ is continuous and $t - s \ge 0$ so by \cref{eq:admissible:dde} the right-hand side is in $jX$. The statements then follows from \cref{prop:admindependent}.
\end{proof}

A `local' modification of \cite[Theorem 16]{ADDE1} provides a one-to-one correspondence between the maximal mild solutions of \cref{eq:adde-ic} and the maximal solutions of \cref{eq:aie0-adde}, where the latter equation is a particular case of \cref{eq:aie0}. \Cref{prop:admrange-dde} and the smoothness of $F$ imply that $G$ satisfies \cref{hyp:G}. We define the semiflow $\Sigma$ on $X$ as in \cref{eq:semiflow}, using the maximal solutions of \cref{eq:aie0-adde}.

Let $\hat{\phi} = 0$ be an equilibrium of $\Sigma$ or, equivalently, let $F(0) = 0$. We split $G$ as in \cref{eq:splittingGLR} to obtain
\begin{equation}
  \label{eq:LR-adde}
  L\phi \DEF \ell DF(0)\phi, \qquad R(\phi) = \ell(F(\phi) - DF(0)\phi), \qquad \phi \in X.
\end{equation}
The $C^k$-smooth operator $R$ is an admissible nonlinear perturbation for the $\mathcal{C}_0$-semigroup $\{T(t)\}_{t \ge 0}$. This semigroup itself is defined by perturbing the shift semigroup $\{T_0(t)\}_{t \ge 0}$ with the admissible linear perturbation $L$. \Cref{prop:linearization} shows that $T(t)$ is recovered as the partial derivative of $\Sigma$ with respect to the state at the point $(t,\hat{\phi})$, for any $t \ge 0$.

As a specialized counterpart to \cref{thm:cm-summary}, we have:

\begin{theorem}\label{thm:cm-dde}
  Let $\hat{\phi} = 0$ be an equilibrium of the semiflow $\Sigma$ associated with the maximal solutions of \cref{eq:aie0-adde} and let $\{T(t)\}_{t \ge 0}$ be the $\mathcal{C}_0$-semigroup with generator $A$, obtained by perturbing the shift-semigroup $T_0$ by $L$, with $L$ and $R$ as in \cref{eq:LR-adde}. If the $\mathcal{C}_0$-semigroup $\{S(t)\}_{t \ge 0}$ generated by $B$ in \cref{eq:adde} is immediately norm continuous, $\sigma(A_{\cc})$ satisfies the conditions of \cref{thm:hypcm} and $X_0$ is finite dimensional, then there exist a $C^k$-smooth mapping $\mathcal{C} : X_0 \to X$ and an open neighborhood $U$ of the origin in $X_0$ such that $\mathcal{C}(0) = 0$, $D\mathcal{C}(0) = I_{X_0 \to X}$ and $\LCM = \mathcal{C}(U)$ is locally positively invariant for $\Sigma$ and contains every solution of \cref{eq:aie-st} that exists on $\RR$ and remains sufficiently small for all time.
\end{theorem}
\begin{proof}
  This is a direct consequence of \cref{thm:cm-summary,thm:hypcm} combined with \cite[Theorem VI.6.6]{Engel2000}.
\end{proof}

In particular, if the $\mathcal{C}_0$-semigroup $\{S(t)\}_{t \ge 0}$ is immediately compact or analytic, then it is immediately norm continuous. Clearly the conditions in the above theorem are sufficient, but not necessary. For instance, it would have been sufficient to directly assume eventual norm continuity of $\{T(t)\}_{t \ge 0}$ itself. However, I have chosen a formulation that I believe to be suitable for application to specific examples. If more generality is required, then one may want to return to \cref{thm:cm-summary}.

\section{Conclusion and outlook}\label{sec:conclusion}
The notions of admissible range and admissible perturbation for a given $\mathcal{C}_0$-semigroup are natural abstractions of the approach to non-$\odot$-reflexivity proposed in \cite{Diekmann2008,DiekmannGyllenberg2012}, and the subspace $X^{\odot\times}$ of $X^{\odot\star}$ introduced in \cite{VanNeerven1992} naturally occurs in this context as the largest among all admissible ranges. Therefore, the admissibility problem is in principle resolved. The clause \emph{in principle} refers to the fact that for concrete classes of delay equations, one still has to verify that the perturbation $G$ appearing in \cref{eq:aie0} indeed takes its values in $X^{\odot\times}$. For abstract DDEs this was done in \cite{ADDE1}.

Interestingly, already in \cite[Section 4.5]{VanNeerven1992} it is remarked that ``From our point of view this assumption [i.e. $\odot$-reflexivity] is used only to achieve $X^{\odot\times} = X^{\odot\star}$. By replacing $X^{\odot\star}$ by $X^{\odot\times}$ many of the results [of Cl\'ement, Diekmann, Gyllenberg, Heijmans and Thieme] generalize to the non-$\odot$-reflexive case.'' Furthermore, in their introduction to \cite[Section VIII.2]{Delay1995} on stable and unstable manifolds, the authors comment that $\odot$-reflexivity is assumed ``mostly for ease of formulation; the same techniques yield analogous results when, for instance, we do not have $\odot$-reflexivity but still solutions of the AIE [abstract integral equation, i.e. \cref{eq:aie0}] define a nonlinear semigroup on $X$.'' In the light of these comments, and of course also with the benefit of hindsight, it would have been better if the authors of \cite{Delay1995} had relaxed the assumption of $\odot$-reflexivity in favor of a systematic use of the subspace $X^{\odot\times}$ instead of the full dual space $X^{\odot\star}$.  

As an example that aims to be illustrative as well as useful in its own right, I have discussed in some detail the construction of local center manifolds in the non-$\odot$-reflexive case. However, I have not attempted a full rewrite of \cite[Chapter IX]{Delay1995}, for three reasons. First, the nontrivial consequences of non-$\odot$-reflexivity for this construction are very much localized, as explained in \cref{sec:Keta}. Second, I believe that the proper medium for such a rewrite would be a monograph or a textbook, rather than a research article. Third, it would arguably benefit the development of the general theory to proceed at a level sufficiently minimal and axiomatic to allow for a semilinear extension of the recent advances available in \cite{Twin2019}. It is my understanding that the authors of \cite{Twin2019} are progressing in this direction. (However, their work currently depends in a non-trivial way on the perturbations involved being of finite rank, and this precludes a direct application of their work to abstract DDEs.)

Motivation for theoretical developments often comes from specific problems, and this work is no exception \cite{VanGils2013,Dijkstra2015,SpekVanGilsKuznetsov2019}. More recently, I have been inspired by a class of control-theoretic examples for which the feedback-controlled system is an abstract DDE of the form \cref{eq:adde}. This will be explored further in collaboration with S. M. Verduyn Lunel. We note that \cite{HeijmansControl1987} contains an early control-theoretic application of dual perturbation theory, in a $\odot$-reflexive setting. Meanwhile, O. Diekmann has suggested to me a class of abstract renewal equations inspired by structured population dynamics. The results from \cref{sec:admissibility,sec:inhomogeneous,sec:linearization,sec:cm} would in principle apply to such equations, depending on the precise model formulation, and we intend to investigate this jointly. Finally, an efficient numerical approach to the linear stability problem for abstract DDEs can be found in \cite{BredaMasetVermiglio2009}, and it would be interesting to see its implementation for some of the examples mentioned above.

\subsection*{Acknowledgements}
I thank Prof. O. Diekmann (Utrecht University) for his lasting interest in this work and for suggesting the population dynamical class of examples mentioned in \cref{sec:conclusion}, Prof. S. M. Verduyn Lunel (Utrecht University) for introducing me to the control-theoretic class of examples mentioned in \cref{sec:conclusion}, and both of them for their stimulating comments and hospitality during various visits to Utrecht.

In addition, I thank Alina Andrei (Erasmus University Rotterdam) for her unconditional support of this work.

\clearpage
\appendix
\numberwithin{equation}{section}

\let\savedsigma\sigma
\let\savedrho\rho
\renewcommand{\sigma}{\mu}
\renewcommand{\rho}{\nu}

\section{Spectral decomposition}
\label{sec:decomposition-appendix}

\subsection{Lifting of the spectral decomposition from \texorpdfstring{$X$}{X} to \texorpdfstring{$X^{\odot\times}$}{Xsuncross}}
\emph{All vector spaces in \cref{sec:cm} are over $\RR$, but in \cref{sec:lifting} the scalar field may be real or complex.}\label{sec:lifting}

\medskip
Let $\{T(t)\}_{t \ge 0}$ be a $\mathcal{C}_0$-semigroup on a Banach space $X$ such that \cref{hyp:cm} in \cref{sec:cm} holds. We show how \cref{hyp:cm} induces a decomposition of $X^{\odot\times}$ with similar properties. For this it is necessary to understand how the decomposition \cref{eq:decomposition-of-x} behaves with respect to the sun-star duality structure. A systematic use of the following simple lemma, already presented as a matter of fact in \cite[p.100]{Delay1995}, is very helpful.

\begin{lemma}\label{lem:iota}
  Let $P$ be a continuous projector on a Banach space $E$ with range $\mathcal{R}(P)$. Then
  \[
    \iota : \mathcal{R}(P^{\star}) \to \mathcal{R}(P)^{\star}, \qquad \iota y^{\star} \DEF y^{\star}|_{\mathcal{R}(P)}
  \]
  is an isometric isomorphism.
\end{lemma}
\begin{proof}
  We show that $\iota$ is an isometry onto $\mathcal{R}(P)^{\star}$.
  \begin{steps}
  \item
    We show that $\iota$ is an isometry. It holds that
    \begin{align*}
      \|\iota y^{\star}\| &= \sup\{|\PAIR{y}{y^{\star}}| \,:\, y \in \mathcal{R}(P) \text{ and } \|y\| \le 1\},\\
      \|y^{\star}\| &= \sup\{|\PAIR{y}{y^{\star}}| \,:\, y \in E \text{ and } \|y\| \le 1\},
    \end{align*}
    so clearly $\|\iota y^{\star}\| \le \|y^{\star}\|$. On the other hand, if $y \in E$ and $\|y\| \le 1$ then $Py \in \mathcal{R}(P)$ and $\|Py\| \le 1$ and
    \[
      |\PAIR{Py}{y^{\star}}| = |\PAIR{y}{P^{\star}y^{\star}}| = |\PAIR{y}{y^{\star}}|
    \]
    and this shows that $\|y^{\star}\| \le \|\iota y^{\star}\|$ as well.
  \item
    We show that $\iota$ is a surjection. Let $w^{\star} \in \mathcal{R}(P)^{\star}$ be arbitrary. Define $y^{\star} \in E^{\star}$ by
    \[
      \PAIR{y}{y^{\star}} \DEF \PAIR{Py}{w^{\star}} \qquad \ALL\,y \in E.
    \]
    Then $y^{\star} \in \mathcal{N}(P)^{\perp}$ and the continuity of $P$ and the closed range theorem then imply that $y^{\star} \in \mathcal{R}(P^{\star})$. Clearly $\iota y^{\star} = w^{\star}$ so $\iota$ is a surjection. \qedhere
  \end{steps}
\end{proof}

It is convenient to introduce the symbol $\sigma \in \{-, 0, +\}$. For any of the three possible values of $\sigma$, let $P_{\sigma}$ be the projector in $\BND(X)$ associated with the closed subspace $X_{\sigma}$ in \cref{eq:decomposition-of-x} and let $\{T_{\sigma}(t)\}_{t \ge 0}$ be the restricted semigroup. The adjoint $P_{\sigma}^{\star}$ is in $\BND(X^{\star})$ with closed range $[X^{\star}]_{\sigma}$. The operator $\iota_{\sigma} : [X^{\star}]_{\sigma} \to [X_{\sigma}]^{\star}$ denotes the isometric isomorphism obtained from \cref{lem:iota} with $E = X$ and $P = P_{\sigma}$. 

\begin{proposition}\label{prop:xstar-cm}
  The space $X^{\star}$ and the semigroup $\{T^{\star}(t)\}_{t \ge 0}$ on $X^{\star}$ have the following properties.
  \begin{propenum}
  \item
    $X^{\star}$ admits a topological direct sum decomposition
    \begin{equation}
      \label{eq:decomposition-of-xstar}
      X^{\star} = [X^{\star}]_- \oplus [X^{\star}]_0 \oplus [X^{\star}]_+.
    \end{equation}
  \item
    The subspaces $[X^{\star}]_{\sigma}$ are positively $T^{\star}$-invariant.
  \item
    \label{prop:xstar-cm:group}
    It holds that $\iota_{\sigma} [T^{\star}]_{\sigma}(t) = [T_{\sigma}]^{\star}(t) \iota_{\sigma}$ for all $t \ge 0$,
  \item[]
    so $\{T^{\star}(t)\}_{t \ge 0}$ extends to a group on $[X^{\star}]_0$ and on $[X^{\star}]_+$.
  \item
    \label{prop:xstar-cm:trichotomy}
    Decomposition \cref{eq:decomposition-of-xstar} is an exponential trichotomy on $\RR$ with the same constants as in \cref{hyp:cm:trichotomy}.
  \end{propenum}
\end{proposition}
\begin{proof}
  \begin{steps}[label=\roman*.]
  \item
    The decomposition of $X$ in \cref{eq:decomposition-of-x} implies that
    \[
     P_{\sigma}^2 = P_{\sigma}, \qquad P_{\sigma}P_{\rho} = 0 \quad \text{if } \sigma \neq \rho, \qquad \textstyle \sum_{\sigma}P_{\sigma} = I.
    \]
    The $P^{\star}_{\sigma}$ clearly satisfy the same properties, which yields the direct sum decomposition \cite[p. 248]{TaylorLay1980}. The continuity of each $P_{\sigma}^{\star}$ implies that the direct sum is topological.
  \item
    It holds that  $T^{\star}(t)P_{\sigma}^{\star} = (P_{\sigma}T(t))^{\star} = (T(t)P_{\sigma})^{\star} = P_{\sigma}^{\star}T^{\star}(t)$ for all $t \ge 0$, where the second equality is due to the positive $T$-invariance of $X_{\sigma}$.
  \item
    It follows easily from the definition of $\iota_{\sigma}$ that, for any $t \ge 0$,
    \[
      \PAIR{x}{\iota_{\sigma}[T^{\star}]_{\sigma}(t)x^{\star}} = \PAIR{x}{[T_{\sigma}]^{\star}(t)\iota_{\sigma}x^{\star}} \qquad \ALL\,x \in X_{\sigma} \text{ and } x^{\star} \in [X^{\star}]_{\sigma}
    \]
    which proves the first statement. For the second statement, let $\sigma \in \{0,+\}$. Then $\{T_{\sigma}(t)\}_{t \in \RR}$ is a group on $X_{\sigma}$, so $\{[T_{\sigma}]^{\star}(t)\}_{t \in \RR}$ is a group on $[X_{\sigma}]^{\star}$. The first statement implies that $\{[T^{\star}]_{\sigma}(t)\}_{t \ge 0}$ extends to a group on $[X^{\star}]_{\sigma}$. This is the second statement.
  \item
    Let $t \ge 0$ (if $\sigma = -$) or $t \in \RR$ (if $\sigma = 0$) or $t \le 0$ (if $\sigma = +$) and let $x^{\star} \in [X^{\star}]_{\sigma}$ be given. Since $\iota_{\sigma}$ is an isometry,
    \[
      \|\underbracket{T^{\star}(t)x^{\star}}_{[X^{\star}]_{\sigma}}\| = \|\underbracket{\iota_{\sigma} T^{\star}(t)x^{\star}}_{[X_{\sigma}]^{\star}}\| = \sup_{\substack{x \in X_{\sigma}\\ \|x\| \le 1}}|\PAIR{x}{T^{\star}(t)x^{\star}}| \le \sup_{\substack{x \in X_{\sigma}\\ \|x\| \le 1}} \|T(t)x\| \cdot \|x^{\star}\|.
    \]
    Now use \cref{hyp:cm:trichotomy} to estimate $\|T(t)x\|$ for $x \in X_{\sigma}$ with $\|x\| \le 1$. This proves the statement. \qedhere
  \end{steps}
\end{proof}

Next, we prepare to formulate an analogous result for $X^{\odot}$. For any $\sigma \in \{-,0,+\}$, the positive $T^{\star}$-invariance of $[X^{\star}]_{\sigma}$ implies that $P_{\sigma}^{\star}$ maps $X^{\odot}$ into itself. The restriction $P_{\sigma}^{\odot} \DEF P_{\sigma}^{\star}|_{X^{\odot}}$ is a projector in $\BND(X^{\odot})$. We denote its range by $[X^{\odot}]_{\sigma}$. It is easily checked that
\begin{equation}
  \label{eq:xsunsigma-intersection}
  [X^{\odot}]_{\sigma} = [X^{\star}]_{\sigma} \cap X^{\odot}.
\end{equation}

{\renewcommand{\ss}{\text{s}}
  \begin{lemma}\label{lem:xsun-semigroup-group}
    Let $\{U(t)\}_{t \in \RR}$ in $\BND(E)$ be a group on a Banach space $E$ and set $E^{\ss} \DEF \{x \in E \,:\, U(\cdot)x \text{ is continuous on } \RR\}$ and $E^{\ss}_+ \DEF \{x \in E \,:\, U(\cdot)x \text{ is continuous on } \RR_+\}$. Then $E^{\ss} = E^{\ss}_+$.
  \end{lemma}
  \begin{proof}
    It is clear that $E^{\ss} \subseteq E^{\ss}_+$. For the other inclusion, we first show that $E^{\ss}_+$ is $U$-invariant. Let $x \in E^{\ss}_+$ and $t \in \RR$ be arbitrary. Then $U(\delta)U(t)x = U(t)U(\delta)x \to U(t)x$ in norm as $\delta \downarrow 0$, so $U(t)x \in E^{\ss}_+$, which proves the invariance. Hence we can consider $\{U(t)\}_{t \in \RR}$ as a group on $E^{\ss}_+$. This group is strongly continuous by \cite[Exercise I.5.9.(5)]{Engel2000}, so $E^{\ss}_+ \subseteq E^{\ss}$ as well.
  \end{proof}
  
  \begin{lemma}\label{lem:xsun-similar}
    Let $h : E \to \tilde{E}$ be a topological isomorphism of Banach spaces $E$ and $\tilde{E}$. Suppose that the semigroups $\{U(t)\}_{t \ge 0}$ in $\BND(E)$ and $\{\tilde{U}(t)\}_{t \ge 0}$ in $\BND(\tilde{E})$ satisfy 
    \[
      h U(t) = \tilde{U}(t)h \qquad \ALL\,t \ge 0.
    \]
    If $E^{\ss}$ and $\tilde{E}^{\ss}$ denote the respective domains of strong continuity, then $h E^{\ss} = \tilde{E}^{\ss}$.
  \end{lemma}
  \begin{proof}
    For arbitrary $y \in E^{\ss}$ we have $\tilde{U}(t) h y = h U(t) y \to h y$ as $t \downarrow 0$ so $h y \in \tilde{E}^{\ss}$. Conversely, given arbitrary $\tilde{y} \in \tilde{E}^{\ss}$ we define $y \DEF h^{-1}\tilde{y}$ and note that $U(t)y = h^{-1}\tilde{U}(t)\tilde{y} \to h^{-1}\tilde{y} = y$ as $t \downarrow 0$ so $y \in E^{\ss}$. Hence $h E^{\ss} = \tilde{E}^{\ss}$.
  \end{proof}
}

\begin{lemma}\label{lem:iota-sun}
  $\iota_{\sigma}$ maps $[X^{\odot}]_{\sigma}$ onto $[X_{\sigma}]^{\odot}$.
\end{lemma}
\begin{proof}
  \Cref{prop:xstar-cm:group} shows that \cref{lem:xsun-similar} applies with
  \[
    E = [X^{\star}]_{\sigma}, \qquad \tilde{E} = [X_{\sigma}]^{\star}, \qquad h = \iota_{\sigma}, \qquad U(t) = [T^{\star}]_{\sigma}(t), \qquad \tilde{U}(t) = [T_{\sigma}]^{\star}(t).
  \]
  This yields that $\iota_{\sigma}$ maps the domain of strong continuity of $\{[T^{\star}]_{\sigma}(t)\}_{t \ge 0}$ onto $[X_{\sigma}]^{\odot}$, and \cref{eq:xsunsigma-intersection} implies that $[X^{\odot}]_{\sigma}$ is the domain of strong continuity of $\{[T^{\star}]_{\sigma}(t)\}_{t \ge 0}$.
\end{proof}

By restricting the results from \cref{prop:xstar-cm}, we obtain an analogue on $X^{\odot}$.

\begin{proposition}\label{prop:xsun-cm}
  The space $X^{\odot}$ and the $\mathcal{C}_0$-semigroup $\{T^{\odot}(t)\}_{t \ge 0}$ on $X^{\odot}$ have the following properties.
  \begin{propenum}
  \item
    $X^{\odot}$ admits a topological direct sum decomposition
    \begin{equation}
      \label{eq:decomposition-of-xsun}
      X^{\odot} = [X^{\odot}]_- \oplus [X^{\odot}]_0 \oplus [X^{\odot}]_+.
    \end{equation}
  \item
    The subspaces $[X^{\odot}]_{\sigma}$ are positively $T^{\odot}$-invariant.
  \item
    \label{prop:xsun-cm:group}
    $\{T^{\odot}(t)\}_{t \ge 0}$ extends to a $\mathcal{C}_0$-group on $[X^{\odot}]_0$ and on $[X^{\odot}]_+$.
  \item
    Decomposition \cref{eq:decomposition-of-xsun} is an exponential trichotomy on $\RR$ with the same constants as in \cref{hyp:cm:trichotomy}.
  \end{propenum}
\end{proposition}
Also, it holds that $\iota_{\sigma}^{\odot} [T^{\odot}]_{\sigma}(t) = [T_{\sigma}]^{\odot}(t) \iota_{\sigma}^{\odot}$ for all $t \in J_{\sigma}$. This is proven by taking the restriction to $[X^{\odot}]_{\sigma}$ in the first part of \cref{prop:xstar-cm:group} and then applying \cref{lem:iota-sun}.
\begin{proof}
  The first two statements follow directly from the corresponding statements in \cref{prop:xstar-cm}.
  \begin{steps}[start=3,label=\roman*.]
  \item
    Let $\sigma \in \{0,+\}$. It follows from \cref{eq:xsunsigma-intersection} that the domain of strong continuity of the semigroup $\{[T^{\star}]_{\sigma}(t)\}_{t \ge 0}$ equals $[X^{\odot}]_{\sigma}$. \Cref{lem:xsun-semigroup-group} implies that the same holds for the group $\{[T^{\star}]_{\sigma}(t)\}_{t \in \RR}$. In particular, $[X^{\odot}]_{\sigma}$ is invariant with respect to this group. For any $t \ge 0$ and any $x^{\odot} \in [X^{\odot}]_{\sigma}$,
    \[
      [T^{\odot}]_{\sigma}(t)x^{\odot} = [T^{\star}]_{\sigma}(t)x^{\odot},
    \]
    while for each $t < 0$ we \emph{define} the bounded linear operator $[T^{\odot}]_{\sigma}(t)$ on $[X^{\odot}]_{\sigma}$ by
    \[
      [T^{\odot}]_{\sigma}(t)x^{\odot} \DEF [T^{\star}]_{\sigma}(t)x^{\odot}.
    \]
    This gives an extension to a strongly continuous group $\{[T^{\odot}]_{\sigma}(t)\}_{t \in \RR}$.
  \item
    Let $t \ge 0$ (if $\sigma = -$) or $t \in \RR$ (if $\sigma = 0$) or $t \le 0$ (if $\sigma = +$) and let $x^{\odot} \in [X^{\odot}]_{\sigma} \subseteq [X^{\star}]_{\sigma}$ be given. Then $[T^{\odot}]_{\sigma}(t)x^{\odot} = [T^{\star}]_{\sigma}(t)x^{\odot}$ where in case $t < 0$ the equality holds because of the previous statement. The right-hand side can be estimated using \cref{prop:xstar-cm:trichotomy}. \qedhere
  \end{steps}
\end{proof}

\Cref{prop:xsun-cm} shows that \cref{hyp:cm} holds true if we replace $X$ by $X^{\odot}$ and $\{T(t)\}_{t \ge 0}$ by $\{T^{\odot}(t)\}_{t \ge 0}$. We can therefore apply \cref{prop:xstar-cm} with these substitutions.

\begin{corollary}\label{cor:xsunstar-cm}
  The space $X^{\odot\star}$ and the semigroup $\{T^{\odot\star}(t)\}_{t \ge 0}$ on $X^{\odot\star}$ have the following properties.
  \begin{corenum}
  \item
    $X^{\odot\star}$ admits a topological direct sum decomposition
    \begin{equation}
      \label{eq:decomposition-of-xsunstar}
      X^{\odot\star} = [X^{\odot\star}]_- \oplus [X^{\odot\star}]_0 \oplus [X^{\odot\star}]_+.
    \end{equation}
  \item
    The subspaces $[X^{\odot\star}]_{\sigma}$ are positively $T^{\odot\star}$-invariant.
  \item
    \label{cor:xsunstar-cm:group}
    $\{T^{\odot\star}(t)\}_{t \ge 0}$ extends to a group on $[X^{\odot\star}]_0$ and on $[X^{\odot\star}]_+$.
  \item
    Decomposition \cref{eq:decomposition-of-xsunstar} is an exponential trichotomy on $\RR$ with the same constants as in \cref{hyp:cm:trichotomy}.
  \end{corenum}  
\end{corollary}

\begin{lemma}\label{lem:Pj}
  The canonical embedding $j : X \to X^{\odot\star}$ has the following properties.
  \begin{lemenum}
  \item
    \label{lem:Pj:1}
    $P_{\sigma}^{\odot\star}j = jP_{\sigma}$.
  \item
    \label{lem:Pj:2}
    $j$ maps $X_{\sigma}$ into $[X^{\odot\star}]_{\sigma}$.
  \item
    \label{lem:Pj:3}
    $T^{\odot\star}(t)j = j T(t)$ on $X_{\sigma}$ for $\sigma \in \{0,+\}$ and all $t \in \RR$.
  \end{lemenum}
\end{lemma}
\begin{proof}
  The first statement is easily checked using the definitions, and the second statement follows directly from the first. We know that the third statement is true for $t \ge 0$,
  \[
    T^{\odot\star}(t)j x = j T(t) x \qquad \ALL\,x \in X_{\sigma} \text{ and } t \ge 0.
  \]
  Let $t \ge 0$ be arbitrary. We use \cref{cor:xsunstar-cm:group} to act on both sides with $T^{\odot\star}(-t)$ to obtain $j x = T^{\odot\star}(-t) j T(t) x$ for all $x \in X_{\sigma}$. By \cref{hyp:cm:group}  we can substitute $x = T(-t)y$ for any $y \in X_{\sigma}$ to obtain
  \[
    j T(-t) y = T^{\odot\star}(-t) j y.
  \]
  Since $t \ge 0$ and $y \in X_{\sigma}$ are arbitrary, this yields the third statement.
\end{proof}

For any $\sigma \in \{-,0,+\}$, the commutativity of $R(\lambda,A^{\odot\star})$ and $P_{\sigma}^{\odot\star}$ and \cref{lem:Pj:1} together imply that $P_{\sigma}^{\odot\star}$ maps $X^{\odot\times}$ into itself. The restriction $P_{\sigma}^{\odot\times} \DEF P_{\sigma}^{\odot\star}|_{X^{\odot\times}}$ is a projector in $\BND(X^{\odot\times})$. We denote its range by $[X^{\odot\times}]_{\sigma}$. By restricting the results from \cref{cor:xsunstar-cm}, we obtain:

\begin{proposition}\label{prop:xsuncross-cm}
  The space $X^{\odot\times}$ and the semigroup $\{T^{\odot\star}(t)\}_{t \ge 0}$ on $X^{\odot\times}$ have the following properties.
  \begin{propenum}
  \item
    $X^{\odot\times}$ admits a topological direct sum decomposition
    \begin{equation}
      \label{eq:decomposition-of-xsuncross}
      X^{\odot\times} = [X^{\odot\times}]_- \oplus [X^{\odot\times}]_0 \oplus [X^{\odot\times}]_+.
    \end{equation}
  \item
    \label{prop:xsuncross-cm:invariance}
    The subspaces $[X^{\odot\times}]_{\sigma}$ are positively $T^{\odot\star}$-invariant.
  \item
    \label{prop:xsuncross-cm:group}
    $\{T^{\odot\star}(t)\}_{t \ge 0}$ extends to a group on $[X^{\odot\times}]_0$ and on $[X^{\odot\times}]_+$.
  \item
    \label{prop:xsuncross-cm:trichotomy}
    Decomposition \cref{eq:decomposition-of-xsuncross} is an exponential trichotomy on $\RR$ with the same constants as in \cref{hyp:cm:trichotomy}.
  \end{propenum}  
\end{proposition}

\subsection{Verification of \texorpdfstring{\cref{hyp:cm} and \cref{hyp:cm:j}}{H-II and H-III}}\label{sec:real-case}
Let $\{T(t)\}_{t \ge 0}$ be a $\mathcal{C}_0$-semigroup on a \emph{real} Banach space $X$. In this case \cref{hyp:cm} and \cref{hyp:cm:j} are typically verified in two steps. First, using spectral theory it is shown that these hypotheses hold for the complexified semigroup $\{T_{\cc}(t)\}_{t \ge 0}$ on the complexified Banach space $X_{\cc}$, where
\[
  T_{\cc}(t) \DEF [T(t)]_{\cc} \qquad \ALL\,t \ge 0.
\]
Next, it is proven that \cref{hyp:cm} and \cref{hyp:cm:j} also hold for the original semigroup $\{T(t)\}_{t \ge 0}$ on the real space $X$. \Cref{thm:hypcm} in \cref{sec:decomposition} is the result of the application of this procedure to a particular but still rather large class of $\mathcal{C}_0$-semigroups. In this appendix we prove a number of propositions that will then be used to give the deferred proof of \cref{thm:hypcm}.

We assume familiarity with the procedure of complexification of a real normed linear space, its dual space, and the linear operators defined on them; I recommend the presentation in \cite[Section III.7]{Delay1995} that is partly based on \cite[Section 1.3]{Ruston1986}. We begin with a trivial lemma.

\begin{lemma}\label{lem:restrict-complexify}
  Let $V$ be a real vector space, $L : V \to V$ a linear operator and $M$ a subspace of $V$ such that $LM \subseteq M$. Then $L_{\cc}M_{\cc} \subseteq M_{\cc}$ and $L_{\cc}|_{M_{\cc}} = [L|_M]_\cc$, so complexification and restriction commute.
\end{lemma}

The following result shows that \cref{hyp:cm} holds for $\{T(t)\}_{t \ge 0}$ on $X$ if \cref{hyp:cm} holds for the complexifications and additionally the associated projectors are symmetric for conjugation; see \cite[Definition III.7.22]{Delay1995}. When the projectors can be obtained via contour integration, this symmetry condition is rather natural.

\begin{proposition}\label{prop:complex-to-real}
  If \cref{hyp:cm} holds for $\{T_{\cc}(t)\}_{t \ge 0}$ on $X_{\cc}$ and $\BAR{P}_{\sigma} = P_{\sigma}$ for $\sigma \in \{-,0,+\}$, then \cref{hyp:cm} holds for $\{T(t)\}_{t \ge 0}$ on $X$.
\end{proposition}
\pagebreak
\begin{proof}
  We verify the four parts of \cref{hyp:cm} for $\{T(t)\}_{t \ge 0}$ on $X$.
  \begin{steps}[label=\roman*.]
  \item
    Let $\sigma,\rho \in \{-,0,+\}$ be arbitrary and write $X_{\cc\sigma}$ for the respective subspace in the decomposition of $X_{\cc}$. The second assumption and \cite[Lemma III.7.23]{Delay1995} imply that $P_{\sigma}$ is the complexification of an operator $P_{\sigma}^X \in \BND(X)$. In particular,
    \begin{equation}
      \label{eq:PsigmaX}
      P_{\sigma}(x + i0) = P_{\sigma}^X x + i0 \qquad \ALL\,x \in X,
    \end{equation}
    and therefore
    \[
      P_{\sigma}P_{\rho}(x + i0) = P_{\sigma}^XP_{\rho}^X x + i0 \qquad \ALL\,x \in X.
    \]
    The previous two equalities easily imply that
    \[
      (P_{\sigma}^X)^2 = P_{\sigma}^X, \qquad P_{\sigma}^X P_{\rho}^X = 0 \quad \text{if } \sigma \neq \rho, \qquad \textstyle \sum_{\sigma}P_{\sigma}^X = I,
    \]
    so by \cite[p. 248]{TaylorLay1980} we have the decomposition in \cref{eq:decomposition-of-x} with $X_{\sigma} \DEF \mathcal{R}(P_{\sigma}^X)$ and each summand is closed due to the continuity of $P_{\sigma}^X$. For use below we also note that
    \begin{equation}
      \label{eq:XCsigma}
      X_{\cc\sigma} = X_{\sigma\cc} \qquad \ALL\,\sigma \in \{-,0,+\}.
    \end{equation}
  \item
    For any $\sigma \in \{-,0,+\}$ the semigroup $\{T_{\cc}(t)\}_{t \ge 0}$ commutes with $P_{\sigma}$. By \cref{eq:PsigmaX} this implies that
    \[
      T(t)P_{\sigma}^X x + i0 = P_{\sigma}^XT(t)x + i0 \qquad \ALL\,x \in X,
    \]
    so $\{T(t)\}_{t \ge 0}$ commutes with $P_{\sigma}^X$ or, equivalently, $X_{\sigma}$ is positively $T$-invariant.
  \item
    For $\sigma \in \{0,+\}$ let $\{T_{\cc\sigma}(t)\}_{t \in \RR}$ be the extension of $\{T_{\cc}(t)\}_{t \ge 0}$ to a $\mathcal{C}_0$-group on $X_{\cc\sigma} = X_{\sigma\cc}$, where the latter equality was noted in \cref{eq:XCsigma}. For all $t \ge 0$ \cref{lem:restrict-complexify} with $V = X$, $L = T(t)$ and $M = X_{\sigma}$ gives $T_{\cc\sigma}(t) = T_{\sigma\cc}(t)$ and consequently
    \begin{align*}
      \BAR{T_{\cc\sigma}(-t)} &= \BAR{T_{\cc\sigma}(t)^{-1}}\\
                              &= \BAR{T_{\cc\sigma}(t)}^{\,-1}\\
                              &= T_{\cc\sigma}(t)^{-1} = T_{\cc\sigma}(-t),
    \end{align*}
    so $T_{\cc\sigma}(-t)$ on $X_{\sigma\cc}$, too, is the complexification of an operator, denoted by $T_{\sigma}(-t)$, in $\BND(X_{\sigma})$. We conclude that
    \[
      T_{\cc\sigma}(t) = T_{\sigma\cc}(t) \qquad \ALL\,t \in \RR.
    \]
    From this equality and the group property of $\{T_{\cc\sigma}(t)\}_{t \in \RR}$ it follows immediately that
    \[
      T_{\sigma\cc}(t + s) = T_{\sigma\cc}(t)T_{\sigma\cc}(s) \qquad \ALL\,t,s \in \RR,
    \]
    and by acting on arbitrary $x + i0 \in X_{\sigma\cc}$ we obtain the group property for $\{T_{\sigma}(t)\}_{t \in \RR}$. Its strong continuity follows from \cite[Exercise I.5.9.(5)]{Engel2000} and the strong continuity of the semigroup $\{T_{\sigma}(t)\}_{t \ge 0}$.
  \item
    In the following we use that $X_{\cc}$ is endowed with a norm that is admissible with respect to the norm on $X$; see \cite[Definition III.7.5]{Delay1995}. Let $t \ge 0$ (if $\sigma = -$) or $t \in \RR$ (if $\sigma = 0$) or $t \le 0$ (if $\sigma = +$) and let $x \in X_{\sigma}$ be given. Then
    \[
      \|T_{\sigma}(t)x\| = \|T_{\sigma}(t)x + i0\| = \|T_{\sigma\cc}(t)(x + i0)\| = \|T_{\cc\sigma}(t)(x + i0)\|.
    \]
    By \cref{eq:XCsigma} we have $x + i0 \in X_{\sigma\cc} = X_{\cc\sigma}$, so we can estimate the right-hand side by the appropriate exponential factor from \cref{eq:trichotomy} times $\|x + i0\| = \|x\|$. \qedhere
  \end{steps}
\end{proof}

\renewcommand{\rho}{\savedrho}
\renewcommand{\sigma}{\savedsigma}

\renewcommand{\ll}{\text{l}}
\newcommand{\rr}{\text{r}}

\begin{proposition}\label{prop:PbarP}
  Let $L \in \BND(X)$ and suppose there exists $\alpha \in \RR$ such that $\sigma(L_{\cc})$ is the disjoint union of the nonempty closed subsets $\sigma_{\ll} \DEF \{\lambda \in \sigma(L_{\cc})\,:\,\RE{\lambda} < \alpha\}$ and $\sigma_{\rr} \DEF \{\lambda \in \sigma(L_{\cc})\,:\,\RE{\lambda} > \alpha\}$. If $P$ denotes the spectral projector corresponding to either subset, then $\BAR{P} = P$.
\end{proposition}
\begin{proof}
  The spectral projector $P$ corresponding to $\sigma_{\rr}$ is given by the contour integral
  \[
    P = \frac{1}{2\pi i}\int_{\Gamma}R(\lambda, L_{\cc})\,d\lambda,
  \]
  where $\Gamma \subseteq \rho(L_{\cc})$ is any simple closed rectifiable curve that encloses $\sigma_{\rr}$ and the integral itself is of the Riemann-Stieltjes type. We recall that $\BAR{\sigma}(L_{\cc}) = \sigma(L_{\cc})$ and $\BAR{R(\lambda,L_{\cc})} = R(\BAR{\lambda},L_{\cc})$ for all $\lambda \in \rho(L_{\cc})$.
  \begin{steps}
  \item
    There exists $M > 0$ such that $\alpha < \RE{\lambda} < M$ and $-M < \IM{\lambda} < M$ for all $\lambda$ in the compact set $\sigma_{\rr}$. Let $\Gamma$ be the boundary of the rectangle with vertices $\alpha \pm i M$ and $M \pm i M$. Then $\Gamma$ is a simple closed rectifiable curve in $\rho(L_{\cc})$ that encloses $\sigma_{\rr}$ and is symmetric with respect to the real axis. Let $\gamma : [0,1] \to \CC$ be the natural parametrization of $\Gamma$, say in the counterclockwise direction, with $\gamma(0) = \alpha = \gamma(1)$.
  \item
    In order to be explicit, let us write $C : X_{\cc} \to X_{\cc}$ for the conjugation operator, i.e. $C(x + iy) \DEF \BAR{x + iy}$. It follows from the admissibility of the norm on $X_{\cc}$ that $C$ is continuous. For any $x + iy \in X_{\cc}$ we have
    \begin{align*}
      PC(x + iy) &= \frac{1}{2\pi i} \int_{\Gamma}R(\lambda,L_{\cc})C(x + iy)\,d\lambda\\
                 &=  \frac{1}{2\pi i} \int_{\Gamma}C \BAR{R(\lambda,L_{\cc})}(x + iy)\,d\lambda\\
                 &= \frac{1}{2\pi i} \int_{\Gamma}C R(\BAR{\lambda},L_{\cc})(x + iy)\,d\lambda\\
      &= \frac{1}{2\pi i} \int_0^1 C R(\BAR{\gamma}(t),L_{\cc})(x + iy)\,d\gamma(t)
    \end{align*}
    $C$ is additive but not homogeneous, since $C \lambda(x + iy) = \BAR{\lambda} C(x + iy)$. By inspection of the Riemann-Stieltjes sums in \cite[Appendix A.2]{ADDE1} and the continuity of $C$ we see that $C$ can be brought in front of the integral at the cost of conjugating $\gamma$, so
    \begin{align*}
      PC(x + iy) &= C\frac{1}{2\pi(-i)} \int_0^1 R(\BAR{\gamma}(t),L_{\cc})(x + iy)\,d\BAR{\gamma}(t)\\
                 &= C\frac{-1}{2\pi (-i)}\int_0^1 R(\gamma(t),L_{\cc})(x + iy)\,d\gamma(t)\\
                 &= C \frac{1}{2\pi i} \int_{\Gamma}R(\lambda,L_{\cc})(x + iy)\,d\lambda = C P(x + iy).
    \end{align*}
    Since $x + iy$ is arbitrary, this proves that $\BAR{P} = CPC = P$. \qedhere
  \end{steps}
\end{proof}

We recall from the complexified duality diagram in \cite[Section III.7]{Delay1995} that $[X_{\cc}]^{\odot\star}$ and $[X^{\odot\star}]_{\cc}$ may be identified, so the latter space can be used to represent the former. Under this identification, it is not difficult to check that the canonical embedding $j : X_{\cc} \to [X^{\odot\star}]_{\cc}$ is the complexification of the canonical embedding $j^X : X \to X^{\odot\star}$, i.e.
\begin{equation}
  \label{eq:jjX}
  j(x + iy) = j^X x + i j^X y \qquad \ALL\, x + iy \in X_{\cc}.
\end{equation}
Regarding the following result, we recall from the remark following \cref{eq:x0suncross} that $X^{\odot\times}$ is defined using the resolvent of $A^{\odot\star}$ and that both operators are \emph{real} objects.

\begin{proposition}\label{prop:xcross-complexification}
Under the identification of $[X_{\cc}]^{\odot\star}$ and $[X^{\odot\star}]_{\cc}$ it holds that $[X_{\cc}]^{\odot\times} = [X^{\odot\times}]_{\cc}$.
\end{proposition}
\begin{proof}
  Complexification and inversion of an invertible linear operator commute, which implies that 
  \[
    R(\lambda, [A^{\odot\star}]_{\cc}) = R(\lambda,A^{\odot\star})_{\cc} \qquad \ALL\,\lambda \in \rho([A^{\odot\star}]_{\cc}) \cap \rho(A^{\odot\star}).
  \]
  Fix $\lambda$ in the above (non-empty) intersection of resolvent sets. For any $x^{\odot\star} + iy^{\odot\star} \in [X^{\odot\times}]_{\cc}$ we have
  \begin{align*}
    R(\lambda,[A^{\odot\star}]_{\cc})(x^{\odot\star} + i y^{\odot\star}) &= R(\lambda,A^{\odot\star})x^{\odot\star} + iR(\lambda,A^{\odot\star})y^{\odot\star}\\
    &= j^X x + i j^X y = j(x + iy)
  \end{align*}
  for certain $x, y \in X$, where \cref{eq:jjX} was used as well. We conclude that $x^{\odot\star} + iy^{\odot\star} \in [X_{\cc}]^{\odot\times}$. Conversely, assume that this conclusion holds. Then, similarly,
  \begin{align*}
    R(\lambda,A^{\odot\star})x^{\odot\star} + iR(\lambda,A^{\odot\star})y^{\odot\star} &= R(\lambda,[A^{\odot\star}]_{\cc})(x^{\odot\star} + iy^{\odot\star})\\
                                                                                       &= j(x + iy) = j^X x + j^X y,
  \end{align*}
  for certain $x + iy \in X_{\cc}$. Hence $x^{\odot\star}$ and $y^{\odot\star}$ are in $X^{\odot\times}$ or, in other words, $x^{\odot\star} + i y^{\odot\star} \in [X^{\odot\times}]_{\cc}$.
\end{proof}

We can now give a proof of \cref{thm:hypcm} from \cref{sec:decomposition}. 

\newcommand{\cu}{\text{cu}}
\hypcmthm*
\begin{proof}
  The starting point is contained in the proof of \cite[Theorem A-III.3.3]{Positive1986}. It is convenient to introduce the shorthand notation
  \begin{equation}
    \label{eq:shorthands}
    E \DEF X_{\cc}, \quad U(t) \DEF T_{\cc}(t), \quad D \DEF A_{\cc}, \qquad \ALL\,t \ge 0.
  \end{equation}
  Let $\omega_0(U)$ be the growth bound of $\{U(t)\}_{t \ge 0}$. The spectral mapping theorem for the resolvent implies that, for any real $\lambda_0 > \omega_0(U)$, the spectrum of the \emph{bounded} operator $R_0(D) \DEF R(\lambda_0,D) \in \BND(E)$ is the disjoint union of the closed set $\tau_- \cup \{0\}$ and the compact set $\tau_{\cu}$, where
  \[
    \tau_- \DEF \{(\lambda_0 - \lambda)^{-1}\,:\,\lambda \in \sigma_-\}, \qquad \tau_{\cu} \DEF \{(\lambda_0 - \lambda)^{-1}\,:\,\lambda \in \sigma_{\cu}\}.
  \]
  Let $P_{\cu} \in \BND(E)$ be the spectral projector corresponding to $\tau_{\cu}$ in the above decomposition of $\sigma(R_0(D))$. Explicitly,
  \begin{equation}
    \label{eq:Pcu-integral}
    P_{\cu} = \frac{1}{2\pi i}\int_{\Gamma_{\cu}} R(\lambda,R_0(D))\,d\lambda,
  \end{equation}
  where $\Gamma_{\cu} \subseteq \rho(R_0(D))$ is any simple closed rectifiable curve that encloses $\tau_{\cu}$. In the steps below, we will first verify the hypotheses of \cref{prop:complex-to-real}, which will let us conclude that \cref{hyp:cm} holds for $\{T(t)\}_{t \ge 0}$ on $X$. After this we will prove that \cref{hyp:cm:j} is satisfied.    
  \begin{steps}
  \item
    The proof of \cite[Theorem A-III.3.3]{Positive1986} shows that $P_{\cu}$ is a spectral projector corresponding to the decomposition $\sigma(D) = \sigma_- \cup \sigma_{\cu}$ where $\sigma_{\cu} \DEF \sigma_0 \cup \sigma_+$. So, if $E_{\cu} \DEF \mathcal{R}(P_{\cu})$ and $E_- \DEF \mathcal{R}(I - P_{\cu})$ then $E = E_- \oplus E_{\cu}$, the subspaces $E_-$ and $E_{\cu}$ are non-trivial, closed, and positively $U$-invariant and
    \[
      \sigma(D_-) = \sigma_-, \qquad \sigma(D_{\cu}) = \sigma_0 \cup \sigma_+.
    \]
    Here $D_-$ and $D_{\cu}$ are the generators of the restricted semigroups $\{U_-(t)\}_{t \ge 0}$ and $\{U_{\cu}(t)\}_{t \ge 0}$ and moreover $D_{\cu}$ is bounded.
  \item
    Rather than applying \cite[Theorem A-III.3.3]{Positive1986} for a second time, we use the boundedness of $D_{\cu}$ to obtain that the contour integrals
    \begin{equation}
      \label{eq:Qmu-integral}
      Q_{\mu} \DEF \frac{1}{2\pi i}\int_{\Gamma_{\mu}}R(\lambda,D_{\cu})\,d\lambda, \qquad \mu \in \{0,+\}, 
    \end{equation}
    are spectral projectors that give the unique decomposition $E_{\cu} = E_0 \oplus E_+$ with the property that the subspaces $E_{\mu} \DEF \mathcal{R}(Q_{\mu})$ are non-trivial, closed, and positively $D_{\cu}$-invariant and
    \[
      \sigma(D_0) = \sigma_0, \qquad \sigma(D_+) = \sigma_+.
    \]
    Here $D_0$ and $D_+$ are the respective restrictions of $D$ to $E_0$ and $E_+$. These subspaces are positively invariant for $\{U_{\cu}(t)\}_{t \ge 0}$ as well, and the restricted semigroups $\{U_0(t)\}_{t \ge 0}$ and $\{U_+(t)\}_{t \ge 0}$ are generated by $D_0$ and $D_+$. Both generators are bounded, so the semigroups extend to uniformly continuous groups.
  \end{steps}
  
  In summary, we have obtained a decomposition $E = E_- \oplus E_0 \oplus E_+$ into non-trivial closed positively $U$-invariant subspaces of $E$ such that on both $E_0$ and $E_+$ the restricted semigroups $\{U_0(t)\}_{t \ge 0}$ and $\{U_+(t)\}_{t \ge 0}$ extend to uniformly continuous groups. This is more than what is needed to conclude that \cref{hyp:cm:sum}, \cref{hyp:cm:invariance}, and \cref{hyp:cm:group} are satisfied by $\{U(t)\}_{t \ge 0}$ on $E$. We continue by checking \cref{hyp:cm:trichotomy}.
  
  \begin{steps}[resume]
  \item
    We fix $\delta \in (0,\min\{\gamma_+, -\gamma_-\})$ and set $a \DEF \gamma_- + \delta < 0$ and $b \DEF \gamma_+ - \delta > 0$. Let $\EPS > 0$ be given. An application of \cite[Corollary A-III.3.4]{Positive1986} to $\{U_{\cu}(t)\}_{t \ge 0}$ on $E_{\cu}$ yields a constant $m_+ \ge 1$ such that
    \[
      m_+ e^{b(-t)}\|U_+(t)x\| \le \|U_+(-t)U_+(t)x\| \qquad \ALL\, t \le 0 \text{ and } x \in E_+,
    \]
    which can be rewritten as
    \begin{equation}
      \label{eq:U+estimate}
      \|U_+(t)x\| \le \frac{1}{m_+}e^{bt}\|x\| \qquad \ALL\, t \le 0 \text{ and } x \in E_+.
    \end{equation}
    The same corollary also yields constants $m_{\EPS} \ge 1$ and $M_{\EPS} \ge 1$ such that
    \[
      \|U_0(t)x\| \le M_{\EPS} e^{\EPS t}\|x\| \qquad \ALL\,t \ge 0 \text{ and } x \in E_0
    \]
    and
    \[
      m_{\EPS} e^{-\EPS(-t)}\|U_0(t)x\| \le \|U_0(-t)U_0(t)x\| \qquad \ALL\,t \le 0 \text{ and } x \in E_0.
    \]
    We rewrite the second inequality and combine it with the first inequality to obtain
    \begin{equation}
      \label{eq:U0estimate}
      \|U_0(t)x\| \le \max\Bigl\{M_{\EPS},\frac{1}{m_{\EPS}}\Bigr\}e^{\EPS|t|}\|x\| \qquad \ALL\,t \in \RR \text{ and } x \in E_0.
    \end{equation}
    On the complementary space $E_-$ we apply the spectral mapping theorem  \cite[Theorem A-III.6.6]{Positive1986} for eventually norm continuous $\mathcal{C}_0$-semigroups to $\{U_-(t)\}_{t \ge 0}$. Its growth bound satisfies $\omega_0(U_-) = \gamma_- < a$ strictly, so there exists a constant $M_- \ge 1$ such that
    \begin{equation}
      \label{eq:U-estimate}
      \|U_-(t)x\| \le M_- e^{at}\|x\| \qquad \ALL\,t \ge 0 \text{ and } x \in E_-.
    \end{equation}
    So, for $K_{\EPS} \DEF \max\Bigl\{M_-, \frac{1}{m_+}, M_{\EPS}, \frac{1}{m_{\EPS}}\Bigr\}$ we see from \cref{eq:U+estimate,eq:U0estimate,eq:U-estimate} that $\{U(t)\}_{t \ge 0}$ on $E$ satisfies \cref{hyp:cm:trichotomy}.
  \end{steps}
  
  The first hypothesis of \cref{prop:complex-to-real} has been verified. Concerning the second hypothesis, for simplicity we only consider the (more difficult) case that $D$ is unbounded.
  
  \begin{steps}[resume]
  \item
    Let $r > 0$ be the radius of a ball containing $\sigma_{\cu}$ and fix $\lambda_0 > \max\{\omega_0(U),2r,-\frac{r^2}{\gamma_-}\}$. It is not difficult to check that this choice for $\lambda_0$ implies that
    \[
      \RE{\frac{1}{\lambda_0 - \lambda}} \le \frac{1}{\lambda_0 - \gamma_-} < \frac{\lambda_0}{\lambda_0^2 + r^2} \le \RE{\frac{1}{\lambda_0 - \kappa}} \qquad \ALL \lambda \in \sigma_- \text{ and all } \kappa \in \sigma_{\cu}.
    \]
    Also, $\lambda_0$ is real and $\BAR{D} = D$, which implies that $\BAR{R_0(D)} = R_0(D)$, so $R_0(D)$ is the complexification of an operator $R_0^X(D) \in \BND(X)$. \Cref{prop:PbarP} with $L = R_0^X(D)$, $\sigma_{\ll} = \tau_- \cup \{0\}$ and $\sigma_{\rr} = \tau_{\cu}$ then shows that $\BAR{P}_{\cu} = P_{\cu}$ and consequently $\BAR{P}_- = \BAR{I - P_{\cu}} = P_-$. 
  \item
    It follows that $P_{\cu}$ is the complexification of a projector $P_{\cu}^X \in \BND(X)$ and we note that if $X_{\cu} \DEF \mathcal{R}(P_{\cu}^X)$, then $E_{\cu} = [X_{\cu}]_{\cc}$. The equalities $\BAR{D} = D$ and $\BAR{P}_{\cu} = P_{\cu}$ together imply that $\BAR{D}_{\cu} = D_{\cu}$, so $D_{\cu}$ is the complexification of an operator $D_{\cu}^X \in \BND(X_{\cu})$. We can therefore apply \cref{prop:PbarP} with $L = D_{\cu}^X$, $\sigma_{\ll} = \sigma_0$ and $\sigma_{\rr} = \sigma_+$ to obtain that $\BAR{Q}_0 = Q_0$ and $\BAR{Q}_+ = Q_+$. It remains to note that
    \begin{equation}
      \label{eq:PmubarPmu}
      \BAR{P}_{\mu} = \BAR{Q}_{\mu}\BAR{P}_{\cu} = Q_{\mu}P_{\cu} = P_{\mu}, \qquad \mu \in \{0,+\}.
    \end{equation}
  \end{steps}
  The second hypothesis of \cref{prop:complex-to-real} has been verified, and we infer that \cref{hyp:cm} holds for $\{T(t)\}_{t \ge 0}$ on $X$. Next we prove that \cref{hyp:cm:j} holds, at first for $\{U(t)\}_{t \ge 0}$ on $E$.
  \begin{steps}[resume]
  \item
    The subspace $E_{\cu}$ of $E$ is positively $R(\lambda_0,D)$-invariant, by construction. The compactness of $\sigma_0$ and $\sigma_+$ implies the existence of $\eta > 0$ such that $-\eta \le \IM{\lambda} \le \eta$ for all $\lambda \in \sigma_0 \cup \sigma_+$. Define $\Delta_{\cu} \subseteq \rho(D)$ as the \emph{complement} in the open half-plane $\{\lambda \in \CC\,:\,\RE{\lambda} > \gamma_-\}$ of the union of the line segment and the rectangle
    \[
      \{\lambda \in \CC\,:\,\RE{\lambda} = 0 \text{ and } \IM{\lambda} \in [-\eta,\eta]\}, \qquad \{\lambda \in \CC\,:\,\RE{\lambda} \in [\gamma_+,\omega_0(U)] \text{ and } \IM{\lambda} \in [-\eta,\eta]\}.
    \]
    Then $\Delta_{\cu}$ is open and path-connected. We will show that
    \begin{equation}
      \label{eq:REcu}
      R(\lambda,D)E_{\cu} \subseteq E_{\cu} \qquad \ALL\, \lambda \in \Delta_{\cu}.
    \end{equation}
    For every $\lambda_1 \in \Delta_{\cu}$ there exists a continuous function $p : [0,1] \to \Delta_{\cu}$ with $p(0) = \lambda_0$ and $p(1) = \lambda_1$. The function $R(p(\cdot),D) : [0,1] \to \BND(E)$ is uniformly continuous on its compact domain, so there exists $\EPS > 0$ such that $\max_{0 \le t \le 1} \|R(p(t),D)\| < \frac{1}{\EPS}$ and there exists $\delta > 0$ such that
    \[
      |t - s| \le \delta \quad \text{implies} \quad |p(t) - p(s)| \le \EPS \qquad \ALL\,t,s \in [0,1].
    \]
    Then
    \begin{equation}
      \label{eq:Rexpansion}
      R(p(t),D) = \sum_{n=0}^{\infty} (p(s) - p(t))^n R(p(s),D)^{n+1}
    \end{equation}
    for all $t,s \in [0,1]$ with $|t - s| \le \delta$. So, choose $N \in \NN$ such that $\frac{1}{N} \le \delta$ and let $t_k \DEF \frac{k}{N}$ for $k = 0,\ldots,N$. By induction on $k$ it follows from \cref{eq:Rexpansion} and the closedness of $E_{\cu}$ that $R(p(t_k),D)E_{\cu} \subseteq E_{\cu}$ for all $k = 0,\ldots,N$. In particular, $R(\lambda_1,D)E_{\cu} \subseteq E_{\cu}$ and since $\lambda_1 \in \Delta_{\cu}$ is arbitrary, this proves \cref{eq:REcu}.
  \item
    For $\mu \in \{0,+\}$ let $\Gamma_{\mu} \subseteq \Delta_{\cu}$ be a simple closed rectifiable curve that encloses $\sigma_{\mu}$. It follows from \cref{eq:REcu} and an application of \cite[Lemma IV.1.15]{Engel2000} that
    \[
      \lambda \in \rho(D_{\cu}) \quad \text{and} \quad R(\lambda,D_{\cu}) = R(\lambda,D)|_{E_{\cu}} \qquad \ALL\, \lambda \in \Gamma_{\mu}.
    \]
    Together with \cref{eq:Qmu-integral} this implies
    \[
      P_{\mu} = Q_{\mu} P_{\cu} = \frac{1}{2\pi i}\int_{\Gamma_{\mu}} R(\lambda,D) P_{\cu}\,d\lambda, \qquad \mu \in \{0,+\}.
    \]
    The positive invariance of $E^{\odot}$ for each $P_{\mu}^{\star}$ shows that the same holds true for $P_{\cu}^{\star} = P_0^{\star} + P_+^{\star}$ and we denote $P_{\cu}^{\odot} \DEF P_{\cu}^{\star}|_{E^{\odot}}$. Hence the restriction to $E^{\odot}$ of $(R(\lambda,D)P_{\cu})^{\star} = P_{\cu}^{\star}R(\lambda,D^{\star})$ is given by $P_{\cu}^{\odot}R(\lambda,D^{\odot})$. Taking adjoints once more, we obtain
    \begin{equation}
      \label{eq:Pmu-sunstar-contour}
      P_{\mu}^{\odot\star} = \frac{1}{2\pi i}\int_{\Gamma_{\mu}} R(\lambda,D^{\odot\star}) P_{\cu}^{\odot\star}\,d\lambda, \qquad \mu \in \{0,+\}.
    \end{equation}
    By the comments preceding \cref{prop:xsuncross-cm} each $P_{\mu}^{\odot\star}$ maps $E^{\odot\times}$ into itself, so the same holds true for $P_{\cu}^{\odot\star} = P_0^{\odot\star} + P_+^{\odot\star}$. It follows that the integrand in \cref{eq:Pmu-sunstar-contour} maps $E^{\odot\times}$ into $jE$. The closedness of $jE$ then implies that
    \begin{equation}
      \label{eq:Esuncrossmu}
      [E^{\odot\times}]_{\mu} = P_{\mu}^{\odot\star}E^{\odot\times} \subseteq jE, \qquad \mu \in \{0,+\}.
    \end{equation}
  \end{steps}

  In conclusion we show that \cref{hyp:cm:j} also holds for the original semigroup $\{T(t)\}_{t \ge 0}$ on $X$. Here we need to recall the shorthands introduced in \cref{eq:shorthands}.
  
  \begin{steps}[resume]
  \item
    From \cref{eq:PmubarPmu} it follows that for $\mu \in \{0,+\}$ the projector $P_{\mu} \in \BND(X_{\cc})$ is the complexification of a projector $P_{\mu}^X \in \BND(X)$. We prove that $P_{\mu}^{\odot\star}$ is the complexification of $P_{\mu}^{X\odot\star}$. An application of \cite[Lemma III.7.16]{Delay1995} to $P_{\mu} = [P_{\mu}^X]_{\cc}$ yields $P_{\mu}^{\star} = [P_{\mu}^{X\star}]_{\cc}$. Together with the identification $[X_{\cc}]^{\odot} \simeq [X^{\odot}]_{\cc}$ this implies $P_{\mu}^{\odot} = [P_{\mu}^{X\odot}]_{\cc}$. A second application of the same lemma then indeed gives
    \begin{equation}
      \label{eq:PmuPmuX}  
      P_{\mu}^{\odot\star} = [P_{\mu}^{X\odot\star}]_{\cc}, \qquad \mu \in \{0,+\}.
    \end{equation}
  \item
    For $\mu \in \{0,+\}$ it follows from 
    \[
      P_{\mu}^{X\odot\star}X^{\odot\times} + i0 = P_{\mu}^{\odot\star}(X^{\odot\times} + i0) \subseteq P_{\mu}^{\odot\star}E^{\odot\times} \subseteq jE,
    \]
    where the equality is due to \cref{eq:PmuPmuX}, the first inclusion follows from \cref{prop:xcross-complexification} and the second inclusion is due to \cref{eq:Esuncrossmu}. From \cref{eq:jjX} we infer that $jE = j^XX + i j^XX$, so we conclude that
    \[
      [X^{\odot\times}]_{\mu} = P_{\mu}^{X\odot\star}X^{\odot\times} \subseteq j^XX, \qquad \mu \in \{0,+\},
    \]
    which proves that \cref{hyp:cm:j} holds for $\{T(t)\}_{t \ge 0}$ on $X$. \qedhere
  \end{steps}
\end{proof}

It is interesting to note that the eventual norm continuity of $\{T(t)\}_{t \ge 0}$ is \emph{only} used to obtain the equality of the spectral and growth bounds of $\{U_-(t)\}_{t \ge 0}$, and from there the estimate \cref{eq:U-estimate} on the complementary space $E_-$. So, if $\{T(t)\}_{t \ge 0}$ fails to be eventually norm continuous, but that estimate can still be established via some other route (perhaps by using the structure of the particular equation at hand), then \cref{thm:hypcm} can still be used to verify \cref{hyp:cm} and \cref{hyp:cm:j}.


\clearpage
\appendix
\numberwithin{equation}{section}
\numberwithin{theorem}{section}

\pdfbookmark[0]{References}{myreferences}
\printbibliography

\end{document}